\documentclass[12pt]{amsart}
\usepackage{amsmath}
\usepackage{amscd}
\usepackage{amssymb}
\usepackage{amsfonts}
\usepackage{amsthm}
\usepackage{bbm}
\usepackage{cancel}
\usepackage{color}
\usepackage{eucal}
\usepackage{enumerate,yfonts}
\usepackage{enumitem}
\usepackage{tikz}

\usepackage{pdfsync}
\usepackage[all,cmtip]{xy}
\usepackage{graphicx}
\usepackage{graphics}
\usepackage{hyperref}
\usepackage{latexsym}
\usepackage{mathrsfs}
\usepackage{placeins}
\usepackage{pstricks}
\usepackage{bookmark}
\usepackage{mathtools}
\usepackage{stmaryrd}
\usepackage{url}

\newtheorem{thm}{Theorem}[section]
\newtheorem{corollary}[thm]{Corollary}
\newtheorem{lemma}[thm]{Lemma}
\newtheorem{prop}[thm]{Proposition}
\newtheorem{thm-def}[thm]{Theorem-Definition}
\newtheorem{prop-def}[thm]{Proposition-Definition}
\newtheorem{conj}[thm]{Conjecture}
\newtheorem{remark}[thm]{Remark}
\newtheorem{definition}[thm]{Definition}

\newtheorem{example}[thm]{Example}
\newtheorem{convention}[thm]{Convention}

\numberwithin{equation}{section}

\newcommand{\nc}{\newcommand}

\newcommand{\cD}{\mathcal{D}}
\newcommand{\cO}{\mathcal{O}}
\newcommand{\cA}{\mathcal{A}}

\newcommand{\cH}{\mathcal{H}}

\newcommand{\cP}{\mathcal{P}}
\newcommand{\cR}{\mathcal{R}}
\newcommand{\cL}{\mathcal{L}}

\newcommand{\cZ}{\mathcal{Z}}
\newcommand{\cS}{\mathcal{S}}
\newcommand{\cM}{\mathcal{M}}
\newcommand{\cG}{\mathcal{G}}
\newcommand{\cX}{\mathcal{X}}
\newcommand{\cK}{\mathcal{K}}

\newcommand{\bC}{{\mathbb C}}
\newcommand{\bZ}{{\mathbb Z}}

\newcommand{\bR}{{\mathbb R}}

\newcommand{\bG}{{\mathbb G}}

\newcommand{\bi}{{\mathbf i}}

\newcommand{\La}{{\mathfrak{a}}}

\newcommand{\Lg}{{\mathfrak g}}

\newcommand{\Lp}{{\mathfrak{p}}}
\newcommand{\Lm}{{\mathfrak{m}}}
\newcommand{\Lk}{{\mathfrak{K}}}

\newcommand{\fF}{{\mathfrak{F}}}

\newcommand{\rD}{\mathrm{D}}

\newcommand{\inv}{{\mathbin{/\mkern-4mu/}}}

\nc{\lp}{{\left(}}
\nc{\rp}{{\right)}}

\nc{\ot}{\otimes}
\nc{\on}{\operatorname}

\nc{\oh}{{\operatorname{H}}}
\nc{\gr}{{\operatorname{gr}}}
\nc{\rk}{{\operatorname{rank}}}
\nc{\codim}{{\operatorname{codim}}}
\nc{\img}{{\operatorname{Im}}}
\nc{\IC}{{\operatorname{IC}}}

\newcommand{\p}{\perp}

\newcommand{\beqn}{\begin{equation*}}
\newcommand{\eeqn}{\end{equation*}}

\newcommand{\beq}{\begin{equation}}
\newcommand{\eeq}{\end{equation}}

\newcommand{\bern}{\begin{eqnarray*}}
\newcommand{\eern}{\end{eqnarray*}}

\begin{document}

\newcommand{\WquotMap}{g}
\newcommand{\Cartan}{C}
\newcommand{\lra}{\longrightarrow}
\newcommand{\notalpha}{{! \alpha}}
\newcommand{\degalpha}{{e_\alpha}}
\newcommand{\barw}{{\overline w}}
\newcommand{\barone}{{\overline 1}}
\newcommand{\Gaf}{{G_\alpha^{\rm f}}}
\newcommand{\Kaf}{{K_\alpha^{\rm f}}}
\newcommand{\FT}{\fF}
\newcommand{\InvTheta}{\vartheta}
\newcommand{\cOne}{{c_1}}
\newcommand{\cTwo}{{c_2}}
\newcommand{\PL}{{P\!L}}
\newcommand{\barValpha}{{\bar V_\alpha}}
\newcommand{\pTau}{{\tau}}

\title{Nearby Cycle Sheaves for Stable Polar Representations}

\author{Mikhail Grinberg}
\email{misha@mishagrinberg.net}

\author{Kari Vilonen}
\address{School of Mathematics and Statistics, University of Melbourne,
VIC 3010, Australia, also Department of Mathematics and Statistics,
University of Helsinki, Helsinki, Finland}
\email{kari.vilonen@unimelb.edu.au, kari.vilonen@helsinki.fi}
\thanks{KV was supported in part by the ARC grants DP150103525,
DP180101445, FL200100141 and the Academy of Finland.
}

\author{Ting Xue}
\address{School of Mathematics and Statistics, University of Melbourne,
VIC 3010, Australia, also Department of Mathematics and Statistics,
University of Helsinki, Helsinki, Finland} 
\email{ting.xue@unimelb.edu.au}
\thanks{TX  was supported in part by the ARC grant DP150103525.}

\date{September 16, 2025}

\begin{abstract}
Let $G | V$, $G$ connected, reductive over $\bC$, be a stable polar
representation in the sense of \cite{DK}, satisfying some mild additional
hypotheses.  Given a $G$-equivariant rank one local system $\cL_\chi$ on the
general fiber of the quotient map $f : V \to V \inv G$, we compute the Fourier
transform of the corresponding nearby cycle sheaf $P_\chi$ on the zero-fiber
$X_0 = f^{-1} (0)$.  This provides a partial generalization of the results of 
\cite{Gr1} and \cite{GVX1}.  Our main intended application is to the theory of
character sheaves for graded Lie algebras over $\bC$.
\end{abstract}

\maketitle

\tableofcontents

\section{Introduction}

In this paper, we perform a detailed study of a nearby cycle construction which
produces character sheaves. Our work provides a partial generalization
of the results of \cite{Gr1} and \cite{GVX1}, and can be seen as a generalization
of the classical Springer construction of Weyl group representations.  As is well
known, this latter construction and its generalization by Lusztig play an important
role in understanding characters of finite groups of Lie type, via Lusztig's
theory of character sheaves.

Let $G$ be a complex reductive group acting on a vector space $V$.  We
will write $\on{Perv}_G (V)^{\text{nil}}$ for the category of $G$-equivariant
$\bC^*$-conic perverse sheaves on $V$ whose singular support is nilpotent;
and we will refer to simple objects of $\on{Perv}_G (V)^{\text{nil}}$ as
character sheaves.  This definition is consistent with Lusztig's use of
the term by \cite{MV}.  In a number of contexts, objects of
$\on{Perv}_G (V)^{\text{nil}}$ arise from systems of invariant differential
equations.  Lie algebra characters of real reductive groups and relative 
characters of such groups, relative to a maximal compact, are solutions to
such systems of differential equations.  In the first case, we have $V = \Lg$
with the adjoint action of $G$, and the second case arises from an involution
$\theta$ of a larger group $\widetilde G$.  Both of these cases fall under the
situations considered in this paper. 

In this paper, we study a rather general set-up of a stable polar representation
$G | V$, satisfying some mild additional hypotheses.  The key feature of this
set-up is the existence of a Cartan subspace $\Cartan \subset V$ and a complex
reflection group $W \subset \on{Aut} (\Cartan)$, such that the invariant theory
has a particularly nice form: $V \inv G = \Cartan / W$. 

The nearby cycle construction we study is a special case of the following
general process.  Let us write $f: V \to V \inv G$ for the quotient map.  We start
with a $G$-equivariant perverse sheaf $F$ on one of the fibers of $f$, take
nearby cycles to obtain a $\bC^*$-conic perverse sheaf on the null-cone in $V$,
and then take the Fourier transform.  By construction, we obtain an object
$\check F \in \on{Perv}_G (V^*)^{\text{nil}}$, whose simple constituents are
character sheaves on $V^*$.  The main content of this paper is in carrying
out this construction in the prototypical example of a stable polar representation
as above and a $G$-equivariant rank one local system on the general fiber of
$f$.

Our primary motivation comes from the closely related class of representations
arising from graded Lie algebras.  This is the case where $\theta$ is a finite
order automorphism of a reductive group $\widetilde G$, and we have
$G = \widetilde G^{\theta, 0}$, the identity component of the fixed points of
$\theta$, and $V = \widetilde \Lg_i$, an eigenspace of the automorphism
$\theta$ acting on $\widetilde \Lg$.  In this setting, the nearby cycle construction
is particularly efficient at generating character sheaves (as demonstrated in
\cite{CVX}, \cite{VX1}, \cite{VX2}, \cite{VX3}).  Two remarkable features of this
construction (in this setting) are as follows.  First, it can be shown that the
process described above takes simple IC-sheaves on the fibers of $f$ to
(not necessarily simple) IC-sheaves on $V^*$ (\cite{GVX2}).  And second,
it is shown in \cite{LTVX} that all cuspidal character sheaves arise from the
nearby cycle construction, starting with a very small set of character sheaves
with nilpotent support.  Thus, in the graded case, the nearby cycle construction
provides us with a new kind of induction.

\subsection{The main result}
\label{subsec-main-result}

We now describe our setting more precisely.  The class of polar representations
$G | V$, for $G$ connected, reductive over $\bC$, was introduced by Dadok
and Kac in \cite{DK}, as a class of linear actions whose invariant theory works
by analogy with the class of adjoint representations.  In fact, Dadok and Kac
defined polar representations for general reductive $G$, but the polarity
condition is most natural when $G$ is connected.  For every polar $G | V$
with $G$ connected, there exists a linear subspace $\Cartan \subset V$,
which is defined uniquely up to the action of $G$, with the property that:
\beq\label{eqn-quot}
V \inv G = \Cartan / W,
\eeq
where $W = N_G (\Cartan) / Z_G (\Cartan)$.  The subspace $\Cartan$ is called
a {\em Cartan subspace} of $V$, and the group $W$ is called the {\em Weyl group}
of $G | V$.  The group $W$ is a finite complex reflection group acting on $\Cartan$.
By a general result on complex reflection groups (see \cite[Theorem 4.1]{Br}), the
quotient $\Cartan / W$ is isomorphic to an affine space of the same dimension as
$\Cartan$.  Polar representations include many of the classical invariant problems
of linear algebra, as well as all representations arising from graded Lie algebras,
studied by Vinberg in \cite{Vi} (see Section \ref{subsec-motivation} below).

In this paper, we consider a polar representation $G | V$ of a connected reductive
group $G$, which is {\em stable} in the sense of \cite{DK}.  This stability condition
means that:
\beq\label{eqn-stability}
V = \Cartan \oplus \Lg \cdot \Cartan,
\eeq
where $\Lg = \on{Lie} (G)$.  We are interested in the quotient map:
\beqn
f : V \to Q \coloneqq \Cartan / W,
\eeqn
arising from the identification \eqref{eqn-quot}.  We assume that:
\beq\label{eqn-rank}
\on{rank} (G | V) \coloneqq \dim \Cartan - \dim \Cartan^G \geq 1.
\eeq
Let:
\beqn
\Cartan^{reg} = \{ c \in \Cartan \; | \; Z_W (c) = \{ 1 \} \} \;\; \text{and} \;\;
Q^{reg} = f (C^{reg}).
\eeqn
Pick a basepoint $c_0 \in \Cartan^{reg}$ and define:
\beq\label{eqn-X-zero}
\bar c_0 = f (c_0), \;\;
X_0 = f^{-1} (0), \;\;
X_{\bar c_0} = f^{-1} (\bar c_0),
\eeq
where $0 \in \Cartan / W$ is the image of $0 \in \Cartan$.  The stability condition
\eqref{eqn-stability} implies that $X_{\bar c_0} = G \cdot c_0$.

We consider a nearby cycle functor:
\beqn
\psi_f = \psi_f [\bar c_0]: \on{Perv}_G (X_{\bar c_0}) \to
\on{Perv}_G (X_0)_{\bC^*\text{-conic}},
\eeqn
from (shifted) $G$-equivariant local systems on $X_{\bar c_0}$ to $G$-equivariant
$\bC^*$-conic perverse sheaves on $X_0$; both considered with coefficients in
$\bC$.  This functor is defined as a specialization to the asymptotic cone, as in
\cite{Gr2}.  However, in order to justify the notation $\psi_f$, we wish to ensure
that the functor $\psi_f [\bar c_0]$ ``varies in a local system'' over the regular locus
$Q^{reg} \ni \bar c_0$.  To accomplish this, we impose the following further
condition on the representation $G | V$.  We assume that either $G | V$ is
{\em visible}, meaning that the zero-fiber $X_0$ consists of finitely many
$G$-orbits, or we have $\on{rank} (G | V) = 1$ (see equation \eqref{eqn-rank}).
Among other things, this ensures that the pure braid group:
\beqn
P\!B_W \coloneqq \pi_1 (\Cartan^{reg}, c_0),
\eeqn
acts on the functor $\psi_f [\bar c_0]$.

Our main result, Theorem \ref{thm-main}, computes the functor $\psi_f$ for
local systems of rank one, subject to two additional technical assumptions (see
\eqref{eqn-locality-assumption} and \eqref{eqn-tilde-r}).  Specifically, given a
character:
\beq\label{eqn-chi-I}
\chi : I \coloneqq \pi_1^G (X_{\bar c_0}, c_0) \to \bG_m \, ,
\eeq
we write $\cL_\chi$ for the corresponding rank one $G$-equivariant local system
on $X_{\bar c_0}$, and we let:
\beqn
P_\chi = \psi_f (\cL_\chi [-]),
\eeqn
where $[-]$ denotes the appropriate shift.  The statement of Theorem
\ref{thm-main} takes the following form:
\beq\label{eqn-main}
\FT P_\chi \cong \on{IC} ((V^*)^{rs}, \cM_\chi).
\eeq
Here, $\FT : \on{Perv}_G (V)_{\bC^*\text{-conic}} \to
\on{Perv}_G (V^*)_{\bC^*\text{-conic}}$ is the topological Fourier
transform functor, $(V^*)^{rs}$ is the union of all closed $G$-orbits of maximal
dimension in $V^*$, and the RHS is the IC-extension of a certain local system
$\cM_\chi$ on $(V^*)^{rs}$, with $\on{rank} (\cM_\chi) = |W|$.  The stability
condition \eqref{eqn-stability} implies that $(V^*)^{rs} \subset V^*$ is an open
subset, and the local system $\cM_\chi$ is specified by equation \eqref{eqn-M-chi}.

The fact that the sheaf $P_\chi$ admits a description of the form \eqref{eqn-main} 
follows readily from the results of \cite{Gr2} and \cite{Gr1}.  Thus, the main content
of Theorem \ref{thm-main} is in describing the local system $\cM_\chi$.  This 
description proceeds by reduction to the case of rank one.  We now give a
partial sketch of the construction of $\cM_\chi$.

Let $V^{rs} = f^{-1} (Q^{reg})$.  By the stability condition \eqref{eqn-stability}
and the additional assumption \eqref{eqn-locality-assumption},
the subset $V^{rs} \subset V$ is the union of all closed $G$-orbits of maximal
dimension.  We use a suitable Hermitian inner product on $V$ to pick a basepoint
$l_0 \in (V^*)^{rs}$, corresponding to the basepoint $c_0 \in V^{rs}$, and to identify
the equivariant fundamental group $\pi_1^G ((V^*)^{rs}, l_0)$ with $\widetilde B_W 
\coloneqq \pi_1^G (V^{rs}, c_0)$.  Thus, we can view $\cM_\chi$ as a
representation of the group $\widetilde B_W$.  This group fits into an exact
sequence:
\beq\label{eqn-widetilde-B-W-ES}
1 \lra I \lra \widetilde B_W \mathop{\lra}^{\tilde q} B_W \lra 1,
\eeq
where $B_W \coloneqq \pi_1 (Q^{reg}, \bar c_0)$ is the braid group of $W$.

Next, we define a subgroup $W_\chi^0 \subset W $ as follows.  The Weyl 
group $W$ acts naturally on the characters $\hat I \coloneqq \on{Hom}
(I, \bG_m)$.  Let $W_\chi = \on{Stab}_W (\chi)$.  Let
$\{ \Cartan_\alpha \}_{\alpha \in A}$ ($A$ a finite index set) be the set of all
reflection hyperplanes for $W$.  For each $\alpha \in A$, let $W_\alpha =
\on{Stab}_W (\Cartan_\alpha)$ and let $W_{\alpha, \chi} = W_\alpha \cap 
W_\chi$.  Then $W_\chi^0 \subset W$ is the subgroup generated by all
the $\{ W_{\alpha, \chi} \}_{\alpha \in A}$.  In other words, $W_\chi^0$ is the
largest complex reflection group contained in $W_\chi$.

Let $p : B_W \to W$ be the natural map, and let:
\beq\label{eqn-B-chi-zero}
B_W^{\chi, 0} = p^{-1} (W_\chi^0), \;\;\;\;
\widetilde B_W^{\chi, 0} = \tilde q^{-1} (B_W^{\chi, 0}).
\eeq
The $\widetilde B_W$-representation $\cM_\chi$ is obtained by inducing from
the subgroup $\widetilde B_W^{\chi, 0} \subset \widetilde B_W$, i.e., we have:
\beq\label{eqn-M-chi-induced}
\cM_\chi = \bC [\widetilde B_W] \otimes_{\bC [\widetilde B_W^{\chi, 0}]}
\cM_\chi^0 \, ,
\eeq
for some $\widetilde B_W^{\chi, 0}$-representation $\cM_\chi^0$, with
$\dim \cM_\chi^0 = |W_\chi^0|$.

To describe the representation $\cM_\chi^0$, let $A_\chi^0 \subset A$ be the set
of all reflection hyperplanes for $W_\chi^0$, i.e., the set of all $\alpha \in A$ with
$W_{\alpha, \chi} \neq \{ 1 \}$.  To each $\alpha \in A$ we associate a monic
polynomial $\bar R_{\chi, \alpha} \in \bC [z]$ of degree $|W_{\alpha, \chi}|$.
Let $B_{W_\chi^0}$ be the braid group associated to the complex reflection
group $W_\chi^0$.  Using the polynomials $\bar R_{\chi, \alpha}$ for $\alpha \in
A_\chi^0$, we define a Hecke algebra $\cH_{W_\chi^0}$ as a quotient of the
group algebra $\bC [B_{W_\chi^0}]$.  The algebra $\cH_{W_\chi^0}$ is naturally
a $\widetilde B_W^{\chi, 0}$-representation, and we have:
\beq\label{eqn-M-chi-zero}
\cM_\chi^0 \cong \cH_{W_\chi^0} \otimes
\bC_{\hat\chi \,\cdot\, \tilde \rho \,\cdot\, \tau} \, ,
\eeq
where $\hat\chi \cdot \tilde \rho \cdot \tau : \widetilde B_W^{\chi, 0} \to \bG_m$
is a certain character, obtained as a product of three ingredients, introduced
in \eqref{eqn-hat-chi}, \eqref{eqn-tilde-rho}, and \eqref{eqn-char-tau}, respectively.
In particular, the character $\tilde \rho$ encodes the polynomials
$\bar R_{\chi, \alpha}$ for $\alpha \in A - A_\chi^0$.

In conclusion, we briefly indicate the origin on the polynomials
$\bar R_{\chi, \alpha}$.  For each $\alpha \in A$, there is an associated
stable polar representation $G_\alpha | V_\alpha$ of rank one.  Namely,
$G_\alpha \subset G$ is the connected subgroup corresponding to the
stabilizer $\Lg_\alpha \coloneqq Z_\Lg (\Cartan_\alpha)$, and $V_\alpha =
\Cartan \oplus \Lg_\alpha \cdot \Cartan \subset V$.  Let $f_\alpha : V_\alpha
\to Q_\alpha \coloneqq V_\alpha \inv G_\alpha$ be the quotient map, and let:
\beqn
X_{0, \alpha} = f_\alpha^{-1} (0), \;\;
\breve c_0 = f_\alpha (c_0), \;\;
X_{\breve c_0, \alpha} = f_\alpha^{-1} (\breve c_0), \;\;
I_\alpha = \pi_1^{G_\alpha} (X_{\breve c_0, \alpha}, c_0).
\eeqn
The character $\chi \in \hat I$ restricts to a character $\chi_\alpha \in \hat I_\alpha$.
We can thus apply the construction of the nearby cycle sheaf $P_\chi$ to the data
$(G_\alpha | V_\alpha, \chi_\alpha)$, to obtain a sheaf $P_{\chi^{}_\alpha} \in
\on{Perv}_{G_\alpha} (X_{0, \alpha})$.  The analog of the isomorphism
\eqref{eqn-main} for the sheaf $P_{\chi^{}_\alpha}$ defines a local system
$\cM_{\chi^{}_\alpha}$ on $(V_\alpha^*)^{rs} \subset V_\alpha^*$.  The basepoint
$l_0 \in (V^*)^{rs}$ restricts to a basepoint $l_{0, \alpha} \in (V_\alpha^*)^{rs}$,
and the polynomial $\bar R_{\chi, \alpha}$ is defined as the minimal polynomial of
the holonomy of $\cM_{\chi^{}_\alpha}$ along a suitable element of
$\pi_1^{G_\alpha} ((V_\alpha^*)^{rs}, l_{0, \alpha})$.

Thus, we see that Theorem \ref{thm-main} stops short of giving an explicit
description of the local system $\cM_\chi$.  Rather, it describes $\cM_\chi$
in terms of the much simpler local systems $\cM_{\chi^{}_\alpha}$ for
$\alpha \in A_\chi^0$.  However, in many examples, the geometry of the
rank one representations $G_\alpha | V_\alpha$ is tractable, leading to an
explicit description of $\cM_\chi$.  See, for example, the discussion of
\cite{VX2} and \cite{VX3} in Section \ref{subsec-motivation} below.

\subsection{Motivation and relationship to prior work}
\label{subsec-motivation}

The motivating example for this work is the class of representations arising from
graded Lie algebras, studied in \cite{Vi}, and the motivating problem is to
understand character sheaves in this setting.  More precisely, let $\widetilde G$
be a reductive group over $\bC$, and let $\theta$ be an
automorphism:
\beq\label{eqn-theta}
\theta : \widetilde G \to \widetilde G,
\eeq
of finite order $m > 1$.  We have an eigenspace decomposition:
\beq\label{eqn-grading}
\widetilde \Lg \coloneqq \on{Lie} (\widetilde G) = \bigoplus_{i = 0}^{m-1} \,
\widetilde \Lg_i \, ,
\eeq
where $\theta |_{\widetilde \Lg_i} = \exp (2 \pi \bi \, i / m)$ and $\bi = \sqrt{-1}$.
Decomposition \eqref{eqn-grading} defines a grading on the Lie algebra
$\widetilde \Lg$.  Let $\widetilde G^\theta \subset \widetilde G$ be the group
of fixed points of $\theta$, and let $\widetilde G^{\theta, 0}$ be the identity
component of $\widetilde G^\theta$.  Note that we have $\on{Lie}
(\widetilde G^\theta) = \widetilde \Lg_0$.  The adjoint action of $\widetilde G$
restricts to an action of $\widetilde G^{\theta, 0}$ on each of the
$\widetilde \Lg_i$.  The case of a general $i \in \{1, \dots, m-1 \}$ is readily
reduced to the case $i = 1$.  We will focus on the representation $G | V$,
with $G = \widetilde G^{\theta, 0}$ and $V = \widetilde \Lg_1$.  This
representation is polar and visible, but not necessarily stable in the sense of
\eqref{eqn-stability}.  However, in many important cases, the representation
$G | V$ is stable, and then it satisfies the hypotheses of Theorem \ref{thm-main}.
Our results can be extended to some nonstable situations (see the discussion
of \cite{VX3} below), but we do not explore this direction in the present paper.

By a character sheaf on $V^*$ we mean a simple perverse sheaf $P \in
\on{Perv}_G (V^*)_{\bC^*\text{-conic}}$ with nilpotent singular support:
\beqn
SS (P) \subset V^* \times X_0 \subset V^* \times V,
\eeqn
where $X_0 \subset V$ is the null-cone, defined as in \eqref{eqn-X-zero}.
In situations where the representation $G | V$ is stable, equation \eqref{eqn-main}
is a rich source of character sheaves with full support on $V^*$.  Indeed, for
every character $\chi$ and every simple subquotient $\cK$ of $\cM_\chi$, the
IC-extension $\on{IC} ((V^*)^{rs}, \cK)$ is a character sheaf.  We should note
that, in all of the examples of this form that we have studied, the fundamental
group $I$ of equation \eqref{eqn-chi-I} is abelian (see Remark \ref{rmk-I-abelian}).
This partly explains why, in this paper, we restrict attention to local systems
$\cL_\chi$ of rank one.

In \cite{GVX1}, we considered the special case of the above situation arising from
an involutive automorphism $\theta$ as in \eqref{eqn-theta}, i.e., the case $m = 2$.
In fact, in that paper, we studied the sheaf $P_\chi$ equivariantly with respect to
the larger group $\widetilde G^\theta \supset \widetilde G^{\theta, 0}$.  The main
result of \cite{GVX1} (\cite[Theorem 3.6]{GVX1}) is essentially equivalent to the
corresponding claim of Theorem \ref{thm-main} (see Remark \ref{rmk-GVX}).
The papers \cite{CVX} and \cite{VX1} use the results of \cite{GVX1} to provide
a classification of character sheaves for classical symmetric pairs.  The papers
\cite{VX2} and \cite{VX3} use the results of the present paper to study character
sheaves arising from higher order automorphisms, yielding a classification of
cuspidal character sheaves for Vinberg's type I classical graded Lie algebras.

The paper \cite{VX2} studies character sheaves arising from the so-called GIT
stable gradings.  A grading is GIT stable if the corresponding representation
space $V$ contains a $G$-stable vector in the sense of geometric invariant
theory, that is, a vector $v \in V$ such that the orbit $G \cdot v$ is closed and
the centralizer $G_v$ is finite.  GIT stability implies the stability of $G | V$ in
the sense of \eqref{eqn-stability}, but not vice versa.  In this setting, subject to
a finite number of exceptions, the geometry of each rank one representation
$G_\alpha | V_\alpha$ is captured by a normal crossings stratification of
$V_\alpha$, and the polynomials $\bar R_{\chi, \alpha}$ can be computed
using the well-known quiver description of perverse sheaves on this stratification.

The paper \cite{VX3} studies character sheaves arising from a class of gradings
that afford cuspidal character sheaves.  This class includes all GIT
stable gradings, as well as some gradings for which the representation $G | V$
is not stable in the sense of \eqref{eqn-stability}, but possesses certain rather
special properties.  With some significant additional argument, the methods of
the present paper can be extended to these nonstable gradings.  In this setting,
the polynomials $\bar R_{\chi, \alpha}$ and their nonstable analogs can be
computed using $\cD$-module techniques, by being related to known b-function
calculations for certain prehomogeneous vector spaces.

The paper \cite{Gr1} considers the case of nearby cycles with constant
coefficients for an arbitrary (not necessarily stable) polar representation
$G | V$, which is visible or of rank one.  In the stable case, the main result
of \cite{Gr1} (\cite[Theorem 3.1]{Gr1}) is quite close, but not identical, to the 
$\chi = 1$ case of Theorem \ref{thm-main} (see Remark \ref{rmk-Gr1}).
In the present paper, we freely borrow the results of \cite{Gr1}, so we can
mostly focus on the aspects of the problem related to the twisting of the local
system $\cL_\chi$.  We should note that the proof of \cite[Theorem 3.1]{Gr1}
contained a gap, which was fixed in \cite{Gr3}, and the argument presented
here depends on this fix (see \cite[Section 3]{Gr3} and Proposition
\ref{prop-cyclic-vector} below).

We now comment on the main difference between the arguments in
\cite{GVX1} and in the present paper.  In both cases, the proof proceeds by
analyzing the Picard-Lefschetz theory of the restriction $l_0 |_{X_{\bar c_0}}$,
for a suitable choice of $l_0 \in (V^*)^{rs}$.  The critical points of this restriction 
are contained in the Cartan subspace $C \subset V$, and are indexed by the
Weyl group $W$.  In the case considered in \cite{GVX1}, the group $W$ is a
Coxeter group.  This enabled us to use a real structure on $\Cartan$, to place
all the critical values of $l_0 |_{X_{c_0}}$ on the real line $\bR \subset \bC$.
This, in turn, enabled us to choose a special system of ``cuts'' (lying in the
upper half-plane), and to define a distinguished Picard-Lefschetz basis up to
sign for the stalk $(\FT P_\chi)^{}_{l_0}$.  The proof of \cite[Theorem 3.6]{GVX1}
proceeded as a calculation in terms of this distinguished basis.  In the present
paper, we consider a situation where $W$ need not be a Coxeter group.  For
a general complex reflection group, we do not know how to arrange the critical
values of $l_0 |_{X_{\bar c_0}}$ in any special way, to produce a distinguished
basis.  Thus, to prove Theorem \ref{thm-main}, we had to find a ``more invariant''
argument.  In particular, the argument presented here differs substantially from
the argument in \cite{GVX1}, even in the case considered in that paper.

We also comment on the decision to write this paper in the generality of a
stable polar representation, as opposed to focusing on the corresponding
class of graded Lie algebras of \cite{Vi}.  We feel that the class of polar
representations neatly abstracts the key geometric features of the quotient
map $f : V \to Q$ which are needed for our arguments.  Thus, we hope that
working with polar representations makes the arguments clearer than if we
worked specifically with graded Lie algebras.  In addition, the class of stable
polar representations includes many other examples of potential future interest.
The downside of working with polar representations generally is that they are
rather less well known or well understood than graded Lie algebras.  As a
result, the statement of Theorem \ref{thm-main} contains some features
which are trivial in all of the examples studied by the authors.  See
Remarks \ref{rmk-DK-conj} and \ref{rmk-rho}.  It is possible that these
features are ``phantoms'', which can be ruled out through further study
of polar representations.  Regardless, we hope that they do not complicate
the application of our results in specific examples.

\subsection{Contents of this paper}
The paper is organized as follows.  Section \ref{sec-statement} introduces the
ingredients of our main Theorem \ref{thm-main}, culminating in the statement
of the theorem in Section \ref{subsec-statement}.  At this stage, our definition
of the polynomial $\bar R_{\chi, \alpha}$, for $\alpha \in A$, remains contingent
on some assertions about the representation $G_\alpha | V_\alpha$, whose
proofs are deferred to Section \ref{sec-carousel}.  See Propositions
\ref{prop-monic} and \ref{prop-R-alpha-vanishing}.  Our definition of
$\bar R_{\chi, \alpha}$ does not lead to a general explicit formula.  However,
Proposition \ref{prop-R-chi-alpha-at-zero}, which is also proved in Section
\ref{sec-carousel}, gives an explicit formula for the constant term
$\bar R_{\chi, \alpha} (0)$.  This formula leads to a number of other
explicit formulas later in the paper, such as equation \eqref{eqn-rho-explicit}
for the character $\rho$, underlying the ingredient $\tilde \rho$ of equation
\eqref{eqn-M-chi-zero}.

Section \ref{sec-prelim} furnishes the proofs of several preliminary results
stated in Section \ref{sec-statement}, needed to formulate Theorem
\ref{thm-main}, and providing context for this theorem.  In particular,
in Section \ref{subsec-hessians}, we discuss the Hessian $\cH [v, l]$
of the restriction $l |_{G \cdot v}$, for $v \in \Cartan$ and
$l \in (\Lg \cdot \Cartan)^\perp \subset V^*$.  Such Hessians play
a key role in the proof of Theorem \ref{thm-main}, which is given in
Sections \ref{sec-fourier}-\ref{sec-proof}.

In Section \ref{sec-fourier}, we invoke the results of \cite{Gr1} and \cite{Gr2}
to show that the Fourier transform $\FT P_\chi$ is an IC-sheaf with full
support on $V^*$.  More precisely, we establish isomorphism \eqref{eqn-main}
with $\cM_\chi = M (P_\chi)$, the Morse local system of $P_\chi$ at the origin
(see Proposition \ref{prop-fourier}).  We then begin the study of the local
system $M (P_\chi)$ on $(V^*)^{rs}$, using the methods of Picard-Lefschetz
theory.  Continuing with the basepoint $l_0 \in (V^*)^{rs}$, we interpret the
stalk $M_{l_0} (P_\chi)$ as a relative homology group of the general fiber
$X_{\bar c_0}$ with coefficients in $\cL_\chi$, and recall a standard
construction of Picard-Lefschetz classes in this group.  Using this
interpretation, we obtain our first three results about the local system
$M (P_\chi)$.  First, we show that $\on{rank} M (P_\chi) = |W|$.  Second,
we invoke the results of \cite{Gr1} and \cite{Gr3} to show that there exists
a cyclic vector $u_0 \in M_{l_0} (P_\chi)$ for the microlocal monodromy
(i.e., the holonomy) action of $\widetilde B_W \cong \pi_1^G ((V^*)^{rs}, l_0)$
(see Proposition \ref{prop-cyclic-vector}).  This existence of a cyclic vector is a
key step in identifying the microlocal monodromy as an induced representation,
as in \eqref{eqn-M-chi-induced}.  Finally, we compute the microlocal monodromy
action of the subgroup $I \subset \widetilde B_W$ (see exact sequence
\eqref{eqn-widetilde-B-W-ES}) on Picard-Lefschetz classes in $M_{l_0} (P_\chi)$
(see Proposition \ref{prop-microlocal-inertia}).

In Section \ref{sec-monodromy-construct}, we introduce the monodromy in
the family representation:
\beq\label{eqn-introduce-monodromy}
\mu : B_W^\chi \coloneqq p^{-1} (W_\chi) \to \on{Aut} (P_\chi).
\eeq
This structure arises out of the dependence of the sheaf $P_\chi$ on the choice
of the basepoint $c_0 \in \Cartan^{reg}$, and it plays a key role in the rest of the
argument.  In Section \ref{sec-monodromy}, we investigate the monodromy
action of a braid generator in $B_W^{\chi, 0} \subset B_W^\chi$ (see
\eqref{eqn-B-chi-zero}).  More precisely, for $\alpha \in A$, we let
$\degalpha = |W_\alpha| / |W_{\alpha, \chi}|$.  If $\sigma_\alpha \in B_W$ is
a braid generator corresponding to the hyperplane $C_\alpha \subset C$
(as defined in Section \ref{subsec-generators}), then we have
$\sigma_\alpha^\degalpha \in B_W^{\chi, 0}$, and we consider the operator
$\mu (\sigma_\alpha^\degalpha) \in \on{Aut} (P_\chi)$.  We interpret the minimal
polynomial $\bar R^\mu_{\chi, \alpha}$ of this operator in terms of the rank one
representation $G_\alpha | V_\alpha$ (see Proposition \ref{prop-family-min-poly}),
and we use the geometry of the full representation $G | V$ to establish the degree
bound:
\beq\label{eqn-mu-deg-bound}
\deg \bar R^\mu_{\chi, \alpha} \leq |W_{\alpha, \chi}|,
\eeq
(see Proposition \ref{prop-degree-bound}).

In Section \ref{sec-carousel}, we make a detailed study of the rank one
representation $G_\alpha | V_\alpha$, $\alpha \in A$.  Picard-Lefschetz theory in
this setting can be understood quite explicitly, in terms of the so-called carousel
technique of singularity theory (see Remark \ref{rmk-carousel}).  This enables us
to relate the microlocal  monodromy and the monodromy in the family actions on
the Morse group $M_{l_{0, \alpha}} (P_{\chi^{}_\alpha})$, where $l_{0, \alpha} =
l_0 |_{V_\alpha} \in (V_\alpha^*)^{rs}$ (see Proposition \ref{prop-carousel}).
We then use this relationship, together with the degree bound
\eqref{eqn-mu-deg-bound}, to complete the definition of the polynomial
$\bar R_{\chi, \alpha}$, and to relate it to the polynomial
$\bar R^\mu_{\chi, \alpha}$ (see equation \eqref{eqn-apply-theta} and
Remark \ref{rmk-apply-theta}).

In Section \ref{sec-microlocal-generator}, we partially extend the relationship
between the microlocal monodromy and the monodromy in the family,
established in Section \ref{sec-carousel}, to the setting of the full representation
$G | V$ (see Proposition \ref{prop-microlocal-generator}).  A key input here is
the discussion of the Hessians in Section \ref{subsec-hessians}.  In Section
\ref{sec-proof}, we assemble the geometric inputs of Sections
\ref{sec-fourier}-\ref{sec-microlocal-generator} to complete the proof of Theorem
\ref{thm-main}.  The argument here is based on the interplay between two
commuting actions on the Morse group $M_{l_0} (P_\chi)$: the microlocal
monodromy action of $\widetilde B_W$ and the monodromy in the family
action of $B_W^{\chi, 0}$.  Finally, in Section \ref{sec-conj}, we formulate a
conjecture regarding the vanishing of certain intersection numbers, which
provides some context for the statement and the proof of Theorem \ref{thm-main},
form the point of view of the classical Picard-Lefschetz formula (see Conjecture
\ref{conj-vanishing}).

A number of arguments in this paper are closely parallel, in whole or in part, to
corresponding arguments in \cite{GVX1}.  In those cases, we refer to \cite{GVX1}
and indicate the necessary adaptations, rather than repeating the entire
argument.  See, for example, the proofs of Propositions \ref{prop-regular-exists},
\ref{prop-A-F2}, and \ref{prop-family-min-poly}.

{\bf Acknowledgements.} We thank Michel Brou\'e, David Massey, and Jean
Michel for helpful pointers to the literature.

\section{Statement of the main theorem}
\label{sec-statement}

\subsection{The polar representation \texorpdfstring{$G | V$}{Lg}}
\label{subsec-polar-rep}

Let $G | V$ be a polar representation of a connected reductive algebraic group
over $\bC$ (see \cite{DK}).  Let $\Cartan \subset V$ be a Cartan subspace, and
let $W = N_G (\Cartan) / Z_G (\Cartan)$ be the associated Weyl group; it is a
finite complex reflection group acting on $\Cartan$.  We assume that:
\beqn
\on{rank} (G | V) \coloneqq \dim \Cartan - \dim \Cartan^G \geq 1.
\eeqn
Let $Q = \Cartan / W = V \inv G$, and let $f : V \to Q$ be the quotient map.
The main properties of $G | V$ related to the geometry $f$ are summarized in
\cite[Section 2.3]{Gr1}.

The representation $G | V$ is called visible if the zero-fiber $X_0 = f^{-1} (0)$
consists of finitely many $G$-orbits.  We assume that:
\beq\label{eqn-visibility}
\text{either $G | V$ is visible or $\on{rank} (G | V) = 1$.}
\eeq
We also assume that $G | V$ is stable in the sense of \cite[p. 506]{DK}, i.e.,
that for every $c \in \Cartan$ with $Z_W (c) = \{ 1 \}$ and every $v \in V$, we
have $\dim G \cdot c \geq \dim G \cdot v$.  Equivalently, this condition can be
expressed by equation \eqref{eqn-stability} (see \cite[Corollary 2.5]{DK}).

We now introduce some notation related to the complex reflection group $W$.
Let $\{ C_\alpha \}_{\alpha \in A}$ be the set of all reflection hyperplanes for
$W$ (where $A$ is a finite index set).  We write:
\beq\label{eqn-C-sing}
\Cartan^{sing} = \bigcup_{\alpha \in A} \Cartan_\alpha \, .
\eeq
Let $\Lg = \on{Lie} (G)$.  For every $\alpha \in A$, let $\Lg_\alpha =
Z_\Lg (\Cartan_\alpha)$, let $G_\alpha \subset G$ be the connected
subgroup corresponding to $\Lg_\alpha \subset \Lg$, and let:
\beq\label{eqn-V-alpha}
V_\alpha = \Cartan \oplus \Lg_\alpha \cdot \Cartan \subset V.
\eeq
The subspace $V_\alpha$ is preserved by the action of $G_\alpha$, and the
representation $G_\alpha | V_\alpha$ is also polar and stable, with Cartan
subspace $\Cartan$ (see \cite[Theorem 2.12]{DK}).  Let:
\beq\label{eqn-W-alpha}
W_\alpha = N_{G_\alpha} (\Cartan) / Z_{G_\alpha} (\Cartan),
\eeq
be the Weyl group associated to $G_\alpha | V_\alpha$.  Note that we have
$\on{rank} (G_\alpha | V_\alpha) \leq 1$, and therefore, the group $W_\alpha$
is cyclic, with $W_\alpha = \{ 1 \}$ iff $\on{rank} (G_\alpha | V_\alpha) = 0$ (in
which case, we have $\Lg_\alpha = Z_\Lg (\Cartan)$ and $V_\alpha = \Cartan$).
We make the following further ``locality'' assumption:
\beq\label{eqn-locality-assumption}
W_\alpha = Z_W (\Cartan_\alpha) \text{ for every } \alpha \in A.
\eeq
This assumption guarantees that:
\beq\label{eqn-rank-V-alpha}
\on{rank} (G_\alpha | V_\alpha) = 1,
\eeq
for every $\alpha \in A$.  It also ensures that our definition of $\Cartan^{sing}$
in equation \eqref{eqn-C-sing} is consistent with the definition of $c_{sing}
\subset c$ in \cite{DK} (see top of \cite[p. 515]{DK}).

\begin{remark}\label{rmk-DK-conj}{\em
Assumption \eqref{eqn-locality-assumption} holds in all examples known to
the authors.  Moreover, in \cite[p. 521, Conjecture 2]{DK}, Dadok and Kac
conjecture that, for every polar representation $G | V$ of a connected $G$,
the Weyl group $W$ is generated by the subgroups
$\{ W_\alpha \}_{\alpha \in A}$, defined as in \eqref{eqn-W-alpha}.
This property of $G | V$ is equivalent to our assumption
\eqref{eqn-locality-assumption}, as can be seen using
\cite[Corollary 2.36]{BCM}.}
\end{remark}

\subsection{The fundamental group \texorpdfstring{$\pi_1^G (V^{rs}, c_0)$}{Lg}}
\label{subsec-pi-one}

Let:
\beq\label{eqn-C-reg}
\Cartan^{reg} = \Cartan - \Cartan^{sing}, \;\;
Q^{reg} = f (\Cartan^{reg}) = \Cartan^{reg} / W, \;\;
V^{rs} = f^{-1} (Q^{reg}).
\eeq
By a theorem of Steinberg (\cite[Theorem 1.5]{St}; see also \cite{Leh}), we have:
\beqn
\Cartan^{reg} = \{ c \in \Cartan \; | \; Z_W (c) = \{ 1 \} \}.
\eeqn
The stability assumption \eqref{eqn-stability} implies that:
\beq\label{eqn-V-rs}
V^{rs} = G \cdot \Cartan^{reg}.
\eeq
Assumption \eqref{eqn-locality-assumption} further implies that $V^{rs}$ is the
union of all closed $G$-orbits of maximal dimension in $V$.  We can therefore
refer to $V^{rs} \subset V$ as the regular semisimple locus, and be consistent
with the terminology of \cite{DK}.  In this subsection, we describe the
$G$-equivariant fundamental group of $V^{rs}$.

Pick a basepoint $c_0 \in \Cartan^{reg}$, and let $\bar c_0 = f (c_0) \in Q^{reg}$.
We define:
\beqn
P\!B_W \coloneqq \pi_1 (\Cartan^{reg}, c_0), \;\;
B_W \coloneqq \pi_1 (Q^{reg}, \bar c_0), \;\;
\widetilde B_W \coloneqq \pi_1^G (V^{rs}, c_0).
\eeqn
Here, $B_W$ is the braid group associated to the complex reflection group $W$, 
$P\!B_W \subset B_W$ is the subgroup of pure braids, and we use the same 
conventions on equivariant fundamental groups as in \cite[Section 2.3]{GVX1}.
Namely, given an algebraic group $\cG$, we write $\cG^0 \subset \cG$ for the 
identity component, and we have $\pi_1^\cG (pt) = \cG / \cG^0$.  Further, if $\cG$
acts on a connected variety  $\cX$ with a basepoint $x_0 \in \cX$, we identify
$\cG$-equivariant local systems on $\cX$ with left representations of $\pi_1^{\cG}
(\cX, x_0)$.  In other words, we multiply loops in $\pi_1^{\cG} (\cX, x_0)$ by
tracing out the second loop first.  Let $X_{\bar c_0} = f^{-1} (\bar c_0)$ and
define:
\beq\label{eqn-I}
I \coloneqq \pi_1^G (X_{\bar c_0}, c_0) = Z_G (\Cartan) / Z_G (\Cartan)^0,
\eeq
(where $Z_G (\Cartan)^0 \subset Z_G (\Cartan)$ is the identity component).
Note that $I$ is a finite group.

\begin{remark}\label{rmk-I-abelian}{\em
In all of the applications of the results of this paper currently envisioned by
the authors, the group $I$ happens to be abelian.  However, we do not know
how generally this is true.  The proof of Theorem \ref{thm-main} would not
be simplified significantly if we assumed that $I$ is abelian.}
\end{remark}

Define:
\beq\label{eqn-tilde-W}
\widetilde W \coloneqq N_G (\Cartan) / Z_G (\Cartan)^0.
\eeq
As in \cite[Section 2.3]{GVX1}, we have a commutative diagram:
\beq
\label{eqn-main-diagram}
\begin{CD}
1 @>>> I @>>> \widetilde B_W @>{\tilde q}>> B_W @>>> 1 \;\,
\\
@. @| @VV{\tilde p}V @VV{p}V @.
\\
1 @>>> I @>>> \widetilde W @>{q}>> W @>>> 1 \, ,
\end{CD}
\eeq
the right square of which is Cartesian, i.e., we have:
\beq
\label{eqn-concrete}
\widetilde B_W \cong \{ (\widetilde w, b) \in \widetilde W \times B_W \; | \;
q (\widetilde w) = p (b) \}.
\eeq
To define the maps in this diagram, note that we have:
\begin{align*}
B_W & =
\pi_1 (\Cartan^{reg} / W, \bar c_0) \cong \pi_1^W (\Cartan^{reg}, c_0), \\
\widetilde B_W & =
\pi_1^G (V^{rs}, c_0) \cong \pi_1^{N_G (\Cartan)}
(\Cartan^{reg}, c_0) \cong \pi_1^{\widetilde W} (\Cartan^{reg}, c_0).
\end{align*}
The maps $p$ and $\tilde p$ are given by mapping $\Cartan^{reg}$ to
a point, the map $q$ is the natural quotient map, and the map $\tilde q$ is
induced by $q$.  Note that the maps $p$ and $\tilde p$ are surjective,
and we have:
\beq\label{eqn-PBW}
\ker (p) = \ker (\tilde p) = P\!B_W.
\eeq

It is helpful to be able to think about the map $p$ and the isomorphism 
\eqref{eqn-concrete} geometrically, in terms of loops.  Let $\gamma : [0,1] \to
Q^{reg}$ be a continuous path with $\gamma (0) = \gamma (1) = \bar c_0$.
Then $\gamma$ represents an element of $B_W$.  Let $\widetilde\gamma : 
[0, 1] \to \Cartan^{reg}$ be the unique lift of $\gamma$ with
$\widetilde\gamma (0) = c_0$.  Then the image $p (\gamma) \in W$ is defined 
by the equation:
\beq\label{eqn-map-p}
\widetilde\gamma (1) = p (\gamma)^{-1} \, c_0 \, .
\eeq
See \cite[Section 2.2]{Gr3} for a detailed discussion of the appearance of
the inverse in the RHS of equation \eqref{eqn-map-p} (in the notation of 
\cite{Gr3}, the map $p$ would be called $\eta_{c_0}^\flat$).

Suppose now $\widetilde w \in \widetilde W$ is an element with 
$q (\widetilde w) = p (\gamma)$.  Pick a representative $g_1 \in N_G (\Cartan)$
of $\widetilde w$ and a path $g : [0, 1] \to G$ with $g (0) = 1$ and $g (1) = g_1$.
The pair $(g, \gamma)$ defines a path $\Gamma : [0, 1] \to V^{rs}$ as follows:
\beq\label{eqn-concrete-loop}
\Gamma (t) = 
\left\{ \begin{array}{ll}
\widetilde\gamma (2 t) \in \Cartan^{reg} & \text{ for } t \in [0, 1/2] \\
g (2 t - 1) \, \widetilde\gamma (1) \in X_{\bar c_0} & \text{ for } t \in [1/2, 1].
\end{array}\right.
\eeq
By equation \eqref{eqn-map-p}, we have $\Gamma (1) = q (\widetilde w) \,
p (\gamma)^{-1} \, c_0 = p (\gamma) \, p (\gamma)^{-1} \, c_0 = c_0 =
\Gamma (0)$.  Thus, the path $\Gamma$ is a closed loop, representing an
element of $\widetilde B_W$.  It is not hard to check that  $\tilde p (\Gamma) =
\widetilde w$ and $\tilde q (\Gamma) = \gamma$.  Thus, the element
$\Gamma \in \widetilde B_W$ corresponds to the pair $(\widetilde w, \gamma)
\in \widetilde W \times B_W$ under the isomorphism \eqref{eqn-concrete}.

Let $\hat I \coloneqq \on{Hom} (I, \bG_m)$ be the set of characters of $I$.
It follows from diagram~\eqref{eqn-main-diagram} that the conjugation action of
$\widetilde B_W$ on itself gives rise to an action of $B_W$ on $\hat I$, and that 
this action factors through the map $p : B_W \to W$.  We will denote these
actions of $B_W$ and $W$ on $\hat I$ by ``$\; \cdot \;$'', so that we have:
\beq\label{eqn-factor-action}
b \cdot \chi = p (b) \cdot \chi \in \hat I \;\; \text{for all} \;\;
b \in B_W \;\; \text{and} \;\; \chi \in \hat I.
\eeq
In the event the group $I$ is abelian (see Remark \ref{rmk-I-abelian}), we also
obtain actions of $B_W$ and $W$ on $I$, which are compatible with the actions
on $\hat I$.

\subsection{Braid generators}
\label{subsec-generators}

We will need the following construction of elements of $B_W$.
Pick an $\alpha \in A$.  Let:
\beq\label{eqn-not-alpha}
A_\notalpha = A - \{ \alpha \}, \;\;\;\;
\Cartan_\notalpha \; = \; \bigcup_{\beta \in A_\notalpha} \Cartan_\beta \, , \;\;\;\;
\Cartan_\alpha^{reg} \; = \; \Cartan_\alpha - \Cartan_\notalpha \, .
\eeq
Let $s_\alpha \in W_\alpha$ be the counter-clockwise primitive generator, and let:
\beq\label{eqn-n-alpha}
n_\alpha = |W_\alpha|.
\eeq
Pick a point $c_\alpha \in C_\alpha^{reg}$ and a nearby point
$c_{\alpha, 1} \in \Cartan^{reg}$, such that:
\beq\label{eqn-c-alpha}
c_\alpha = \frac{1}{n_\alpha} \, \cdot \,
\sum_{w \in W_\alpha} w \, c_{\alpha, 1} \, .
\eeq
Let $\Gamma_\alpha [c_{\alpha, 1}] : [0, 1] \to \Cartan^{reg}$ be the continuous
path given by:
\beq\label{eqn-gamma-alpha}
\Gamma_\alpha [c_{\alpha, 1}] \, (t) = c_\alpha +
\exp (2 \pi \bi \, t / n_\alpha) (c_{\alpha, 1} - c_\alpha).
\eeq
Note that $\Gamma_\alpha [c_{\alpha, 1}] \, (0) = c_{\alpha, 1}$ and
$\Gamma_\alpha [c_{\alpha, 1}] \, (1) = s_\alpha \, c_{\alpha, 1}$.  Thus, the
composition $f \circ \Gamma_\alpha [c_{\alpha, 1}] : [0, 1] \to Q^{reg}$ is a
closed loop.  Pick a continuous path $\Gamma : [0, 1] \to \Cartan^{reg}$ with
$\Gamma (0) = c_0$ and $\Gamma (1) = c_{\alpha, 1}$.  Define:
\beq\label{eqn-sigma-alpha}
\sigma_\alpha [\Gamma] = (f \circ \Gamma^{-1}) \star
(f \circ \Gamma_\alpha [c_{\alpha, 1}]) \star (f \circ \Gamma) \in B_W \, ,
\eeq
where ``$\; \star \;$'' denotes the composition of paths.  By equation
\eqref{eqn-map-p}, we have:
\beq\label{eqn-p-sigma-alpha}
p (\sigma_\alpha [\Gamma]) = s_\alpha^{-1} \, .
\eeq
We will refer to the element $\sigma_\alpha [\Gamma] \in B_W$ as a
counter-clockwise braid generator for $\alpha$, and we will write $B_W [\alpha]
\subset B_W$ for the set of all such counter-clockwise braid generators for
$\alpha$.  Note that any two elements $\sigma_1, \sigma_2 \in B_W [\alpha]$
are conjugate to each other by an element of $P\!B_W$, i.e.:
\beq\label{eqn-conjugate}
\forall \, \sigma_1, \sigma_2 \in B_W [\alpha], \;\;
\exists \, b \in P\!B_W \; : \; \sigma_2 = b \, \sigma_1 \, b^{-1}.
\eeq
Note also that, for every $\alpha_1 \in A$, $\sigma_1 \in B_W [\alpha_1]$, and
$b \in B_W$, we have:
\beq\label{eqn-conjugate-two}
b \, \sigma_1 \, b^{-1} \in B_W [\alpha_2],
\eeq
where $\alpha_2 = p (b) \, \alpha_1 \in A$.

\subsection{The concept of a regular splitting}
\label{subsec-reg-split}

To state our main result, we will need another piece of structure related to 
diagram \eqref{eqn-main-diagram}.  Namely, a splitting of the top row of
\eqref{eqn-main-diagram} satisfying a certain ``locality'' condition with respect
to the braid generators in $B_W$.

Recall the subgroup $G_\alpha \subset G$, $\alpha \in A$, introduced following
\eqref{eqn-C-sing}.  For each $\alpha \in A$, define:
\beq\label{eqn-I-alpha}
I_\alpha \coloneqq Z_{G_\alpha} (\Cartan) / Z_{G_\alpha} (\Cartan)^0, \;\;
\widetilde W_\alpha \coloneqq N_{G_\alpha} (\Cartan) /
Z_{G_\alpha} (\Cartan)^0,
\eeq
(cf. \eqref{eqn-I} and \eqref{eqn-tilde-W}).
We have a natural projection $q_\alpha : \widetilde W_\alpha \to W_\alpha$,
and a short exact sequence:
\beq\label{eqn-seq-alpha}
1 \lra I_\alpha \lra \widetilde W_\alpha \xrightarrow{\; q_\alpha \;}
W_\alpha \lra 1 \, .
\eeq
Note that we have $Z_{G_\alpha} (\Cartan)^0 = Z_G (\Cartan)^0$.  Therefore,
we have inclusions $\widetilde W_\alpha \hookrightarrow \widetilde W$ and
$I_\alpha \hookrightarrow I$.  We will use these inclusions to view $\widetilde 
W_\alpha$ and $I_\alpha$ as subgroups of $\widetilde W$ and $I$, respectively.
Note that, by construction:
\beq\label{eqn-I-invariance}
\text{for every $x \in I$ and $\alpha \in A$, we have
$x \, \widetilde W_\alpha \, x^{-1} = \widetilde W_\alpha$ and
$x \, I_\alpha \, x^{-1} = I_\alpha \,$.}
\eeq

\begin{remark}{\em
In the special case where $G | V$ comes from an involutive automorphism
$\theta$, as in Section \ref{subsec-motivation}, it is known that the group $I$
is generated by the subgroups $\{ I_\alpha \}_{\alpha \in A}$.  See
\cite[Theorem 7.55]{K} and \cite[Remark 2.4]{GVX1}.  We do not know
whether this is the case for a general $G | V$ as in this paper.}
\end{remark}

\pagebreak

\begin{definition}\label{defn-reg-split}
Let $\tilde r : B_W \to \widetilde B_W$ be a homomorphism which splits the
top row of diagram \eqref{eqn-main-diagram}, i.e., we have $\tilde q \circ
\tilde r = \on{Id}_{B_W}$.  We say that $\tilde r$ is a regular splitting if the
composition $r = \tilde p \circ\tilde r : B_W \to \widetilde W$ satisfies:
\beqn
r (\sigma_\alpha) \in \widetilde W_\alpha \subset \widetilde W,
\eeqn
for every $\alpha \in A$ and every $\sigma_\alpha \in B_W [\alpha]$.
\end{definition}

\begin{remark}\label{rmk-reg-split}{\em
Let $\tilde r : B_W \to \widetilde B_W$ be a homomorphism which splits the top
row of \eqref{eqn-main-diagram}, and let $\alpha \in A$.  Then, by
\eqref{eqn-conjugate} and \eqref{eqn-I-invariance}, we have:
\beqn
r (\sigma_\alpha) \in \widetilde W_\alpha \text{ for some }
\sigma_\alpha \in B_W [\alpha] \iff
r (\sigma_\alpha) \in \widetilde W_\alpha \text{ for every }
\sigma_\alpha \in B_W [\alpha].
\eeqn}
\end{remark}

\begin{prop}\label{prop-regular-exists}
If there exists a nilpotent $x \in X_0$ which is regular for $f$, i.e., we have
$\on{rank} d_x f = \dim Q$, then there exists a regular splitting
$\tilde r : B_W \to \widetilde B_W$.
\end{prop}

A proof of Proposition \ref{prop-regular-exists} will be given in Section
\ref{subsec-reg-split-proof}.  It uses a normal slice to $X_0$ through the smooth
point $x \in X_0$ to lift loops in $Q^{reg}$ to loops in $V^{rs}$.

\begin{remark}{\em
A regular nilpotent does not always exist for a stable polar representation $G | V$
of a connected $G$.  For example, if $G = \bG_m$ acts on $V = \bC^2$ by
$t : (x, y) \mapsto (t^2 \, x, t^{-3} \, y)$, then there is no regular nilpotent for $G | V$.
However, regular nilpotents exist in all of the applications envisioned by the authors.}
\end{remark}

For the rest of this paper, we assume that:
\beq\label{eqn-tilde-r}
\text{a regular splitting homomorphism $\tilde r : B_W \to \widetilde B_W$
exists and has been fixed,}
\eeq
and we write $r = \tilde p \circ \tilde r : B_W \to \widetilde W$.

\begin{remark}{\em
The splitting homomorphism $\tilde r$ will not typically respect the identification
\eqref{eqn-PBW}, and the map $r$ will not typically factor through $p$.}
\end{remark}

\subsection{The nearby cycle sheaf \texorpdfstring{$P_\chi$}{Lg}}
\label{subsec-char-chi}

Throughout this paper, we work with local systems and sheaves of
$\bC$-vector spaces.  Fix a character:
\beqn
\chi \in \hat I = \on{Hom} (I, \bG_m).
\eeqn
Note that, since $I$ is a finite group, all values of $\chi$ are roots of unity
in $\bG_m = \bC^*$.  The character $\chi$ gives rise to a rank one
$G$-equivariant local system $\cL_\chi$ on $X_{\bar c_0}$, with
$(\cL_\chi)_{c_0} = \bC$.  As in \cite[Section 3.1]{GVX1}, we associate
to the pair $(X_{\bar c_0}, \cL_\chi)$ a nearby cycle sheaf:
\beqn
P_\chi \in \on{Perv}_G (X_0)_{\bC^*\text{-conic}} \, .
\eeqn
The construction is as follows.  We apply base change to the quotient map
$f: V \to Q$, to form a family:
\beq\label{eqn-family-Z}
\cZ_{\bar c_0} = \{ (x, k) \in V \times \bC \; | \; f (x) = k \, \bar c_0 \} \to \bC,
\eeq
where the multiplicative action of $\bC$ on $Q = \Cartan / W$ is induced by
the action on $\Cartan$, so that $f (k \, c) = k \, f (c)$ for all $k \in \bC$ and
$c \in \Cartan$.  Let $\cZ_{\bar c_0}^{rs} = \{ (x, k) \in \cZ_{\bar c_0} \; | \;
k \neq 0 \}$, and let $p_{\bar c_0} : \cZ_{\bar c_0}^{rs} \to X_{\bar c_0}$ be
the map $p_{\bar c_0} : (x, k) \mapsto k^{-1} \, x$.  By construction, the
family \eqref{eqn-family-Z} is $G$-equivariant.  Let us write:
\beqn
\psi_{\bar c_0} : \on{Perv}_G (\cZ_{\bar c_0}^{rs}) \to \on{Perv}_G (X_0),
\eeqn
for the nearby cycle functor with respect to this family.  We define:
\beq\label{eqn-P-chi}
\psi_f [\bar c_0] = \psi_{\bar c_0} \circ p_{\bar c_0}^* :
\on{Perv}_G (X_{\bar c_0}) \to \on{Perv}_G (X_0) \;\; \text{and} \;\;
P_\chi = \psi_f [\bar c_0] \, (\cL_\chi [-]).
\eeq
Here and henceforth, $[-]$ denotes an appropriate cohomological shift, so
that the resulting sheaf is perverse.

By construction, the sheaf $P_\chi$ is $\bC^*$-conic, and our main result,
Theorem \ref{thm-main}, describes the topological Fourier transform
$\FT P_\chi$ as an intersection homology sheaf on the dual space $V^*$.  
To fix the conventions, we will be using \cite[Definition 3.7.8]{KS} to define
the equivalence:
\beqn
\FT : \on{Perv}_G (V)_{\bC^*\text{-conic}} \to
\on{Perv}_G (V^*)_{\bC^*\text{-conic}} \,
\eeqn
up to a shift.  In Section \ref{sec-monodromy-construct}, we will introduce a
monodromy in the family structure, which reflects the dependence of the
sheaf $P_\chi$ on the basepoint $c_0 \in \Cartan^{reg}$.

\subsection{The dual representation \texorpdfstring{$G | V^*$}{Lg}}
\label{subsec-dual-rep}

The dual representation $G | V^*$ is also polar and stable, with Cartan
subspace:
\beq\label{eqn-C-star}
\Cartan^* \coloneqq (\Lg \cdot \Cartan)^\p \subset V^*,
\eeq
and the same Weyl group $W$ (see \cite[Proposition 2.13 (i)]{Gr1}).
We define $(\Cartan^*)^{reg} \subset \Cartan^*$ by analogy with equation
\eqref{eqn-C-reg}, and we let $(V^*)^{rs} = G \cdot (\Cartan^*)^{reg} \subset
V^*$ (cf. equation \eqref{eqn-V-rs}).  One can check that assumption
\eqref{eqn-locality-assumption} implies the analogous assertion about the
representation $G | V^*$, and therefore, the set $(V^*)^{rs}$ is the union of
all closed $G$-orbits of maximal dimension in $V^*$ (cf. discussion following
\eqref{eqn-V-rs}).

\begin{prop}\label{prop-A-f-stratification}
$\;$
\begin{enumerate}[topsep=-1.5ex]
\item[(i)]    There exists a $G$-invariant $\bC^*$-conic (algebraic, Whitney)
stratification $\cS_0$ of $X_0$, such that Thom's $A_f$ condition holds for
the pair $(V^{rs}, S)$, for every $S \in \cS_0$.

\item[(ii)]   For every $\cS_0$ as in part (i) and every $l \in (V^*)^{rs}$, the
pair $(0, l) \in T^* V = V \times V^*$ is a generic covector at the origin for
$\cS_0$, in the sense of stratified Morse theory.

\item[(iii)]  For every $\cS_0$ as in part (i), we have:
\beqn
P_\chi \in \on{Perv}_G (X_0, \cS_0),
\eeqn
i.e., the sheaf $P_\chi$ is constructible with respect to $\cS_0$.
\end{enumerate}
\end{prop}

The reader is referred to \cite[Section 11]{Mat} for a definition of the $A_f$
condition (see also \cite[Section 5.5]{GVX1}).  Proposition
\ref{prop-A-f-stratification} will be proved in Section \ref{subsec-geom-prelim}.
If $G | V$ is visible, we can take $\cS_0$ in part (i) to be the $G$-orbit
stratification.  For the rest of this paper, we fix a stratification $\cS_0$ as in
Proposition \ref{prop-A-f-stratification} (i), and for the sake of concreteness,
we assume that, if $G | V$ is visible, then $\cS_0$ is the $G$-orbit
stratification.  Proposition \ref{prop-A-f-stratification} implies that the Fourier
transform $\FT P_\chi$ restricts to a (shifted) local system on $(V^*)^{rs}$.

Thus, to state our main result, we will need to specify a $G$-equivariant local
system on $(V^*)^{rs}$.  For this, it will be convenient to identify the
$G$-equivariant fundamental groups of $V^{rs}$ and $(V^*)^{rs}$ as follows.
Pick a compact form $K \subset G$ and a $K$-invariant Hermitian inner
product $\langle \; , \, \rangle$ on $V$.  We follow the convention that
$\langle \; , \, \rangle$ is linear in the first argument, and antilinear in the
second.  An element $v \in V$ is called semisimple if the orbit $G \cdot v
\subset V$ is closed.  We say that a semisimple $v \in V$ is of minimal length
if it is of minimal length in its $G$-orbit, with respect to $\langle \; , \, \rangle$.
Write $V^{ms} \subset V$ for the set of all semisimple elements of minimal
length, and let $V^{mrs} = V^{ms} \cap V^{rs}$.  By \cite{KN} (see also
\cite[Theorem 1.1]{DK}), the set $V^{ms}$ can be characterized as follows:
\beq\label{eqn-kempf-ness}
V^{ms} = \{ v \in V \; | \; \langle \Lg \cdot v, v \rangle = 0 \}.
\eeq
By the proof of \cite[Lemma 2.1]{DK}, for every $v \in V^{mrs}$, we have
$\Cartan_v \subset V^{ms}$, where $\Cartan_v$ is the (unique) Cartan
subspace containing $v$.  Therefore, without loss of generality, we can
assume that:
\beq\label{eqn-min-cartan}
\Cartan \subset V^{ms},
\eeq
and we proceed with this assumption.  By \cite{KN} and
\cite[Proposition 2.2]{DK}, we then have:
\beq\label{eqn-core}
V^{ms} = K \cdot \Cartan \;\; \text{and} \;\;
V^{mrs} = K \cdot \Cartan^{reg}.
\eeq
Also, by \eqref{eqn-kempf-ness} and the definition of a Cartan subspace
(see \cite[p. 504]{DK}), we have:
\beq\label{eqn-orthogonal}
\Cartan \perp \Lg \cdot \Cartan
\;\; \text{with respect to} \;\;
\langle \; , \, \rangle,
\eeq
(cf. equation \eqref{eqn-stability}).

\begin{prop}\label{prop-core}
$\;$
\begin{enumerate}[topsep=-1.5ex]
\item[(i)]    The inclusion map $j^{}_V : V^{mrs} \hookrightarrow V^{rs}$ is a
                 homotopy equivalence.

\item[(ii)]   The map $j^{}_V$ induces an isomorphism:
\beqn
(j^{}_V)_* : \pi_1^K (V^{mrs}, c_0) \to \pi_1^G (V^{rs}, c_0).
\eeqn
\item[(iii)]  The inclusion $K \hookrightarrow G$ induces an isomorphism
$N_K (\Cartan) / Z_K (\Cartan)^0 \cong \widetilde W$.   
\end{enumerate}
\end{prop}

In fact, the set $V^{mrs}$ is a strong deformation retract of $V^{rs}$, but we will
not be using this fact.  Proposition \ref{prop-core} will be proved in Section
\ref{subsec-geom-prelim}.  Let $\nu = \nu_{\langle \; , \, \rangle} : V \to V^*$ be
the antilinear map given by:
\beq\label{eqn-nu}
\nu (v_1) : v_2 \mapsto \langle v_2, v_1 \rangle, \;\; v_1, v_2 \in V,
\eeq
and let:
\beq\label{eqn-l-zero}
l_0 = \nu (c_0) \in V^*.
\eeq
By \eqref{eqn-orthogonal}, we have $l_0 \in \Cartan^*$, and by Proposition
\ref{prop-core} (iii), we have:
\beq\label{eqn-dual-cartan-reg}
l_0 \in (\Cartan^*)^{reg} \subset (V^*)^{rs}.
\eeq
We use the inner product $\langle \; , \, \rangle$ to define subsets
$(V^*)^{mrs} \subset (V^*)^{ms} \subset V^*$ by analogy with
$V^{mrs} \subset V^{ms} \subset V$, and we note that assumption
\eqref{eqn-min-cartan} implies that $\Cartan^* \subset (V^*)^{ms}$.
By analogy with Proposition \ref{prop-core}, we have a homotopy
equivalence:
\beqn 
j^{}_{V^*} : (V^*)^{mrs} \to (V^*)^{rs},
\eeqn
and an isomorphism:
\beqn
(j^{}_{V^*})_* : \pi_1^K ((V^*)^{mrs}, l_0) \to \pi_1^G ((V^*)^{rs}, l_0).
\eeqn
Using equations \eqref{eqn-kempf-ness}-\eqref{eqn-core} (see also
\cite[Proposition 2.13 (i)]{Gr1}), one can check that the map $\nu$
restricts to a $K$-equivariant bijection:
\beq\label{eqn-nu-mrs}
\nu^{mrs} : V^{mrs} \to (V^*)^{mrs}.
\eeq
Thus, we obtain an isomorphism:
\beq\label{eqn-identify-pi-one}
\nu^{mrs}_* : \widetilde B_W = \pi_1^G (V^{rs}, c_0) \cong
\pi_1^K (V^{mrs}, c_0) \to \pi_1^K ((V^*)^{mrs}, l_0) \cong
\pi_1^G ((V^*)^{rs}, l_0).
\eeq
We will use this isomorphism to identify the fundamental group
$\pi_1^G ((V^*)^{rs}, l_0)$ with $\widetilde B_W$.

\begin{remark}\label{rmk-reversal}{\em
As a point of comparison, in \cite{GVX1} we used a $G$-invariant symmetric 
bilinear form on $V$ to relate the geometry of $V^*$ to that of $V$, but such a
form need not exist in the present context.  Note that, by the antilinear property
of the map $\nu : V \to V^*$, the isomorphism $\nu^{mrs}_*$ reverses the
direction of braid generators.  More precisely, let $\alpha \in A$ and let
$\tilde\sigma_\alpha \in \widetilde B_W$ be an element satisfying
$\tilde q (\tilde\sigma_\alpha) \in B_W [\alpha]$.  Write $\tilde q^* :
\pi_1^G ((V^*)^{rs}, l_0) \to \pi_1 ((\Cartan^*)^{reg}, l_0)$ for the analog of
the map $\tilde q$ of diagram \eqref{eqn-main-diagram}.  Then the image
$\tilde q^* \circ \nu^{mrs}_* \, (\tilde\sigma_\alpha)$ is the inverse of a
counter-clockwise braid generator (and can be called a clockwise braid
generator).}
\end{remark}

\subsection{The minimal polynomials \texorpdfstring{$R_{\chi, \alpha}$}{Lg}}
\label{subsec-min-poly}

The essential content of our main result is that the Fourier transform
$\FT P_\chi$ can be described in terms of data derived from the rank one
representations $\{ G_\alpha | V_\alpha \}_{\alpha \in A}$ introduced in Section
\ref{subsec-polar-rep} (see equation \eqref{eqn-rank-V-alpha}).  A key such
piece of data is a monic polynomial $R_{\chi, \alpha} \in \bC [z]$ of degree
$n_\alpha = |W_\alpha|$, associated to each $\alpha \in A$, which we now
proceed to define.

Fix an $\alpha \in A$ and a $\sigma_\alpha \in B_W [\alpha]$.  All of the
structures described above for the representation $G | V$ descend to
corresponding structures for the rank one representation $G_\alpha | V_\alpha$.
Note that assumptions \eqref{eqn-visibility} and \eqref{eqn-locality-assumption}
hold automatically for $G_\alpha | V_\alpha$.  The character $\chi$ restricts to
a character $\chi_\alpha$ of the subgroup $I_\alpha \subset I$ (see
\eqref{eqn-I-alpha} and the discussion following \eqref{eqn-seq-alpha}).  Let:
\beqn
f_\alpha : V_\alpha \to Q_\alpha \coloneqq \Cartan / W_\alpha =
V_\alpha \inv G_\alpha \, ,
\eeqn
be the quotient map, and let:
\beq\label{eqn-X-Q-alpha}
X_{0, \alpha} = f_\alpha ^{-1} (0), \;\;
Q_\alpha^{reg} = (\Cartan - \Cartan_\alpha) / W_\alpha \, , \;\;
\breve c_0 = f_\alpha (c_0) \in Q_\alpha^{reg}, \;\;
X_{\breve c_0, \alpha} = f_\alpha^{-1} (\breve c_0).
\eeq
As in Section \ref{subsec-char-chi}, we have a rank one $G_\alpha$-equivariant
local system $\cL_{\chi^{}_\alpha}$ on $X_{\breve c_0, \alpha}$, with
$(\cL_{\chi^{}_\alpha})_{c_0} = \bC$, and a nearby cycle sheaf 
$P_{\chi^{}_\alpha} \in \on{Perv}_{G_\alpha} (X_{0, \alpha})$.

\pagebreak

Next, we define:
\beqn
V^{rs}_\alpha \coloneqq f_\alpha^{-1} (Q_\alpha^{reg}) = G_\alpha \cdot
(\Cartan - \Cartan_\alpha),
\eeqn
\beqn
B_{W_\alpha} \coloneqq \pi_1 (Q_\alpha^{reg}, \breve c_0) \cong \bZ,  \;\;
\widetilde B_{W_\alpha} \coloneqq \pi_1^{G_\alpha} (V_\alpha^{rs}, c_0).
\eeqn
Diagram \eqref{eqn-main-diagram} has a direct counterpart for the
representation $G_\alpha | V_\alpha$:
\beq
\label{eqn-main-diagram-alpha}
\begin{CD}
1 @>>> I_\alpha @>>> \widetilde B_{W_\alpha} @>{\tilde q_\alpha}>>
B_{W_\alpha} @>>> 1 \;\,
\\
@. @| @VV{\tilde p_\alpha}V @VV{p_\alpha}V @.
\\
1 @>>> I_\alpha @>>> \widetilde W_\alpha @>{q_\alpha}>>
W_\alpha @>>> 1 \, ,
\end{CD}
\eeq
where $p_\alpha$, $q_\alpha$, $\tilde p_\alpha$, $\tilde q_\alpha$ are the
analogs of the maps $p$, $q$, $\tilde p$, $\tilde q$.  Just as for
\eqref{eqn-main-diagram}, the right square of \eqref{eqn-main-diagram-alpha}
is Cartesian (cf. equation \eqref{eqn-concrete}).  Note that the bottom row of
\eqref{eqn-main-diagram-alpha} naturally injects into the bottom row of
\eqref{eqn-main-diagram}, but the top row does not.

Let:
\beqn
\sigma \in B_{W_\alpha} \cong \bZ,
\eeqn
be the counter-clockwise generator.  Note that we have:
\beq\label{eqn-p-alpha-sigma}
p_\alpha (\sigma) = s_\alpha^{-1} \in W_\alpha \, ,
\eeq
(see equation \eqref{eqn-p-sigma-alpha}).
The regular splitting $\tilde r : B_W \to \widetilde B_W$, which was fixed in
\eqref{eqn-tilde-r}, defines a splitting homomorphism:
\beq\label{eqn-tilde-r-sigma-alpha}
\tilde r [\sigma_\alpha] : B_{W_\alpha} \to \widetilde B_{W_\alpha} \, ,
\eeq
such that the composition $r [\sigma_\alpha] \coloneqq \tilde p_\alpha \circ
\tilde r [\sigma_\alpha]$ is given by:
\beq\label{eqn-r-sigma-alpha}
r [\sigma_\alpha] \, (\sigma) = r (\sigma_\alpha) \in \widetilde W_\alpha
\subset \widetilde W.
\eeq
Note that the splitting homomorphism $\tilde r [\sigma_\alpha]$ depends
on the choice of the braid generator $\sigma_\alpha \in B_W [\alpha]$.
Note also that $\tilde r [\sigma_\alpha]$ is automatically a regular splitting
in the sense of Definition \ref{defn-reg-split}.  Thus, we can say that,
given a choice of $\sigma_\alpha \in B_W [\alpha]$, all of the assumptions
\eqref{eqn-visibility}, \eqref{eqn-locality-assumption}, \eqref{eqn-tilde-r}
hold for the representation $G_\alpha | V_\alpha$.

Next, we consider the dual representation $G_\alpha | V^*_\alpha$.
Let $\Cartan^*_\alpha = \nu (\Cartan_\alpha) = (\Cartan^*)^{W_\alpha}
\subset \Cartan^*$ (see equation \eqref{eqn-nu} and Proposition
\ref{prop-core} (iii)), let $(V^*_\alpha)^{rs} =
G_\alpha \cdot (\Cartan^* - \Cartan^*_\alpha)$, let:
\beqn
l_{0, \alpha} = l_0 |_{V_\alpha} \in (V_\alpha^*)^{rs},
\eeqn
(see equation \eqref{eqn-l-zero}), and let $K_\alpha = K \cap G_\alpha$.
Using assumption \eqref{eqn-min-cartan} and arguing as in the proof of
\cite[Proposition~1.3]{DK}, one can show that:
\beq\label{eqn-K-alpha}
\text{$K_\alpha$ is a compact form of $G_\alpha$.}
\eeq
As in Section \ref{subsec-dual-rep} (see equation
\eqref{eqn-identify-pi-one}), the compact form $K_\alpha \subset G_\alpha$
and the inner product $\langle \; , \, \rangle$ give rise to an identification:
\beq\label{eqn-identify-rank-one}
\pi_1^{G_\alpha} ((V_\alpha^*)^{rs}, l_{0, \alpha}) \cong
\widetilde B_{W_\alpha} \, .
\eeq
Applying Proposition \ref{prop-A-f-stratification} to the representation
$G_\alpha | V_\alpha$, we obtain a Morse local system $M (P_{\chi^{}_\alpha})$
on $(V_\alpha^*)^{rs}$, of the sheaf $P_{\chi^{}_\alpha}$ at the origin.
In other words, for each $l_\alpha \in (V_\alpha^*)^{rs}$, we have:
\beq\label{eqn-M-l-alpha}
M_{l_\alpha} (P_{\chi^{}_\alpha}) = H^0((\phi_{-l_\alpha} (P_{\chi^{}_\alpha}))^{}_0),
\eeq
where the LHS is the stalk of $M (P_{\chi^{}_\alpha})$ at $l_\alpha$, and the RHS
is the stalk cohomology at the origin of the vanishing cycles of $-l_\alpha$ applied
to $P_{\chi^{}_\alpha}$.  We use isomorphism \eqref{eqn-identify-rank-one} to write:
\beq\label{eqn-lambda-grp-rank-one}
\lambda_{l_0, \alpha} : \widetilde B_{W_\alpha} \to
\on{Aut} (M_{l_{0, \alpha}} (P_{\chi^{}_\alpha})),
\eeq
for the holonomy of this local system, and we refer to $\lambda_{l_0, \alpha}$
as the microlocal monodromy for the sheaf $P_{\chi^{}_\alpha}$.

Let:
\beq\label{eqn-space-R}
\cR = \{ R \in \bC [z] \; | \; \deg R \geq 1, \, R (0) \neq 0, \,
\text{$R$ is monic} \}.
\eeq
We think of $\cR$ as the space of all possible minimal polynomials for an 
invertible element $a \in \cA$ of an associative $\bC$-algebra $\cA$ with unit.
Let:
\beq\label{eqn-W-chi}
W_\chi \coloneqq \on{Stab}_W (\chi),
\eeq
(see equation \eqref{eqn-factor-action}).

\begin{prop-def}\label{prop-def-R-chi-alpha}
For each $\alpha \in A$, we define $R_{\chi, \alpha} \in \cR$ to be the minimal
polynomial of the holonomy operator:
\beq\label{eqn-R-chi-alpha}
\lambda_{l_0, \alpha} \circ \, \tilde r [\sigma_\alpha] \, (\sigma) \in
\on{End} (M_{l_{0, \alpha}} (P_{\chi^{}_\alpha})),
\eeq
for some $\sigma_\alpha \in B_W [\alpha]$.  We claim that:
\begin{enumerate}[topsep=-1.5ex]
\item[(i)]    The polynomial $R_{\chi, \alpha}$ is independent of the choice of
$\sigma_\alpha \in B_W [\alpha]$.

\item[(ii)]   For every $\alpha_1, \alpha_2 \in A$ and $w \in W_\chi$, with
$\alpha_2 = w \, \alpha_1$, we have $R_{\chi, \alpha_1} = 
R_{\chi, \alpha_2} \,$.
\end{enumerate}
\end{prop-def}

A proof of Proposition-Definition \ref{prop-def-R-chi-alpha} will be given in
Section \ref{subsec-full}.

\begin{prop}\label{prop-monic}
For each $\alpha \in A$, we have:
\beqn
\deg R_{\chi, \alpha} = \dim M_{l_{0, \alpha}} (P_{\chi^{}_\alpha}) =
n_\alpha \, .
\eeqn
\end{prop}

A proof of Proposition \ref{prop-monic} will be given in Section \ref{sec-carousel}.

\subsection{The determinant character \texorpdfstring{$\tau \in \hat I$}{Lg}}

Let $\on{det} : G \to \bG_m$ be the determinant character of the
representation $G | V$.

\begin{prop}\label{prop-char-tau}
For every $g \in Z_G (\Cartan)$, we have $\on{det} (g) = \pm 1$.
\end{prop}

Proposition \ref{prop-char-tau} will be proved in Section \ref{subsec-geom-prelim}.
It implies that $\on{det} : G \to \bG_m$ descends to to a character:
\beq\label{eqn-char-tau-I}
\tau : I \to \{ \pm 1 \},
\eeq
(see equation \eqref{eqn-I}).  Moreover, the character $\tau \in \hat I$ is fixed by
the action of $W$ on $\hat I$.  Therefore, we can use the splitting homomorphism
$\tilde r : B_W \to \widetilde B_W$ (see assumption \eqref{eqn-tilde-r}) to extend
$\tau$ to a character:
\beq\label{eqn-char-tau}
\tau : \widetilde B_W \to \{ \pm 1 \},
\eeq
such that $\tau \circ \tilde r \, (b) = 1$ for every $b \in B_W$.

\begin{example}\label{ex-tau-not-one}
{\em
The following example illustrates that the character $\tau$ of equation
\eqref{eqn-char-tau-I} can be non-trivial.
Take $V = \on{Hom} (\bC^2, \bC) \oplus \on{Hom} (\bC, \bC^3)$, and let
$G = SO (2) \times GL (1) \times SO (3)$ act on $V$ by:
\beqn
(g_1, a, g_2) . (x_1, x_2) = (a \, x_1 \, {g_1}^{-1}, g_2 \, x_2 \, a^{-1}).
\eeqn
The representation $G | V$ is polar and stable, of rank one.  For this
representation, we have: $W \cong \bZ / 4$, $I \cong \bZ / 2$,
$\widetilde W \cong W \times I$, and $\tau : I \to \{ \pm 1 \}$ is the non-trivial
character.  Note that  the extension $\tau : \widetilde B_W \to \{ \pm 1 \}$
of equation \eqref{eqn-char-tau}, in this example, will depend on the choice
of the splitting homomorphism $\tilde r$.}
\end{example}

\begin{remark}
{\em
The character $\tau$ of equation \eqref{eqn-char-tau-I} is analogous to the
character $\tau : I \to \{ \pm 1 \}$ of \cite[Section 3.4]{GVX1}.  However, in
the situation (and the notation) of \cite{GVX1}, the character $\tau$ is only
non-trivial because we consider the action of the full fixed point group
$K = G^\theta$, rather than the identity component $K^0 \subset K$
(cf. Remark \ref{rmk-GVX}).}
\end{remark}

For each $\alpha \in A$, we can repeat the construction of the character
$\tau \in \hat I$ for the representation $G_\alpha | V_\alpha$, to obtain
a character:
\beqn
\tau_\alpha : I_\alpha \to \{ \pm 1\}.
\eeqn
Recall that we have $I_\alpha \subset I$, as discussed in Section
\ref{subsec-reg-split}, following \eqref{eqn-seq-alpha}.

\begin{prop}\label{prop-restrict-tau}
For each $\alpha \in A$, we have $\tau_\alpha = \tau |_{I_\alpha}$.
\end{prop}

A proof of Proposition \ref{prop-restrict-tau} will be given in Section
\ref{subsec-hessians}.

The character $\tau_\alpha \in \hat I_\alpha$ turns out to be related to the
value of the minimal polynomial $R_{\chi, \alpha}$ at zero.  Let:
\beq\label{eqn-d-alpha}
d_\alpha = \dim X_{\breve c_0, \alpha},
\eeq
(see \eqref{eqn-X-Q-alpha}).
Recall that we write $n_\alpha = |W_\alpha|$ (see \eqref{eqn-n-alpha}),
and recall the map $r = \tilde p \circ \tilde r : B_W \to \widetilde W$ defined
following assumption \eqref{eqn-tilde-r}.

\begin{prop}\label{prop-R-chi-alpha-at-zero}
Let $\alpha \in A$, $\sigma_\alpha \in B_W [\alpha]$.  Let
$x = r (\sigma_\alpha^{n_\alpha}) \in I_\alpha \subset I \subset \widetilde W$
(see diagram \eqref{eqn-main-diagram} and Defintion \ref{defn-reg-split}).
We then have:
\beqn
R_{\chi, \alpha} (0) = (-1)^{d_\alpha+1} \cdot \chi (x) \cdot \tau_\alpha (x).
\eeqn
\end{prop}

A proof of Proposition \ref{prop-R-chi-alpha-at-zero} will be given in Section
\ref{sec-carousel}.

\subsection{The group \texorpdfstring{$W_\chi^0$}{Lg} and the Hecke
algebra \texorpdfstring{$\cH_{W_\chi^0}$}{Lg}}
\label{subsec-W-zero}

For each $\alpha \in A$, let:
\beq\label{eqn-W-alpha-chi}
W_{\alpha, \chi} = W_\alpha \cap W_\chi
\;\; \text{and} \;\;
\degalpha = n_\alpha / |W_{\alpha, \chi}| \in \bZ,
\;\; \text{so that} \;\;
W_{\alpha, \chi} = \langle s_\alpha^\degalpha \rangle.
\eeq

\begin{prop}\label{prop-R-alpha-vanishing}
For each $\alpha \in A$, there exists a polynomial $\bar R_{\chi, \alpha} \in \cR$,
with $\deg \bar R_{\chi, \alpha} = n_\alpha / \degalpha$, such that:
\beqn
R_{\chi, \alpha} (z) = \bar R_{\chi, \alpha} (z^{\degalpha}).
\eeqn
\end{prop}

A proof of Proposition \ref{prop-R-alpha-vanishing} will be given in Section
\ref{sec-carousel}.  Clearly, the polynomial $\bar R_{\chi, \alpha} \in \cR$ is
unique; we can think of it as the minimal polynomial of the holonomy operator:
\beq\label{eqn-interpret-bar-R-chi-alpha}
\lambda_{l_0, \alpha} \circ \, \tilde r [\sigma_\alpha] \, (\sigma^\degalpha) \in
\on{End} (M_{l_{0, \alpha}} (P_{\chi^{}_\alpha})),
\eeq
for some $\sigma_\alpha \in B_W [\alpha]$.  Note that:
\beq\label{eqn-assignment-degalpha}
\text{the assignment $\alpha \mapsto \degalpha$ is invariant under the action
of $W_\chi$ on $A$.}
\eeq
Therefore, by Proposition-Definition \ref{prop-def-R-chi-alpha} (ii), we have:
\beq\label{eqn-invariance-bar}
\text{for every $\alpha_1, \alpha_2 \in A$ and $w \in W_\chi$, with
$\alpha_2 = w \, \alpha_1$, we have $\bar R_{\chi, \alpha_1} =
\bar R_{\chi, \alpha_2} \,$.}
\eeq

Let $W_\chi^0 \subset W_\chi$ be the subgroup generated by all the
$\{ W_{\alpha, \chi} \}_{\alpha \in A}$.  Note that $W_\chi^0$ is a complex
reflection group acting on $\Cartan$, and that we have:
\beq\label{eqn-W-zero-alpha}
W_\chi^0 \cap W_\alpha = W_{\alpha, \chi} \;\; \text{for every} \;\;
\alpha \in A.
\eeq

\begin{remark}\label{rmk-W-alpha-chi-one}{\em
The assignment $(G | V, \chi) \mapsto \{ W_\chi^0 \subset W_\chi \subset W \}$
can be applied to the pair $(G_\alpha | V_\alpha, \chi_\alpha)$, to obtain
subgroups $W_{\alpha, \chi^{}_\alpha}^0 \subset W_{\alpha, \chi^{}_\alpha}
\subset W_\alpha$.  By construction, we then have:
\beqn
W_{\alpha, \chi} \subset W_{\alpha, \chi^{}_\alpha}^0 =
W_{\alpha, \chi^{}_\alpha} \, .
\eeqn}
\end{remark}

Let:
\beq\label{eqn-A-zero-one}
A_\chi^0 = \{ \alpha \in A \; | \; \degalpha < n_\alpha \}, \;\;\;\;
A_\chi^1 = \{ \alpha \in A \; | \; \degalpha = n_\alpha \},
\eeq
so that $A = A_\chi^0 \cup A_\chi^1$.
Let:
\beqn
\Cartan_\chi^{reg} = \Cartan - \bigcup_{\alpha \in A_\chi^0} \Cartan_\alpha \, ,
\eeqn
and define:
\beqn
B_{W_\chi^0} \coloneqq \pi_1 (\Cartan_\chi^{reg} / W_\chi^0, c_0), \;\;\;\;
B_W^{\chi, 0} \coloneqq \pi_1 (\Cartan^{reg} / W_\chi^0, c_0) =
p^{-1} (W_\chi^0) \subset B_W \, ,
\eeqn
where the basepoint $c_0 \in \Cartan^{reg}$ naturally determines basepoints
in $\Cartan_\chi^{reg} / W_\chi^0$ and $\Cartan^{reg} / W_\chi^0$.
Note that $B_{W_\chi^0}$ is the braid group associated to the complex
reflection group $W_\chi^0$.

Let:
\beq\label{eqn-varphi}
\varphi : B_W^{\chi, 0} \to B_{W_\chi^0} \, ,
\eeq
be the map induced by the inclusion $\Cartan^{reg} \to \Cartan_\chi^{reg}$.
Note that, for every $\alpha \in A_\chi^0$ and $\sigma_\alpha \in B_W [\alpha]$,
we have $\sigma_\alpha^\degalpha \in B_W^{\chi, 0}$.  Moreover, by
\eqref{eqn-W-zero-alpha}, the image $\varphi (\sigma_\alpha^{\degalpha}) \in
B_{W_\chi^0}$ is a braid generator for the group $B_{W_\chi^0}$.  We define a
Hecke algebra $\cH_{W_\chi^0}$ as the quotient of the group algebra
$\bC [B_{W_\chi^0}]$ by all relations of the form:
\beqn
\bar R_{\chi, \alpha} (\varphi (\sigma_\alpha^{\degalpha})) = 0,
\eeqn
for $\alpha \in A_\chi^0$ and $\sigma_\alpha \in B_W [\alpha]$, and we write:
\beq\label{eqn-eta-chi}
\eta_\chi : \bC [B_{W_\chi^0}] \to \cH_{W_\chi^0} \, ,
\eeq
for the quotient map.  By assertion \eqref{eqn-invariance-bar} and
\cite[Theorem 1.3]{Et}, we have:
\beq\label{eqn-dim-H}
\dim \cH_{W_\chi^0} = | W_\chi^0 |.
\eeq

The Hecke algebra $\cH_{W_\chi^0}$ encodes the polynomials
$\bar R_{\chi, \alpha}$ for $\alpha \in A_\chi^0$.  To encode the
polynomials $\bar R_{\chi, \alpha}$ for $\alpha \in A_\chi^1$, we define
a character:
\beq\label{eqn-rho}
\rho = \rho_\chi : B_W^{\chi, 0} \to \bG_m \, ,
\eeq
by requiring that:
\begin{align}
\rho (\sigma_\alpha^\degalpha) & = 1 \;\;
\text{for all} \;\; \alpha \in A_\chi^0 \;\; \text{and} \;\;
\sigma_\alpha \in B_W [\alpha], \notag \\
\bar R_{\chi, \alpha} (\rho (\sigma_\alpha^{n_\alpha})) & = 0 \;\;
\text{for all} \;\; \alpha \in A_\chi^1 \;\; \text{and} \;\;
\sigma_\alpha \in B_W [\alpha].
\label{eqn-define-rho}
\end{align}
Note that, for $\alpha \in A_\chi^1$, we have $\degalpha = n_\alpha$ and 
$\deg \bar R_{\chi, \alpha} = 1$.  The character $\rho$ is well-defined by
assertion \eqref{eqn-invariance-bar}.  Moreover, by Proposition
\ref{prop-R-chi-alpha-at-zero}, we can rewrite condition \eqref{eqn-define-rho}
as follows:
\beq\label{eqn-rho-explicit}
\rho (\sigma_\alpha^{n_\alpha}) =
(-1)^{d_\alpha} \cdot \chi (x) \cdot \tau_\alpha (x),
\eeq
for all $\alpha \in A_\chi^1$ and $\sigma_\alpha \in B_W [\alpha]$, and for
$x = r (\sigma_\alpha^{n_\alpha}) \in I_\alpha$.

\begin{remark}\label{rmk-rho}{\em
In all of the examples computed by the authors, the character $\rho$ is trivial.
But we can not rule out the possibility that it is non-trivial in some examples.}
\end{remark}

\subsection{Statement of the theorem}
\label{subsec-statement}

We begin by constructing a $\bC [\widetilde B_W]$-module $\cM_\chi$
associated to the character $\chi \in \hat I$.  Let:
\beq\label{eqn-tilde-B-chi}
B_W^\chi \coloneqq \on{Stab}_{B_W} (\chi) = p^{-1} (W_\chi), \;\;
\widetilde B_W^\chi \coloneqq \tilde q^{-1} (B_W^\chi), \;\;
\widetilde B_W^{\chi, 0} \coloneqq \tilde q^{-1} (B_W^{\chi, 0}),
\eeq
(see equation \eqref{eqn-factor-action}).  Note that we have
$\widetilde B_W^{\chi, 0} \subset \widetilde B_W^\chi \subset \widetilde B_W$.
The character $\chi$ extends uniquely to a character:
\beq\label{eqn-hat-chi}
\hat \chi : \widetilde B_W^\chi \to \bG_m \, ,
\eeq
such that $\hat \chi \circ \tilde r \, (b) = 1$ for every $b \in B_W^\chi$.
Let $\bC_\chi$ be a copy of $\bC$, viewed as an $I$-module via the character
$\chi$.  Simultaneously, we view $\bC_\chi$ as a $\widetilde B_W^\chi$-module
(and therefore a $\widetilde B_W^{\chi, 0}$-module) via the character $\hat \chi$
of equation \eqref{eqn-hat-chi}.  Similarly, let $\bC_\tau$ be a copy of $\bC$,
viewed as a $\widetilde B_W$-module via the character $\tau$ of equation
\eqref{eqn-char-tau}.  Also, let $\bC_\rho$ be a copy of $\bC$, viewed as a
$B_W^{\chi, 0}$-module via the character $\rho$ of equation \eqref{eqn-rho}.
Simultaneously, we view $\bC_\rho$ as a $\widetilde B_W^{\chi, 0}$-module
via the composition:
\beq\label{eqn-tilde-rho}
\tilde \rho \coloneqq \rho \circ \tilde q : \widetilde B_W^{\chi, 0} \to \bG_m \, .
\eeq

Recall the group homomorphism $\varphi : B_W^{\chi, 0} \to B_{W_\chi^0}$
of equation \eqref{eqn-varphi}, and let:
\beqn
\widetilde\varphi \coloneqq \varphi \circ \tilde q :
\widetilde B_W^{\chi, 0} \to B_{W_\chi^0} \, .
\eeqn
We can view the Hecke algebra $\cH_{W_\chi^0}$ as a module over
itself via the left multiplication.  Simultaneously, we can view
$\cH_{W_\chi^0}$ as $\bC [B_{W_\chi^0}]$-module via the map
$\eta_\chi$ of equation \eqref{eqn-eta-chi}, as a 
$\bC [B_W^{\chi, 0}]$-module via the composition:
\beqn
\eta_\chi \circ \varphi : \bC [B_W^{\chi, 0}] \to
\cH_{W_\chi^0} \, ,
\eeqn
and as a $\bC [\widetilde B_W^{\chi, 0}]$-module via the composition:
\beqn
\eta_\chi \circ \widetilde\varphi : \bC [\widetilde B_W^{\chi, 0}] \to
\cH_{W_\chi^0} \, .
\eeqn
Here, and in the rest of the paper, we use the following notational
convention.

\begin{convention}\label{conv-group-alg}
{\em
Given a map of discrete groups $\Phi : B_1 \to B_2$, we use the same
symbol for the corresponding map of group algebras $\Phi : \bC [B_1] \to
\bC [B_2]$.  Similarly, given a representation of a discrete group
$\Lambda : B \to \on{Aut} (M)$ on a complex vector space $M$, we use
the same symbol for the corresponding representation of the group
algebra $\Lambda : \bC [B] \to \on{End} (M)$.}
\end{convention}

Consider the tensor product $\bC_\chi \otimes \bC_\rho \otimes
\cH_{W_\chi^0}$, taken over $\bC$, and view it as a
$\bC [\widetilde B_W^{\chi, 0}]$-module by combining the
$\bC [\widetilde B_W^{\chi, 0}]$-module structures on
$\bC_\chi$, $\bC_\rho$, and $\cH_{W_\chi^0}$.
We now induce this module to $\bC [\widetilde B_W]$, and define:
\beq\label{eqn-M-chi}
\cM_\chi \ = \ \left( \bC [\widetilde B_W]
\otimes_{\bC [\widetilde B_W^{\chi, 0}]}
(\bC_\chi \otimes \bC_\rho \otimes \cH_{W_\chi^0}) \right)
\otimes \bC_\tau \, .
\eeq
Note that, by equation \eqref{eqn-dim-H}, we have $\dim \cM_\chi = |W|$.  We
interpret the $\bC [\widetilde B_W]$-module $\cM_\chi$ as a $G$-equivariant
local system on $(V^*)^{rs}$, whose fiber over the basepoint $l_0 \in (V^*)^{rs}$
is equal to $\cM_\chi$, and whose holonomy is given by the
$\bC [\widetilde B_W]$-module structure, via the identification
\eqref{eqn-identify-pi-one}.

\begin{thm}
\label{thm-main}
Let $G | V$ be as above: a stable polar representation, which is visible or of
rank one, satisfies the locality assumption \eqref{eqn-locality-assumption}, and
admits a regular splitting as in \eqref{eqn-tilde-r}.  Consider the nearby cycle
sheaf $P_\chi$ defined in \eqref{eqn-P-chi}.  Its Fourier transform is given by:
\beqn
\FT P_\chi \cong \on{IC} ((V^*)^{rs}, \cM_\chi),
\eeqn
where the RHS is the IC-extension of the local system $\cM_\chi$ defined
in \eqref{eqn-M-chi} above.
\end{thm}

\begin{remark}\label{rmk-GVX}{\em
The statement of Theorem \ref{thm-main} is formally very similar to the
statement of \cite[Theorem 3.6]{GVX1}, the main difference being that the
former applies much more broadly.  However, \cite[Theorem 3.6]{GVX1} is not
technically a special case of Theorem \ref{thm-main}.  There are four reasons
for this.  First, the paper \cite{GVX1} works with coefficients in a general integral
domain $\Bbbk$, while in this paper we work with coefficients in $\bC$.  This
distinction is not material, given \cite[Remark 7.2]{GVX1} and the freeness of
the $\Bbbk$-module $M_l (P_\chi)$ in \cite[Proposition~7.3]{GVX1}.  Second,
\cite[Theorem 3.6]{GVX1} computes the polynomials $\bar R_{\chi, \alpha}$
explicitly.  Third, the paper \cite{GVX1} deals with a possibly disconnected group
$K$ acting on a symmetric space $\Lp$.  And fourth, the group $W_{\La, \chi}^0$
defined in \cite{GVX1} is potentially smaller than the group $W_\chi^0$ defined for
the corresponding situation in this paper, even if $K$ is connected.  Having said
that, it is not difficult to verify that the claim of \cite[Theorem 3.6]{GVX1} for a
symmetric pair $(G, K)$, giving rise to a symmetric space $\Lp$, is equivalent to
the claim of Theorem \ref{thm-main} for the polar representation $K^0 | \Lp$
(cf. Remark \ref{rmk-E-alpha} below).}
\end{remark}

\begin{remark}\label{rmk-Gr1}{\em
The claim of Theorem \ref{thm-main} for a polar representation $G | V$ and 
the trivial character $\chi = 1$ is close, but not equivalent, to the claims of
\cite[Theorems 3.1 $\&$ 5.2]{Gr1} for $G | V$.  The latter provide a
simultaneous description of the Morse local system $M (P_1)$ on $(V^*)^{rs}$,
of the sheaf $P_1$ at the origin, and of the monodromy in the family action of
$B_W$ on $M (P_1)$.  However, the description of $M (P_1)$ in \cite{Gr1} is
not complete.  Theorem \ref{thm-main}, on the other hand, gives a complete
description of $M (P_1)$, but does not discuss the monodromy in the family;
see however Remark \ref{rmk-end-ring} below.}
\end{remark}

\begin{remark}\label{rmk-E-alpha}{\em
In some situations, the factorization of the polynomial $R_{\chi, \alpha}$ in
Proposition \ref{prop-R-alpha-vanishing} holds with $z^\degalpha$ replaced by
$z^{E_\alpha}$, where $E_\alpha \geq \degalpha$ for all $\alpha \in A$,
$E_\alpha > \degalpha$ for some $\alpha \in A$, and the assignment $\alpha
\mapsto E_\alpha$ is invariant under $W_\chi$.  When this is the case, we can
define $W_\chi^1 \subset W_\chi^0$ to be the proper subgroup generated by
the powers $\{ s_\alpha^{E_\alpha} \}_{\alpha \in A}$, and the group $W_\chi^1$
can be used in place of $W_\chi^0$ in the statement of Theorem \ref{thm-main}.
In the context of GIT stably graded Lie algebras, a natural candidate
$W_\chi^{en}$ for $W_\chi^1$ arises from the endoscopic point of view,
as explained in \cite[Section 5]{VX2}.}
\end{remark}

\begin{remark}\label{rmk-end-ring}{\em
A natural question that is not answered by Theorem \ref{thm-main} is the
computation of the endomorphism ring $\on{End} (P_\chi)$.  We expect
that this ring can be described as follows.  Write $\cH_\chi^\circ$ for the
opposite of the Hecke algebra $\cH_{W_\chi^0}$.  There is a uique
homomorphism $\kappa : \bC [B_W^{\chi, 0}] \to \cH_\chi^\circ$, such that:
\begin{align*}
\kappa (\sigma_\alpha^\degalpha) & = k_\alpha \cdot
\eta_\chi \circ \varphi \, (\sigma_\alpha^{-\degalpha})
\;\; \text{for all} \;\; \alpha \in A_\chi^0 \;\; \text{and} \;\;
\sigma_\alpha \in B_W [\alpha], \\
\kappa (\sigma_\alpha^{n_\alpha}) & = (-1)^{d_\alpha} \cdot
\chi (\sigma_\alpha^{-n_\alpha})
\;\; \text{for all} \;\; \alpha \in A_\chi^1 \;\; \text{and} \;\;
\sigma_\alpha \in B_W [\alpha],
\end{align*}
where $\{ k_\alpha \in \{ \pm 1 \} \}_{\alpha \in A}$ are the signs given by
Proposition \ref{prop-carousel}.  The definition of $\kappa$ is designed
to capture the action of $B_W^{\chi, 0} \subset B_W^\chi$ on $\cM_\chi^0
\subset \cM_\chi$ (see equation \eqref{eqn-M-chi-induced}), via the
monodromy in the family \eqref{eqn-introduce-monodromy} (cf. Proposition
\ref{prop-microlocal-generator} and equation \eqref{eqn-rho-mu-explicit}).
The existence and uniqueness of $\kappa$ follow from the proof of
Theorem \ref{thm-main} in Section \ref{sec-proof}.  Use $\kappa$ to view
$\cH_\chi^\circ$ as a $\bC [B_W^{\chi, 0}]$-module, and define:
\beq\label{eqn-E}
E \coloneqq \bC [B_W^\chi] \otimes_{\bC [B_W^{\chi, 0}]} \cH_\chi^\circ \, .
\eeq
Then the $\bC [B_W^\chi]$-module structure on $E$ descends to
a $\bC$-algebra structure, and the monodromy action
\eqref{eqn-introduce-monodromy} gives rise to an isomorphism:
\beq\label{eqn-End-E}
\on{End} (P_\chi) \cong E.
\eeq
In particular, the ring $\on{End} (P_\chi)$ is generated by the monodromy
action \eqref{eqn-introduce-monodromy}, and we have
$\dim \on{End} (P_\chi) = |W_\chi|$.  We expect that isomorphism
\eqref{eqn-End-E} can be established using the methods of this paper,
but we have not carried out a proof.  We note, however, that in the special
case where $W_\chi^0 = W_\chi$, definition \eqref{eqn-E} reduces to
$E = \cH_\chi^\circ$, and isomorphism \eqref{eqn-End-E} follows readily
from the statement and the proof of Theorem \ref{thm-main}.}
\end{remark}

\section{Preliminary results}
\label{sec-prelim}

\subsection{Geometric preliminaries}
\label{subsec-geom-prelim}

In this subsection, we prove Propositions \ref{prop-A-f-stratification},
\ref{prop-core}, and \ref{prop-char-tau}.

\begin{proof}[Proof of Proposition \ref{prop-A-f-stratification}]
A proof of part (i) is essentially contained in the first paragraph of
\cite[Section 3]{Gr1}.  To wit, if $G | V$ is visible, we can take $\cS_0$ to be
the $G$-orbit stratification.  If, instead (see assumption \eqref{eqn-visibility}),
we have $\on{rank} (G | V) = 1$, we can readily reduce to the case $\dim Q = 1$,
then use a general result on the existence of $A_f$ stratifications for functions;
see \cite[p. 248, Corollary 1]{Hi}.

Our proof of part (ii) is similar to the proof of \cite[Lemma 6.1]{GVX1}.
Since the stratification $\cS_0$ is $\bC^*$-conic, it suffices to show that
$0 \in X_0$ is the only stratified critical point of the restriction
$l |_{X_0}$, with respect to $\cS_0$.  By \cite[Proposition 2.13 (i)]{Gr1},
the representation $G | V^*$ is polar with Cartan subspace $\Cartan^* = 
(\Lg \cdot \Cartan)^\p \subset V^*$ (see \eqref{eqn-C-star}).  Therefore,
we can assume, without loss of generality, that $l \in (\Cartan^*)^{reg} =
\Cartan^* \cap (V^*)^{rs}$.  By \cite[Proposition 2.13 (ii)]{Gr1}, we have:
\beqn
\{ v \in V \; | \; l |_{\Lg \cdot v} = 0 \} = \Cartan.
\eeqn
Since the stratification $\cS_0$ is $G$-invariant, it remains to observe
that $X_0 \cap \Cartan = \{ 0 \}$.

Part (iii) follows from \cite[Theorem 5.5]{Gi} and the remark that follows
that theorem (cf. proof of \cite[Corollary 3.2]{GVX1}).
\end{proof}

\begin{proof}[Proof of Proposition \ref{prop-core}]
Let $\Lk = \on{Lie} (K)$.  By \cite[Proposition 1.3]{DK}, for every
$v \in \Cartan^{reg}$, we have:
\beq\label{eqn-stabilizer-core}
Z_G (v) = Z_G (\Cartan) = Z_K (\Cartan) \cdot \exp (\bi \, Z_\Lk (\Cartan)).
\eeq
Part (i) follows by considering $j^{}_V : V^{mrs} \to V^{rs}$ as a map of fiber
bundles over $Q^{reg}$.  Part (ii) follows from part (i) and the definition of the
equivariant $\pi_1$.  Part (iii) follows from equation \eqref{eqn-stabilizer-core}
and \cite[Lemma 2.7]{DK}.
\end{proof}

For every $v \in V$ and $l \in (\Lg \cdot v)^\p \subset V^*$, the restriction
$l |_{G \cdot v}$ has a critical point at $v$, and we write:
\beq\label{eqn-H-v-l}
\cH [v, l] \in Sym^2 ((\Lg \cdot v)^*),
\eeq
for the Hessian of $l |_{G \cdot v}$ at $v$.  By \cite[Corollary 2.16 (ii)]{Gr1}, if 
$v \in V^{rs}$ and $l \in (\Lg \cdot v)^\p \cap (V^*)^{rs}$, then $v$ is a Morse
critical point of $l |_{G \cdot v}$.  In other words, we have:
\beq\label{eqn-morse}
\forall \; v \in V^{rs}, \; l \in (\Lg \cdot v)^\p \cap (V^*)^{rs}, \;\; \cH [v, l]
\text{ is a non-degenerate quadratic form on } \Lg \cdot v.
\eeq

\begin{proof}[Proof of Proposition \ref{prop-char-tau}]
Each $g \in Z_G (\Cartan)$ preserves the direct sum decomposition
$V = \Cartan \oplus \Lg \cdot \Cartan$ of equation \eqref{eqn-stability}.
Pick a $v \in \Cartan^{reg}$ and an $l \in (\Cartan^*)^{reg}$.  By the
definition of a Cartan subspace (see \cite[p. 504]{DK}), we have
$\Lg \cdot \Cartan = \Lg \cdot v$.  The element $g$ acts on $\Lg \cdot v$
preserving the Hessian $\cH [v, l]$.  By \eqref{eqn-morse}, this implies
that $\on{det} (g |_{\Lg \cdot v}) = \pm 1$.  By assumption, we have
$\on{det} (g |_{\Cartan}) = 1$.  The proposition follows.
\end{proof}

\subsection{Hessians and the root space decomposition}
\label{subsec-hessians}

In Section \ref{subsec-geom-prelim}, we introduced the Hessian $\cH [v, l]$
for every $v \in V$ and $l \in (\Lg \cdot v)^\p \subset V^*$ (see equation
\eqref{eqn-H-v-l}).  In this subsection, we make a more detailed study of
$\cH [v, l]$ for $v \in \Cartan$ and $l \in \Cartan^*$, leading up to a proof
of Proposition \ref{prop-restrict-tau}.  We will need the analog of the root
space decomposition for the representation $G | V$, i.e., the following digest
of \cite[Theorem 2.12]{DK}.  Let $\Lm = Z_\Lg (\Cartan)$, and recall that we
write $\Lg_\alpha = Z_\Lg (\Cartan_\alpha)$, $\alpha \in A$.

\begin{thm}\label{thm-root-space-decomp}
$\;$
\begin{enumerate}[topsep=-1.5ex]
\item[(i)]
For all $\alpha \neq \beta$, $\alpha, \beta \in A$, we have:
\beqn
\Lg_\alpha \cap \Lg_\beta = \Lm.
\eeqn

\item[(ii)]
\beqn
\Lg = \sum_{\alpha \in A} \Lg_\alpha \;\;\;\; \text{and} \;\;\;\;
\Lg / \Lm = \bigoplus_{\alpha \in A} \Lg_\alpha / \Lm.
\eeqn

\item[(iii)]
\beqn
V = \Cartan \oplus \bigoplus_{\alpha \in A} \Lg_\alpha \cdot \Cartan.
\eeqn

\item[(iv)]
For every $v \in \Cartan$, we have:
\beqn
\Lg \cdot v = \bigoplus_{\alpha \in A \, : \, v \notin \Cartan_\alpha}
\Lg_\alpha \cdot \Cartan.
\eeqn

\item[(v)]
For every $v \in \Cartan$, we have:
\beqn
Z_\Lg (v) = \sum_{\alpha \in A \, : \, v \in \Cartan_\alpha} \Lg_\alpha \, .
\eeqn
\end{enumerate}
\end{thm}

\begin{proof}
The only assertion of the theorem not contained in \cite[Theorem 2.12]{DK}
is the second equation of part (ii).  It follows from the first equation of part (ii)
and the direct sum decomposition of part (iii).
\end{proof}

The following elementary lemma applies generally to representations of
algebraic groups over $\bC$.

\begin{lemma}\label{lemma-hess}
Let $G | V$ be a representation of an algebraic group over $\bC$, let $v \in V$,
let $l \in (\Lg \cdot v)^\p \subset V^*$, and let $\cH [v, l]$ be the Hessian of
$l |_{G \cdot v}$ at $v$.  Let $x \in \Lg$ and $w \in \Lg \cdot v \subset V$.
Then we have:
\beqn
\cH [v, l] \, (x \cdot v, w) = l (x \cdot w).
\eeqn 
\end{lemma}

\begin{proof}
Let $\Gamma_x : \bG_a \to G$ be the homomorphism of complex analytic
groups generated by $x \in \Lg$.  The assignment $t \mapsto (\Gamma_x (t)
\cdot v, \Gamma_x (t) \cdot w)$, $t \in \bG_a$, defines an analytic curve
$\bG_a \to TG \cdot v$.  We have:
\beqn
\left. \frac{\partial \, \Gamma_x (t) \cdot v}{\partial \, t} \right\rvert_{t = 0} =
x \cdot v
\;\;\;\; \text{and} \;\;\;\;
\left. \frac{\partial \, \Gamma_x (t) \cdot w}{\partial \, t} \right\rvert_{t = 0} =
x \cdot w.
\eeqn
The lemma follows by the definition of a Hessian.
\end{proof}

Recall the inner product $\langle \; , \, \rangle$ which was fixed in Section
\ref{subsec-dual-rep} and is subject to assumption \eqref{eqn-min-cartan}.

\begin{prop}\label{prop-decompose-hessian}
For each $v \in \Cartan$ and $l \in \Cartan^*$, the direct sum decomposition: 
\beqn
\Lg \cdot v =  \bigoplus_{\alpha \in A \, : \, v \notin \Cartan_\alpha}
\Lg_\alpha \cdot \Cartan, 
\eeqn
of Theorem \ref{thm-root-space-decomp} (iv) is orthogonal with respect to
both the Hessian $\cH [v, l]$ and the inner product $\langle \; , \, \rangle$.
\end{prop}

\begin{proof}
Begin with the assertion regarding the Hessian $\cH [v, l]$.  Pick a pair of
distinct elements $\alpha, \beta \in A$ with $v \notin \Cartan_\alpha \cup
\Cartan_\beta$.  Let $x_1 \in \Lg_\alpha$ and
$x_2 \in \Lg_\beta$.  We need to show that:
\beqn
\cH [v, l] \, (x_1 \cdot v, x_2 \cdot v) = 0.
\eeqn
Note that:
\beq\label{eqn-t-beta}
\text{for every} \;\; v_1 \in \Cartan - \Cartan_\beta, \;\; \text{we have} \;\; 
\Lg_\beta \cdot v_1 = \Lg_\beta \cdot v.
\eeq
Using \eqref{eqn-t-beta}, we can choose elements
$v_1 \in \Cartan_\alpha - \Cartan_\beta$ and $x_3 \in \Lg_\beta$, such
that $x_3 \cdot v_1 = x_2 \cdot v$.  We then have:
\beq\label{eqn-compute-Hess}
\begin{split}
\cH [v, l] \, (x_1 \cdot v, x_2 \cdot v) & =
\cH [v, l] \, (x_1 \cdot v, x_3 \cdot v_1) =
l (x_1 \cdot (x_3 \cdot v_1)) \\
& = l ([x_1, x_3] \cdot v_1) + l (x_3 \cdot (x_1 \cdot v_1)),
\end{split}
\eeq
where the second equality follows from Lemma \ref{lemma-hess}.  It remains
to note that each of the summands on the second line of
\eqref{eqn-compute-Hess} vanishes.  Indeed, we have:
\beqn
l ([x_1, x_3] \cdot v_1) = 0,
\eeqn
because $l \in \Cartan^* = (\Lg \cdot \Cartan)^\p \subset V^*$, and we have:
\beqn
l (x_3 \cdot (x_1 \cdot v_1)) = 0,
\eeqn
because $x_1 \in \Lg_\alpha$, $v_1 \in \Cartan_\alpha$, and so 
$x_1 \cdot v_1 = 0$.  This proves the orthogonality with respect to $\cH [v, l]$.

We now proceed to the assertion regarding the inner product
$\langle \; , \, \rangle$.  For each $\alpha \in A$, let $\Lk_\alpha = 
\Lg_\alpha \cap \Lk$.  By \cite[Proposition 1.3]{DK} and
Theorem \ref{thm-root-space-decomp} (v), we have:
\beq\label{eqn-Lg-alpha-Lk-alpha}
\Lg_\alpha = \Lk_\alpha \otimes_\bR \bC, \;\; \alpha \in A.
\eeq
Pick a pair of distinct elements $\alpha, \beta \in A$ with $v \notin
\Cartan_\alpha \cup \Cartan_\beta$.  Let $x_1 \in \Lk_\alpha$,
$x_2 \in \Lk_\beta$, and $v_1 \in \Cartan_\alpha - \Cartan_\beta$.  In view
of \eqref{eqn-t-beta} and \eqref{eqn-Lg-alpha-Lk-alpha}, it suffices to show that
$\langle x_1 \cdot v, x_2 \cdot v_1 \rangle = 0$.  Using the $K$-invariance of
the inner product $\langle \; , \, \rangle$, we compute:
\beqn
\langle x_1 \cdot v, x_2 \cdot v_1 \rangle =
- \langle v, x_1 \cdot (x_2 \cdot v_1) =
\langle v, x_2 \cdot (x_1 \cdot v_1) \rangle -
\langle v, [x_1, x_2] \cdot v_1 \rangle = 0.
\eeqn
Here, $x_1 \cdot v_1 = 0$ by construction, and $[x_1, x_2] \cdot v_1 \in
\Lg \cdot \Cartan$ is orthogonal to $v$, by \eqref{eqn-orthogonal}.
\end{proof}

For each $\alpha \in A$, $v \in \Cartan - \Cartan_\alpha$, and $l \in \Cartan^*$,
we note that $\Lg_\alpha \cdot v = \Lg_\alpha \cdot \Cartan$, and we write:
\beq\label{eqn-H-alpha}
\cH_\alpha [v, l] \coloneqq \cH [v, l] |_{\Lg_\alpha \cdot v} \in
Sym^2 ((\Lg_\alpha \cdot v)^*).
\eeq
Assertion \eqref{eqn-morse} and Propostion \ref{prop-decompose-hessian}
imply that the partial Hessian $\cH_\alpha [v, l]$ is a non-degenerate quadratic
form on $\Lg_\alpha \cdot v$, whenever  $v \in \Cartan^{reg}$ and
$l \in (\Cartan^*)^{reg}$.  The following proposition shows that, in fact, this
non-degeneracy of the partial Hessian holds more generally.  Recall that we
write $\Cartan^*_\alpha \subset \Cartan^*$ for the reflection hyperplane
corresponding to $\alpha$.

\begin{prop}\label{prop-morse-atom}
The partial Hessian $\cH_\alpha [v, l]$ is a non-degenerate quadratic form on
$\Lg_\alpha \cdot v$ for every $\alpha \in A$, $v \in \Cartan - \Cartan_\alpha$,
and $l \in \Cartan^* - \Cartan^*_\alpha$.
\end{prop}

\begin{proof}
Apply assertion \eqref{eqn-morse} to the rank one representation
$G_\alpha | V_\alpha$.
\end{proof}

\begin{proof}[Proof of Proposition \ref{prop-restrict-tau}]
Consider the $G_\alpha$-invariant direct sum decomposition:
\beq\label{eqn-V-alpha-notalpha}
V = V_\alpha \oplus \Lg \cdot \Cartan_\alpha \, ,
\eeq
where $V_\alpha = \Cartan \oplus \Lg_\alpha \cdot \Cartan$ and
$\Lg \cdot \Cartan_\alpha = \bigoplus_{\beta \in A_\notalpha}
\Lg_\beta \cdot \Cartan$ (see \eqref{eqn-V-alpha}, \eqref{eqn-not-alpha}
and Theorem \ref{thm-root-space-decomp} (iii)).  For every $g \in G_\alpha$,
we will write:
\beq\label{eqn-det-alpha-notalpha}
\on{det}_\alpha (g) = \on{det} (g |_{V_\alpha}), \;\;\;\;
\on{det}_\notalpha (g) = \on{det} (g |_{\Lg \cdot \Cartan_\alpha}).
\eeq
We then have:
\beqn
\on{det} (g) = \on{det}_\alpha (g) \cdot \on{det}_\notalpha (g),
\;\; g \in G_\alpha \, .
\eeqn
We will show that:
\beq\label{eqn-notalpha-one}
\on{det}_\notalpha (g) = 1 \;\; \text{for every} \;\; g \in Z_{G_\alpha} (\Cartan),
\eeq
and the proposition will follow.  As in the proof of Proposition
\ref{prop-char-tau}, we will use considerations related to the Hessian
$\cH [v, l]$, for suitable $v \in \Cartan$ and $l \in \Cartan^*$.

Recall the notation of equation \eqref{eqn-not-alpha}, and define:
\beqn
\Cartan_\notalpha^* \; = \; \bigcup_{\beta \in A_\notalpha}
\Cartan_\beta^* \, , \;\;\;\;
(\Cartan_\alpha^*)^{reg} \; = \; \Cartan_\alpha^* - \Cartan_\notalpha^* \, .
\eeqn
Pick some $v \in \Cartan_\alpha^{reg}$ and $l \in (\Cartan_\alpha^*)^{reg}$,
note that $\Lg \cdot v = \Lg \cdot \Cartan_\alpha$, and consider the
Hessian $\cH [v, l]$, as in \eqref{eqn-H-v-l}.  By Theorem
\ref{thm-root-space-decomp} (iv) and Propositions
\ref{prop-decompose-hessian}, \ref{prop-morse-atom}, the Hessian
$\cH [v, l]$ is a non-degenerate quadratic form on $\Lg \cdot \Cartan_\alpha$.
By the choice of $v \in V$ and $l \in V^*$, the group $G_\alpha$ acts on
$\Lg \cdot \Cartan_\alpha$ preserving the Hessian $\cH [v, l]$.  Equation
\eqref{eqn-notalpha-one} follows, since $G_\alpha$ is connected.
\end{proof}

\subsection{Full equivariance in rank one}
\label{subsec-full}

In this subsection, we prove Proposition-Definition \ref{prop-def-R-chi-alpha}.
Instead of giving the shortest proof possible, we begin with a preliminary
discussion which will be useful later (see, for example, the proof of Proposition
\ref{prop-mu-alpha}).

Fix an $\alpha \in A$.  Let $\Gaf = Z_G (\Cartan_\alpha) \subset G$.  Note that
we have $G_\alpha = (\Gaf)^0$.  Let $\Kaf = K \cap \Gaf$.  By analogy with
\eqref{eqn-K-alpha}, one can show that:
\beqn
\text{$\Kaf$ is a compact form of $\Gaf$.}
\eeqn
By \eqref{eqn-orthogonal} and Proposition \ref{prop-decompose-hessian},
decomposition \eqref{eqn-V-alpha-notalpha} is orthogonal with respect to
$\langle \; , \, \rangle$.  Since the compact form $\Kaf \subset \Gaf$ preserves
both the subspace $\Lg \cdot \Cartan_\alpha \subset V$ and the inner product
$\langle \; , \, \rangle$, we can conclude that:
\beq\label{eqn-Gaf-V-alpha}
\text{the group $\Gaf$ acts on $V$ preserving the decomposition
\eqref{eqn-V-alpha-notalpha}.}
\eeq

By the locality assumption \eqref{eqn-locality-assumption}, we have:
\beqn
V_\alpha \inv \Gaf = \Cartan / W_\alpha = Q_\alpha \, .
\eeqn
It follows that both the local system $\cL_{\chi^{}_\alpha}$ on
$X_{\breve c_0, \alpha}$ and the nearby cycles sheaf $P_{\chi^{}_\alpha}$
can be viewed $\Gaf$-equivariantly (see Section \ref{subsec-min-poly}).
Let:
\beq\label{eqn-tilde-B-W-alpha-f}
\widetilde B_{W_\alpha}^{\rm f} \coloneqq
\pi_1^{\Gaf} (V_\alpha^{rs}, l_{0, \alpha}), \;\;\;\;
\widetilde W_\alpha^{\rm f} \coloneqq
N_{\Gaf} (\Cartan) / Z_{\Gaf} (\Cartan)^0.
\eeq
Evidently, we have $Z_{\Gaf} (\Cartan) = Z_G (\Cartan)$.  Therefore, we
can think of $\widetilde W_\alpha^{\rm f}$ as a subgroup of $\widetilde W$,
and we have $I \subset \widetilde W_\alpha^{\rm f}$.  Moreover,
by \eqref{eqn-locality-assumption}, we have:
\beq\label{eqn-tilde-W-alpha-f}
\widetilde W_\alpha^{\rm f} = q^{-1} (W_\alpha) \subset \widetilde W.
\eeq

By construction, we have:
\beqn
\widetilde W_\alpha \subset \widetilde W_\alpha^{\rm f} \, , \;\;\;\;
\widetilde B_{W_\alpha} \subset \widetilde B_{W_\alpha}^{\rm f} \, .
\eeqn
By analogy with diagrams \eqref{eqn-main-diagram} and 
\eqref{eqn-main-diagram-alpha}, we have a diagram:
\beq
\label{eqn-main-diagram-alpha-full}
\begin{CD}
1 @>>> I @>>> \widetilde B_{W_\alpha}^{\rm f} @>{\tilde q_\alpha^{\rm f}}>>
B_{W_\alpha} @>>> 1 \;\,
\\
@. @| @VV{\tilde p_\alpha^{\rm f}}V @VV{p_\alpha}V @.
\\
1 @>>> I @>>> \widetilde W_\alpha^{\rm f} @>{q_\alpha^{\rm f}}>>
W_\alpha @>>> 1 \, ,
\end{CD}
\eeq
the right square of which is Cartesian.  Moreover, the maps $\tilde q_\alpha,
\tilde p_\alpha, q_\alpha$ of diagram \eqref{eqn-main-diagram-alpha} are the
restrictions of the maps $\tilde q_\alpha^{\rm f}, \tilde p_\alpha^{\rm f},
q_\alpha^{\rm f}$ of diagram \eqref{eqn-main-diagram-alpha-full}.  Thus,
the entire diagram \eqref{eqn-main-diagram-alpha} embeds into diagram
\eqref{eqn-main-diagram-alpha-full}.  This enables us to formulate the 
following lemma.

\begin{lemma}\label{lemma-conjugate}
Let $\sigma_1, \sigma_2 \in B_W [\alpha]$.  Then the elements
$\tilde r [\sigma_1] \, (\sigma), \, \tilde r [\sigma_2] \, (\sigma) \in
\widetilde B_{W_\alpha}$ (cf. equation \eqref{eqn-R-chi-alpha})
are conjugate in $\widetilde B_{W_\alpha}^{\rm f}$ by an element
of $I \subset \widetilde B_{W_\alpha}^{\rm f}$.
\end{lemma}

\begin{proof}
By \eqref{eqn-conjugate}, there exists an element $b \in P\!B_W$, such that
$\sigma_2 = b \, \sigma_1 \, b^{-1}$.  Let $x = r (b) \in I$.  We then have
$r (\sigma_2) = x \; r (\sigma_1) \; x^{-1}$.  By equation \eqref{eqn-r-sigma-alpha}
and the Cartesian property of diagram \eqref{eqn-main-diagram-alpha-full},
we can conclude that $r [\sigma_2] (\sigma) = x \; r [\sigma_1] (\sigma) \; x^{-1}$,
as required.
\end{proof}

Isomorphism \eqref{eqn-identify-rank-one} naturally extends to an isomorphism:
\beqn
\pi_1^{\Gaf} ((V_\alpha^*)^{rs}, l_{0, \alpha}) \cong
\widetilde B_{W_\alpha}^{\rm f} \, .
\eeqn
Using this isomorphism, the $\Gaf$-equivariant structure on the sheaf
$P_{\chi^{}_\alpha}$ produces a microlocal monodromy action:
\beq\label{eqn-lambda-grp-rank-one-full}
\lambda_{l_0, \alpha} : \widetilde B_{W_\alpha}^{\rm f} \to
\on{Aut} (M_{l_{0, \alpha}} (P_{\chi^{}_\alpha})),
\eeq
which extends the action of equation \eqref{eqn-lambda-grp-rank-one}.

\begin{proof}[Proof of Proposition-Definition \ref{prop-def-R-chi-alpha}]
Part (i) follows readily from Lemma \ref{lemma-conjugate}.  Indeed, let
$\sigma_1, \sigma_2 \in B_W [\alpha]$ be a pair of braid generators.
Using Lemma \ref{lemma-conjugate} and the action
\eqref{eqn-lambda-grp-rank-one-full}, we can conclude that the 
linear transformations:
\beqn
\lambda_{l_0, \alpha} \circ \, \tilde r [\sigma_1] \, (\sigma), \;
\lambda_{l_0, \alpha} \circ \, \tilde r [\sigma_2] \, (\sigma)
\in \on{End} (M_{l_{0, \alpha}} (P_{\chi^{}_\alpha})),
\eeqn
are conjugate to each other by an element of
$\on{Aut} (M_{l_{0, \alpha}} (P_{\chi^{}_\alpha}))$.  Therefore, these
linear transformations have the same minimal polynomial, as required.

Turning to part (ii), note that, for every $\alpha \in A$ and $\sigma_\alpha 
\in B_W [\alpha]$, the polynomial $R_{\chi, \alpha}$ is fully determined by the
following data:
\begin{enumerate}[topsep=-1.5ex]
\item[$\bullet$]    the polar representaiton $G_\alpha | V_\alpha$;

\item[$\bullet$]    the character $\chi_\alpha \in \hat I_\alpha$;

\item[$\bullet$]    the element $r (\sigma_\alpha) \in \widetilde W_\alpha$.
\end{enumerate}
Pick $\sigma_1 \in B_W [\alpha_1]$ and $b \in B_W^\chi$ with $p (b) = w$.
Let $\sigma_2 = b \, \sigma_1 \, b^{-1} \in B_W [\alpha_2]$ (see assertion
\eqref{eqn-conjugate-two}).  Let $\widetilde w = r (b) \in \widetilde W =
N_G (\Cartan) / Z_G (\Cartan)^0$ and note that $q (\widetilde w) = w$.  Pick a
representative $g \in N_G (\Cartan)$ of $\widetilde w$.  The conjugation action
of $N_G (\Cartan)$ on $G$ induces actions on $I$, $\hat I$, and $\widetilde W$.
Moreover, by tracing the definitions, one can check that the action of $g$ takes
the data $(G_{\alpha_1} | V_{\alpha_1}, \chi_{\alpha_1}, r (\sigma_1))$ to the
data $(G_{\alpha_2} | V_{\alpha_2}, \chi_{\alpha_2}, r (\sigma_2))$.  Part (ii)
follows.
\end{proof}

\begin{remark}\label{rmk-prop-def-proof}{\em
Our proof of Proposition-Definition \ref{prop-def-R-chi-alpha} (ii) can be readily
adapted to prove part (i) as well.  However, we feel that highlighting the
$\Gaf$-equivariant structure on $P_{\chi^{}_\alpha}$ serves to clarify the
argument.}
\end{remark}

\subsection{Existence of regular splittings}
\label{subsec-reg-split-proof}

In this subsection, we give a proof of Proposition \ref{prop-regular-exists}.  We
begin by describing a homomorphism $\tilde r : B_W \to \widetilde B_W$ which
splits the top row of diagram \eqref{eqn-main-diagram}, then verify that 
$\tilde r$ is a regular splitting in the sense of Definition \ref{defn-reg-split}
(see Proposition \ref{prop-reg-split} below).

Our construction of $\tilde r$ is analogous to the definition of a Kostant-Rallis
splitting in \cite[Section 2.3]{GVX1}.  We begin with the regular element $x \in 
X_0$.  The variety $X_0$ is smooth at $x$.  Pick an affine normal slice
$Y \subset V$ to $X_0$ through $x$.  This means that $Y$ is an affine flat
through $x$ in $V$, and we have:
\beqn
T_x X_0 \oplus T_x Y = T_x V \cong V.
\eeqn
The restriction $f |_Y : Y \to Q$ is a local analytic isomorphism near $x$.
Recall the Hermitian inner product $\langle \; , \, \rangle$ on $V$.  Pick
a small $\epsilon > 0$, and let $\Cartan_\epsilon \subset \Cartan$ be the
open $\epsilon$-ball around $0$, with respect to $\langle \; , \, \rangle$.
Let $Q_\epsilon = f (\Cartan_\epsilon)$.  We assume that $\epsilon$ is
sufficiently small, so that there exists an open neighborhood $Y_\epsilon
\subset Y$ of $x$, such that $f (Y_\epsilon) = Q_\epsilon$ and:
\beqn
f_\epsilon \coloneqq f |_{Y_\epsilon} : Y_\epsilon \to Q_\epsilon \, ,
\eeqn
is a complex analytic isomorphism.  Let $Q_\epsilon^{reg} = 
Q_\epsilon \cap Q^{reg}$, and note that the inclusion
$Q_\epsilon^{reg} \hookrightarrow Q^{reg}$ is a homotopy equivalence.
Pick a $k \in (0, 1]$ such that:
\beq\label{eqn-c-one}
c_1 \coloneqq k \, c_0 \in \Cartan_\epsilon,
\eeq
and let $\bar c_1 = f (c_1)$.  We then have:
\begin{align}
\label{eqn-pi1-iso-1}
B_W & =
\pi_1 (Q^{reg}, \bar c_0) = \pi_1 (Q^{reg}, \bR_+ \cdot \bar c_0) =
\pi_1 (Q^{reg}, \bar c_1) \cong \pi_1 (Q_\epsilon^{reg}, \bar c_1). \\
\label{eqn-pi1-iso-2}
\widetilde B_W & =
\pi_1^G (V^{rs}, c_0) = \pi_1^G (V^{rs}, \bR_+ \cdot c_0)
= \pi_1^G (V^{rs}, c_1).
\end{align}

Let $X_{\bar c_1} = f^{-1} (\bar c_1)$ and $y_1 = f_\epsilon^{-1} (\bar c_1)$,
so that $X_{\bar c_1} \cap Y_\epsilon = \{ y_1 \}$.  Pick a continuous path
$\Gamma_1 : [0, 1] \to X_{\bar c_1}$, with $\Gamma_1 (0) = c_1$
and $\Gamma_1 (1) = y_1$.  We are now prepared to define the
homomorphism $\tilde r : B_W \to \widetilde B_W$.  Let $b \in B_W$ be an
element represented by a path $\Gamma_b : [0, 1] \to Q_\epsilon^{reg}$,
with $\Gamma_b (0) = \Gamma_b (1) = \bar c_1$, using the identification
\eqref{eqn-pi1-iso-1}.  We set:
\beqn
\tilde r (b) = \Gamma_1^{-1} \star (f_\epsilon^{-1} \circ \Gamma_g)
\star \Gamma_1 \in \widetilde B_W \, ,
\eeqn
where ``$\; \star \;$'' denotes the composition of paths, and we use the 
identification \eqref{eqn-pi1-iso-2}.  It is clear from the definition that
$\tilde r$ is a group homomorphism which splits the top row of diagram
\eqref{eqn-main-diagram}.

\begin{prop}\label{prop-reg-split}
The homomorphism $\tilde r$ is a regular splitting in the sense of
Definition \ref{defn-reg-split}.
\end{prop}

Our proof of Proposition \ref{prop-reg-split} will rely on the following lemma.

\begin{lemma}\label{lemma-path-in-V-alpha}
Let $\alpha \in A$, and let $\Gamma : [0, 1] \to V^{rs} \cap V_\alpha$ be a
path with $\Gamma (0) = \Gamma (1) = c_0$, representing an element
$\tilde b \in \widetilde B_W$.  Then we have $\tilde p (\tilde b) \in \widetilde
W_\alpha$.
\end{lemma}

\begin{proof}
For every $v \in V^{rs} \cap V_\alpha$, let $\Cartan_v \subset V$ be the Cartan
subspace containing $v$.  By the definition of $V_\alpha \subset V$, we have:
\beqn
\forall v \in V^{rs} \cap V_\alpha \;\; \exists g \in G_\alpha \, : \;\;
\Cartan_v = g \cdot \Cartan \subset V_\alpha.
\eeqn
The path $t \mapsto \Cartan_{\Gamma (t)}$ can be lifted to a continuous path
$t \mapsto g_t \in G_\alpha$, such that $g_0 = 1$ and $\Cartan_{\Gamma (t)} =
g_t \cdot \Cartan$ for every $t \in [0, 1]$.  The element $\tilde p (\tilde b) \in
\widetilde W$ is then represented by $g_1 \in G_\alpha$ (cf. equation
\eqref{eqn-concrete-loop}).  The lemma follows.
\end{proof}

\begin{proof}[Proof of Proposition \ref{prop-reg-split}]
This proof is a direct adaptation of the proof of \cite[Proposition 2.6]{GVX1},
and we focus on the aspects specific to the present setting.  Pick an
$\alpha \in A$ and a braid generator $\sigma_\alpha = \sigma_\alpha
[\Gamma] \in B_W [\alpha]$, as in equation \eqref{eqn-sigma-alpha}.
Recall that $\Gamma : [0, 1] \to \Cartan^{reg}$ is a continuous path,
with $\Gamma (0) = c_0$ and $\Gamma (1) = c_{\alpha, 1}$, where
the point $c_{\alpha, 1} \in \Cartan^{reg}$ lies near a point $c_\alpha \in
\Cartan_\alpha^{reg}$, which is the projection of $c_{\alpha, 1}$ onto
$\Cartan_\alpha$ (see equation \eqref{eqn-c-alpha}).  We need to show
that $r (\sigma_\alpha) \in \widetilde W_\alpha \subset \widetilde W$.

Without loss of generality, we can assume that $c_1 = c_0$ (see equation
\eqref{eqn-c-one}), and that the path of equation \eqref{eqn-sigma-alpha},
representing $\sigma_\alpha \in B_W$, is contained entirely within
$Q_\epsilon^{reg} \subset Q^{reg}$.  It follows that
$c_\alpha \in \Cartan_\epsilon$.  Let:
\beq\label{eqn-bar-c-alpha}
\bar c_\alpha = f (c_\alpha), \;\;
X_{\bar c_\alpha} = f^{-1} (\bar c_\alpha), \;\;
y_\alpha = f_\epsilon^{-1} (\bar c_\alpha).
\eeq
The main idea of the proof is to observe that, in order to specify the
element $r (\sigma_\alpha) \in \widetilde W$ up to conjugation by elements
of $I \subset \widetilde W$, we do not need the entirety of the normal slice
$Y_\epsilon$, but only the germ of it at $y_\alpha = Y_\epsilon \cap
X_{\bar c_\alpha}$.

More precisely, we will make use of the following construction.  Let
$V^{reg} \subset V$ be the set of regular points for $f$, and let
$X_{\bar c_\alpha}^{reg} = X_{\bar c_\alpha} \cap V^{reg}$.  Pick a point
$z_\alpha \in X_{\bar c_\alpha}^{reg}$ and an affine normal slice
$N \subset V$ through $z_\alpha$ to $X_{\bar c_\alpha}^{reg} \subset
V^{reg}$.  For $\delta > 0$, let $\Cartan [c_\alpha, \delta] \subset \Cartan$
be the open $\delta$-ball around $c_\alpha$.  We will say that $\delta > 0$
is sufficiently small for the normal slice $N$ if we have:
\beqn
\Cartan [c_\alpha, 2 \delta] \subset \Cartan_\epsilon \, ,
\eeqn
and there exits a neighborhood $N_{2 \delta} \subset N$ of $z_\alpha$,
such that:
\beqn
f (N_{2 \delta}) = Q [\bar c_\alpha, 2 \delta] \coloneqq
f (\Cartan [c_\alpha, 2 \delta]),
\eeqn
and the restriction:
\beqn
f_{N, 2 \delta} \coloneqq f |_{N_{2 \delta}} :  N_{2 \delta} \to
Q [\bar c_\alpha, 2 \delta],
\eeqn
is a complex analytic isomorphism.

Let $\delta_1 = \on{dist} (c_\alpha, c_{\alpha, 1})$, and fix a $\delta \in
(0, \delta_1]$ which is sufficiently small for $N$.  Define:
\beq\label{eqn-c-alpha-2}
c_{\alpha, 2} = c_\alpha + (\delta / \delta_1) \cdot (c_{\alpha, 1} - c_\alpha)
\in \Cartan^{reg}, \;\;
\bar c_{\alpha, 2} = f (c_{\alpha, 2}) \in Q^{reg},
\eeq
\beq\label{eqn-z-2}
X_{\bar c_{\alpha, 2}} = f^{-1} (\bar c_{\alpha, 2}) \subset V^{rs}, \;\;
z_2 = f_{N, 2 \delta}^{-1} (\bar c_{\alpha, 2}) =
X_{\bar c_{\alpha, 2}} \cap N_{2 \delta} \, .
\eeq
Pick a continuous path $\Gamma_2 : [0, 1] \to X_{\bar c_{\alpha, 2}}$,
with $\Gamma_2 (0) = c_{\alpha, 2}$ and $\Gamma_2 (1) = z_2$. 
The quadruple $[z_\alpha, N, \delta, \Gamma_2]$ defines an element
$\tilde b = \tilde b \, [z_\alpha, N, \delta, \Gamma_2] \in \widetilde B_W$
as follows:
\beq\label{eqn-tilde-b}
\tilde b =
\Gamma^{-1} \star [c_{\alpha, 2}, c_{\alpha, 1}] \star \Gamma_2^{-1}
\star (f_{N, 2 \delta}^{-1} \circ f \circ \Gamma_\alpha [c_{\alpha, 2}])
\star \Gamma_2 \star [c_{\alpha, 1}, c_{\alpha, 2}] \star \Gamma
\in \widetilde B_W \, ,
\eeq
where $[c_{\alpha, 1}, c_{\alpha, 2}]$ and $[c_{\alpha, 2}, c_{\alpha, 1}]$
are the straight line paths, and $\Gamma_\alpha [c_{\alpha, 2}]$ is
defined as in \eqref{eqn-gamma-alpha}.  In other words, the element
$\tilde b$ is represented by a small loop inside of $N$ which links the
hypersurface $f^{-1} (f (\Cartan_\alpha)) \subset V$ in the counter-clockwise
direction and is connected to the basepoint $c_0 \in V^{rs}$ via the paths
$\Gamma$, $[c_{\alpha, 1}, c_{\alpha, 2}]$, and $\Gamma_2$.  

We now make the following observations regarding the above construction.
First, for suitable $\delta$ and $\Gamma_2$, we have:
\beq\label{eqn-use-construction}
\tilde r (\sigma_\alpha) = \tilde b \, [y_\alpha, Y, \delta, \Gamma_2].
\eeq
Second, write $\widetilde B_W / I$ for the set of $I$-orbits in $\widetilde B_W
\supset I$, under the conjugation action (see diagram \eqref{eqn-main-diagram}).
Then the image of $\tilde b \, [z_\alpha, N, \delta, \Gamma_2]$ in
$\widetilde B_W / I$ is independent of the path $\Gamma_2$.  We denote this
image by:
\beqn
\tilde b \, [z_\alpha, N, \delta] \in \widetilde B_W / I.
\eeqn
Third, by continuity, the element $\tilde b \, [z_\alpha, N, \delta] \in
\widetilde B_W / I$ is independent of the normal slice $N$ and the number
$\delta \in (0, \delta_1]$.  Thus, we can omit $N$ and $\delta$ from the notation,
writing:
\beqn
\tilde b \, [z_\alpha, N, \delta] = \tilde b \, [z_\alpha] \in \widetilde B_W / I.
\eeqn
And fourth, again by continuity, we have:
\beq\label{eqn-con-component}
\text{
$\tilde b \, [z_\alpha] \in \widetilde B_W / I$ is determined by the connected
component of $z_\alpha \in X_{\bar c_\alpha}^{reg} \,$.}
\eeq

Write $\widetilde W / I$ for the set of $I$-orbits in $\widetilde W \supset I$,
under the conjugation action, and let:
\beqn
\omega : \widetilde W \to \widetilde W / I,
\eeqn
be the quotient map.  The map $\tilde p : \widetilde B_W \to \widetilde W$
of diagram \eqref{eqn-main-diagram} induces a map:
\beqn
\tilde p : \widetilde B_W / I \to \widetilde W / I.
\eeqn
Recall that the conjugation action of $I$ on $\widetilde W$ preserves the
subgroup $\widetilde W_\alpha \subset \widetilde W$ (see equation
\eqref{eqn-I-invariance}).  Thus, in view of \eqref{eqn-use-construction},
in order to prove the proposition, it suffices to show that:
\beq\label{eqn-in-tilde-W-alpha}
\tilde p \, (\tilde b \, [y_\alpha]) \in \widetilde W_\alpha / I =
\omega (\widetilde W_\alpha).
\eeq

We will establish \eqref{eqn-in-tilde-W-alpha} by showing that the construction
of equation \eqref{eqn-tilde-b} can be performed entirely within the subspace
$V_\alpha \subset V$; then applying Lemma \ref{lemma-path-in-V-alpha}.
More precisely, consider the intersection $X_{\bar c_\alpha} \cap V_\alpha$.
Let:
\beq\label{eqn-breve-c-alpha}
\breve c_\alpha = f_\alpha (c_\alpha), \;\;
X_{\breve c_\alpha, \alpha} = f_\alpha^{-1} (\breve c_\alpha) \, .
\eeq
Then $X_{\breve c_\alpha, \alpha}$ is the connected component of
$X_{\bar c_\alpha} \cap V_\alpha$ containing $c_\alpha$.  By Theorem 
\ref{thm-root-space-decomp} (iii)-(iv) (see also equation
\eqref{eqn-V-alpha-notalpha}), we have:
\beq\label{eqn-normal-slice}
\text{the subspace $V_\alpha \subset V$ is a normal slice to the closed orbit
$G \cdot c_\alpha$ through $c_\alpha \,$.}  
\eeq
Since $G \cdot c_\alpha$ is the unique closed $G$-orbit in $X_{\bar c_\alpha}$,
we can conclude that every $G$-orbit in $X_{\bar c_\alpha}$ meets the
connected component $X_{\breve c_\alpha, \alpha}$.  Pick a point
$z_\alpha \in G \cdot y_\alpha \cap X_{\breve c_\alpha, \alpha}$, and note that:
\beq\label{eqn-z-alpha-reg}
z_\alpha \in X_{\bar c_\alpha}^{reg} \, .
\eeq
Since $G$ is connected, and in view of \eqref{eqn-con-component}, we have:
\beq\label{eqn-y-eq-z}
\tilde b \, [y_\alpha] = \tilde b \, [z_\alpha] \in \widetilde B_W / I.
\eeq

Consider the orbits $G_\alpha \cdot z_\alpha \subset G \cdot z_\alpha \subset 
X_{\bar c_\alpha}^{reg}$.  Note that $\{ c_\alpha \}$ is the unique closed
$G_\alpha$-orbit in $X_{\breve c_\alpha, \alpha}$.  Therefore, the point
$c_\alpha \in V_\alpha$ is contained in the closure of $G_\alpha \cdot z_\alpha$.
In view of \eqref{eqn-normal-slice}, we can conclude that:
\beq\label{eqn-transverse}
\text{the intersection $G \cdot z_\alpha \cap V_\alpha$ is transverse along
$G_\alpha \cdot z_\alpha \,$.}
\eeq
Let $V_\alpha^{reg} \subset V_\alpha$ be the set of regular points for $f_\alpha$,
and let $X_{\breve c_\alpha, \alpha}^{reg} = X_{\breve c_\alpha, \alpha} \cap
V_\alpha^{reg}$.  Assertions \eqref{eqn-z-alpha-reg} and \eqref{eqn-transverse}
imply that:
\beq\label{eqn-z-alpha-reg-V-alpha}
\text{$z_\alpha$ is a regular point for the restriction map
$f |_{V_\alpha} : V_\alpha \to Q$.}
\eeq
Let $h_\alpha : Q_\alpha \to Q$ be the quotient map, and consider the
factorization:
\beq\label{eqn-factorization-h-alpha}
f |_{V_\alpha} = h_\alpha \circ f_\alpha : V_\alpha \to Q.
\eeq
Assertion \eqref{eqn-z-alpha-reg-V-alpha} and equation
\eqref{eqn-factorization-h-alpha} imply that $z_\alpha \in
X_{\breve c_\alpha, \alpha}^{reg}$ and that:
\beq\label{eqn-breve-c-alpha-reg}
\breve c_\alpha \in Q_\alpha \text{ is a regular point of } h_\alpha \, .
\eeq
Note that assertion \eqref{eqn-breve-c-alpha-reg} can also be inferred
directly from the locality assumption \eqref{eqn-locality-assumption}.  By
\eqref{eqn-factorization-h-alpha} and \eqref{eqn-breve-c-alpha-reg}, we have:
\beq\label{eqn-reg-reg}
X_{\breve c_\alpha, \alpha}^{reg} \subset X_{\bar c_\alpha}^{reg} \, .
\eeq
Let $N \subset V_\alpha$ be an affine normal slice through $z_\alpha$
to $X_{\breve c_\alpha, \alpha}^{reg} \subset V_\alpha^{reg}$.
By \eqref{eqn-transverse} and \eqref{eqn-reg-reg}, the subspace
$N \subset V_\alpha \subset V$ is also an affine normal slice through
$z_\alpha$ to $X_{\bar c_\alpha}^{reg} \subset V^{reg}$.

We can now use the normal slice $N$ to identify the element
$\tilde b \, [z_\alpha]$ of equation \eqref{eqn-y-eq-z}.  Pick a $\delta \in
(0, \delta_1]$ which is sufficiently small for $N$, and define $c_{\alpha, 2}
\in \Cartan^{reg}$ and $z_2 \in N_{2 \delta}$ as in equations 
\eqref{eqn-c-alpha-2}-\eqref{eqn-z-2}.  Let:
\beqn
\breve c_{\alpha, 2} = f_\alpha (c_{\alpha, 2}), \;\;
X_{\breve c_{\alpha, 2}, \alpha} = f_\alpha^{-1} (\breve c_{\alpha, 2}) \, .
\eeqn
Using the factorization \eqref{eqn-factorization-h-alpha}, it is not hard to check
that $z_2 \in X_{\breve c_{\alpha, 2}, \alpha} \subset X_{\bar c_{\alpha, 2}}$.  
Pick a continuous path $\Gamma_2 : [0, 1] \to X_{\breve c_{\alpha, 2}, \alpha}$,
with $\Gamma_2 (0) = c_{\alpha, 2}$ and $\Gamma_2 (1) = z_2$, and consider
the construction of the element $\tilde b = \tilde b \, [z_\alpha, N, \delta,
\Gamma_2] \in \widetilde B_W$ of equation \eqref{eqn-tilde-b}.  We see that
$\tilde b \in \widetilde B_W$ is represented by a closed loop in $V^{rs} \cap
V_\alpha$.  By Lemma \ref{lemma-path-in-V-alpha}, it follows that $\tilde p
(\tilde b) \in \widetilde W_\alpha$.  By \eqref{eqn-y-eq-z}, this proves the
containment \eqref{eqn-in-tilde-W-alpha}.
\end{proof}

\section{Fourier transform of the nearby cycles}
\label{sec-fourier}

As in Section \ref{subsec-min-poly} (see also \cite[Section 7.1]{GVX1}), we use
Proposition \ref{prop-A-f-stratification} to obtain a Morse local system
$M (P_\chi)$ on $V^{rs}$, of the sheaf $P_\chi$ at the origin.  For every
$l \in (V^*)^{rs}$ we write:
\beqn
M_l (P_\chi) = M_{(0, l)} (P_\chi),
\eeqn
for the corresponding Morse group (cf. equation \eqref{eqn-M-l-alpha}).

\begin{prop}\label{prop-fourier}
We have:
\beqn
\FT P_\chi \cong \on{IC} ((V^*)^{rs}, M (P_\chi)).
\eeqn
\end{prop}

\begin{proof}
This is similar to the proof of \cite[Proposition 7.1]{GVX1}.
By \cite[Proposition 3.7.12 (ii)]{KS} (see also the proof of
\cite[Lemma 6.2]{GVX1}), we have:
\beqn
\FT P_\chi |_{V^{rs}} \cong M (P_\chi) \, [\dim V].
\eeqn
By \cite[Proposition 2.17]{Gr1}, the general fiber $X_{\bar c_0} \subset V$ of
the family $\cZ_{\bar c_0} \to \bC$ of equation \eqref{eqn-family-Z} is
transverse to infinity in the sense of \cite{Gr2}.  The proposition follows from
\cite[Theorem 1.1 $\&$ Remark 1.4]{Gr2}.
\end{proof}

Let $l \in (V^*)^{rs}$, and let $\Cartan_l = (\Lg \cdot l)^\p \subset V$
be the Cartan subspace corresponding to $l$.  Write $Z_l \subset 
X_{\bar c_0}$ for the critical locus of the restriction $l |_{X_{\bar c_0}}$.
We then have:
\beq\label{eqn-Z-l}
Z_l = X_{\bar c_0} \cap \Cartan_l \, , \;\; \text{and therefore} \;\; |Z_l| = |W|.
\eeq
Write $\xi_l = \on{Re} (l) : V \to \bR$ and $d = \dim X_{\bar c_0}$.  We have
the following analog of \cite[Lemma 7.4]{GVX1}.

\begin{lemma}\label{lemma-describe-morse-group}
The Morse group $M_l (P_\chi)$ can be identified as follows:
\beqn
M_l (P_\chi) \cong H_d (X_{\bar c_0}, \{ x \in X_{\bar c_0} \; | \;
\xi_l (x) \geq \xi_0 \}; \cL_\chi),
\eeqn
where $\xi_0$ is any real number with $\xi_0 > \xi_l (c)$ for every $c \in Z_l$.
\end{lemma}

\begin{proof}
This is a consequence of Proposition \ref{prop-A-f-stratification}.
The proof is similar to the proofs of \cite[Lemmas 6.6, 6.9, 7.4]{GVX1}.
\end{proof}

\begin{corollary}\label{cor-dim-M}
We have:
\beqn
\dim M_l (P_\chi) = |W|.
\eeqn
\end{corollary}

\begin{proof}
This follows form Lemma \ref{lemma-describe-morse-group}, equation
\eqref{eqn-Z-l}, assertion \eqref{eqn-morse}, and the generic property of $l$,
as expressed by Proposition \ref{prop-A-f-stratification}
(cf. \cite[Lemma 6.10]{GVX1}).
\end{proof}

Recall from Section \ref{subsec-dual-rep} that we identify the fundamental
group $\pi_1^G ((V^*)^{rs}, l_0)$ with $\widetilde B_W$ (see isomorphism
\eqref{eqn-identify-pi-one}).  Using this identification, we will write:
\beq\label{eqn-lambda-grp}
\lambda_{l_0} : \widetilde B_W \to \on{Aut} (M_{l_0} (P_\chi)),
\eeq
for the microlocal monodromy action arising from the structure of $P_\chi$
as a $G$-equivariant perverse sheaf, i.e., for the holonomy of the local
system $M (P_\chi)$.

We will analyze the action $\lambda_{l_0}$ by means of Picard-Lefschetz
theory, and we will use the same notation for Picard-Lefschetz classes in
the Morse groups $\{ M_l (P_\chi) \}$ as in \cite{GVX1}.  More precisely, let
$l \in (V^*)^{rs}$ and let $\xi_0 > 0$ be as in Lemma
\ref{lemma-describe-morse-group}.  Let $c \in Z_l$, and let
$\gamma : [0, 1] \to \bC$ be a smooth path such that:
\begin{flalign*}
& \;\;\;\;\;\;\;\; \text{(P1)} \;\;\;\;
\gamma (0) = l (c); & \\
& \;\;\;\;\;\;\;\; \text{(P2)} \;\;\;\;
\text{$\gamma' (t) \neq 0$ for all $t \in [0, 1]$;} & \\
& \;\;\;\;\;\;\;\; \text{(P3)} \;\;\;\;
\gamma (1) = \xi_0; & \\
& \;\;\;\;\;\;\;\; \text{(P4)} \;\;\;\;
\text{$\gamma (t) \notin l (Z_l)$ for all $t \in (0, 1)$;} & \\
& \;\;\;\;\;\;\;\; \text{(P5)} \;\;\;\;
\text{$\gamma (t_1) \neq \gamma (t_2)$ for all $t_1, t_2 \in [0, 1]$
with $t_1 \neq t_2$.} &
\end{flalign*}
Recall the Hessian $\cH [c, l]$ of $l |_{X_{\bar c_0}}$ at $c$ (see equation
\eqref{eqn-H-v-l}).  Let:
\beq\label{eqn-T-plus}
T_+ [c, \gamma] \subset \Lg \cdot c,
\eeq
be the positive eigenspace of $\gamma' (0)^{-1} \cdot \cH [c, l]$, relative to the
inner product $\langle \; , \, \rangle$, i.e., the direct sum of all the eigenspaces of
the real part $\on{Re} (\gamma' (0)^{-1} \cdot \cH [c, l])$, corresponding to
positive eigenvalues.  By assertion \eqref{eqn-morse}, we have:
\beqn
\dim_\bR T_+ [c, \gamma] = d.
\eeqn
Pick an orientation
$o$ of $T_+ [c, \gamma]$ and an element $a \in (\cL_\chi)_c$.  The quadruple
$[c, \gamma, o, a]$ defines a Picard-Lefschetz class:
\beq\label{eqn-PL}
\PL [c, \gamma, o, a] \in M_l (P_\chi),
\eeq
as in \cite[Sections 6.2, 7.2]{GVX1}.

We now focus on the Morse group $M_{l_0} (P_\chi)$.  Note that
$l_0 (c_0) > 0$, and we have:
\beqn
\xi_{l_0} (c) < l_0 (c_0) \;\; \text{for every} \;\; c \in Z_{l_0} - \{ c_0 \},
\eeqn
(see equation \eqref{eqn-l-zero}).  Fix a $\xi_0 > l_0 (c_0)$, and let
$\gamma_0 : [0, 1] \to \bC$ be the straight line path from $l_0 (c_0)$ to 
$\xi_0$.  Pick an orientation $o_0$ of $T_+ [c_0, \gamma_0]$ and a generator 
$0 \neq a_0 \in (\cL_\chi)_{c_0}$.  We define:
\beq\label{eqn-u-zero}
u_0 = \PL [c_0, \gamma_0, o_0, a_0] \in M_{l_0} (P_\chi).
\eeq

\begin{remark}\label{rmk-not-generic}{\em
Our choice of the basepoint $l_0 \in (V^*)^{rs}$, which is coordinated with
the choice of $c_0 \in \Cartan^{reg} \subset V$ via \eqref{eqn-l-zero}, will
give us a number of advantages.  In particular, it facilitates the definition
of the class $u_0$ in equation \eqref{eqn-u-zero}, and this class will play
a central role in our argument.  However, we point out that the pair 
$(c_0, l_0)$ is not generic in the following sense.  The critical values
$l_0 (Z_{l_0}) \subset \bC$ will not, in general, be distinct.  For example,
if $W = S_3$ is the symmetric group on three letters, generated by simple
reflections $(s_1, s_2)$, then we will always have $l_0 (s_1 s_2 \, c_0) =
l_0 (s_2 s_1 \, c_0)$.}
\end{remark}

\begin{prop}\label{prop-cyclic-vector}
The vector $u_0 \in M_{l_0} (P_\chi)$ is cyclic for the microlocal monodromy
action $\lambda_{l_0}$ of equation \eqref{eqn-lambda-grp}.  More precisely,
we have:
\beqn
\lambda_{l_0} (\bC [\widetilde B_W]) \cdot u_0 = M_{l_0} (P_\chi).
\eeqn
\end{prop}

\begin{proof}
In the case $\chi = 1$, the claim of the proposition is equivalent to the claim
of \cite[Lemma 3.2]{Gr3}.  Moreover, the proof of \cite[Lemma 3.2]{Gr3}
goes through without any changes in the presence of the local system
$\cL_\chi$.
\end{proof}

It is not difficult to describe the microlocal monodromy action 
\eqref{eqn-lambda-grp} of the subgroup $I \subset \widetilde B_W$ on 
Picard-Lefschetz classes (see diagram \eqref{eqn-main-diagram}).

\begin{prop}\label{prop-microlocal-inertia}
Let $u = \PL [c, \gamma, o, a] \in M_{l_0} (P_\chi)$ be a Picard-Lefschetz
class, with $c = w \, c_0 \in Z_{l_0}$, $w \in W$, and let $x \in I$.  We have:
\beqn
\lambda_{l_0} (x) \, u = (w \cdot \chi) (x) \cdot \tau (x) \cdot u \, ,
\eeqn
where $\tau$ is the character of equation \eqref{eqn-char-tau-I}.
\end{prop}

\begin{proof}
This is similar to the proofs of \cite[Lemma 6.19 $\&$ equation (7.13)]{GVX1}.
More precisely, let $g \in Z_G (\Cartan)$ be an element representing $x \in I$.
Note that we have $g \, c = c$ and $g \, l_0 = l_0$.  It follows that $g$
preserves the positive eigenspace $T_+ [c, \gamma] \subset \Lg \cdot c =
\Lg \cdot \Cartan$.  In view of the direct sum decomposition \eqref{eqn-stability},
the effect of $g$ on the orientation of $T_+ [c, \gamma]$ is given by $\tau (x)$.
By chasing the definitions of the $G$-equivariant local system $\cL_\chi$
and of the $W$-action on the characters $\hat I$, one can check that
$g \, a = (w \cdot \chi) (x) \cdot a$.  The proposition follows by the linearity of
the Picard-Lefschetz class $\PL [c, \gamma, o, a]$ in the last two arguments.
\end{proof}

\section{Construction of the monodromy in the family}
\label{sec-monodromy-construct}

Recall the groups $W_\chi = \on{Stab}_W (\chi)$ and  $B_W^\chi = 
\on{Stab}_{B_W} (\chi)$ of equations \eqref{eqn-W-chi} and
\eqref{eqn-tilde-B-chi}.  As in \cite[Section 3.2]{GVX1}, the sheaf $P_\chi$
naturally carries a representation:
\beq\label{eqn-mu}
\mu : B_W^\chi \to \on{Aut} (P_\chi),
\eeq
which we call the monodromy in the family, and which is constructed as follows.
Consider the projection $\Cartan / W_\chi \to \Cartan / W$, and form a 
Cartesian commutative diagram as follows:
\beq\label{diagram-V-tilde-chi}
\begin{CD}
\widetilde V_\chi @>{g}>> V
\\
@V{f_\chi}VV @VV{f}V
\\
\Cartan / W_\chi @>>> \Cartan / W \, .
\end{CD}
\eeq
Let $\widetilde V_\chi^{rs} = f_\chi^{-1} (\Cartan^{reg} / W_\chi) \subset
\widetilde V_\chi$.  Let $\hat c_0 \in \Cartan^{reg} / W_\chi$ be the
image of $c_0 \in \Cartan^{reg}$, and let:
\beq\label{eqn-tilde-c-zero}
\tilde c_0 = (\hat c_0, c_0) \in \widetilde V_\chi^{rs} \, .
\eeq
Then the restriction $g: \widetilde V_\chi^{rs} \to V^{rs}$ is a $G$-equivariant
covering map, and we have:
\beqn
\pi_1^G (\widetilde V_\chi^{rs}, \tilde c_0) = \widetilde B_W^\chi =
\tilde q^{-1} (B_W^\chi) \subset \widetilde B_W \, ,
\eeqn
(see equation \eqref{eqn-tilde-B-chi}).

Recall that the splitting homomorphism $\tilde r$ determines a canonical
extension $\hat \chi : \widetilde B_W^\chi \to \bG_m$ of the character
$\chi : I \to \bG_m$ (see equation \eqref{eqn-hat-chi}).  The character
$\hat \chi$ gives rise to a rank one $G$-equivariant local system $\hat \cL_\chi$
on $\widetilde V_\chi^{rs}$, with $(\hat \cL_\chi)_{\tilde c_0} = \bC$.  Note that
$\hat c_0 = f_\chi (\tilde c_0) \in \Cartan^{reg} / W_\chi$, and let $X_{\hat c_0} =
f_\chi^{-1} (\hat c_0)$.  We have natural identifications:
\beq\label{eqn-hat-no-hat}
X_{\hat c_0} \cong X_{\bar c_0} \;\;\;\; \text{and} \;\;\;\;
\hat \cL_\chi |_{X_{\hat c_0}} \cong \cL_\chi \, .
\eeq
Next, we consider a parametrized version of the family $\cZ_{\bar c_0} \to \bC$
of equation \eqref{eqn-family-Z}:
\beqn
\xymatrix{{
\cZ_\chi = \{ (\tilde v, \hat c, k) \in \widetilde V_\chi \times
(\Cartan^{reg} / W_\chi) \times \bC \; | \; f_\chi (\tilde v) = k \, \hat c \}}
\ar[rr]^-{F} \ar[rd] && \Cartan^{reg} / W_\chi \times \bC \ar[dl] \\& \Cartan^{reg} /
W_\chi \, ,}
\eeqn
where $F (\tilde v, \hat c, k) = (\hat c, k)$.  Let $F_2 : \cZ_\chi \to \bC$ be the
second component of $F$, and let $\cZ_\chi^{rs} = F_2^{-1} (\bC^*)$.  
By abuse of notation, we denote the pull-back of $\hat \cL_\chi$ from
$\widetilde V_\chi^{rs}$ to $\cZ_\chi^{rs}$ by the same symbol $\hat \cL_\chi$.
Let $\cZ_{\chi, 0} = F_2^{-1} (0) = X_0 \times (\Cartan^{reg} / W_\chi)$, and
consider the nearby cycle sheaf:
\beqn
\cP_\chi = \psi_{F_2} (\hat \cL_\chi [-]) \in \on{Perv}_G (\cZ_{\chi, 0}).
\eeqn
The stratification $\cS_0$ of $X_0$, which was fixed in Section
\ref{subsec-dual-rep}, following Proposition \ref{prop-A-f-stratification},
induces a stratification $\cS_{\chi, 0}$ of $\cZ_{\chi, 0}$, given by:
\beqn
\cZ_{\chi, 0} = \bigcup_{S \in \cS_0} S \times (\Cartan^{reg} / W_\chi).
\eeqn

\begin{prop}\label{prop-A-F2}
For every $S \in \cS_0$, the pair of manifolds:
\beqn
(\cZ^{rs}, S \times (\Cartan^{reg} / W_\chi)),
\eeqn
satisfies Thom's $A_{F_2}$ condition.
\end{prop}

\begin{proof}
This is similar to the proof of \cite[Proposition 3.1]{GVX1}, which can be readily
adapted to the present setting.
\end{proof}

Let $\on{Perv}_G (\cZ_{\chi, 0}, \cS_{\chi, 0})$ denote the category of
$G$-equivariant perverse sheaves on $\cZ_{\chi, 0}$, constructible with
respect to $\cS_{\chi, 0}$.  Proposition \ref{prop-A-F2} and
\cite[Theorem 5.5]{Gi} imply that:
\beq\label{eqn-constructible-2}
\cP_\chi \in \on{Perv}_G (\cZ_{\chi, 0}, \cS_{\chi, 0}), 
\eeq
(this is similar to the proof of \cite[Corollary 3.2]{GVX1}).  For every $\hat c \in
\Cartan^{reg} / W_\chi$, let $j_{\hat c} : X_0 \to \cZ_{\chi, 0}$ be the inclusion
$x \mapsto (x, \hat c, 0)$.  In view of \eqref{eqn-constructible-2}, this
inclusion gives rise to a perverse (i.e., properly shifted) restriction functor:
\beqn
j^*_{\hat c} : \on{Perv}_G (\cZ_{\chi, 0}, \cS_{\chi, 0}) \to
\on{Perv}_G (X_0, \cS_0).
\eeqn
For each $\hat c \in \Cartan^{reg} / W_\chi$, we will write:
\beqn
\cP_{\hat c} = j^*_{\hat c} (\cP) \in \on{Perv}_G (X_0, \cS_0).
\eeqn
In view of equation \eqref{eqn-hat-no-hat} and Proposition \ref{prop-A-F2}, 
we have:
\beqn
\cP_{\hat c_0} \cong P_\chi \, .
\eeqn 
The sheaves $\{ \cP_{\hat c} \in \on{Perv}_G (X_0, \cS_0) \}$ form a local
system over $\Cartan^{reg} / W_\chi \ni \hat c$, giving rise to the monodromy
representation $\mu$ of equation \eqref{eqn-mu}.  Note that the construction
of $\mu$ depends on the choice of the regular splitting $\tilde r : B_W \to
\widetilde B_W$ in \eqref{eqn-tilde-r}.

\section{Monodromy in the family for a braid generator}
\label{sec-monodromy}

Fix an $\alpha \in A$ and a $\sigma_\alpha \in B_W [\alpha]$.  In this section,
we discuss the monodromy in the family operator
$\mu (\sigma_\alpha^\degalpha) \in \on{Aut} (P_\chi)$ (see equations
\eqref{eqn-W-alpha-chi} and \eqref{eqn-mu}), and relate it to the corresponding
construction for the rank one representation $G_\alpha | V_\alpha$.

Recall the splitting homomorphism $\tilde r [\sigma_\alpha] : B_{W_\alpha} \to 
\widetilde B_{W_\alpha}$ of equation \eqref{eqn-tilde-r-sigma-alpha}, for the 
rank one representation $G_\alpha | V_\alpha$.  Recall that $\chi_\alpha =
\chi |_{I_\alpha} \in \hat I_\alpha$, and $p_\alpha$, $q_\alpha$, $\tilde p_\alpha$,
$\tilde q_\alpha$ are the maps of diagram \eqref{eqn-main-diagram-alpha}.  Let:
\beqn
W_{\alpha, \chi^{}_\alpha} \coloneqq \on{Stab}_{W_\alpha} (\chi_\alpha), \;\;
B_{W_\alpha}^{\chi_\alpha} \coloneqq p_\alpha^{-1} (W_{\alpha, \chi^{}_\alpha}) 
\subset B_{W_\alpha} \, , \;\;
\widetilde B_{W_\alpha}^{\chi_\alpha} \coloneqq \tilde q_\alpha^{-1}
(B_{W_\alpha}^{\chi_\alpha}) \subset \widetilde B_{W_\alpha} \, ,
\eeqn
(cf. \eqref{eqn-W-chi} and \eqref{eqn-tilde-B-chi}).  By analogy with \eqref{eqn-mu},
the data $(G_\alpha | V_\alpha, \chi_\alpha, \tilde r [\sigma_\alpha])$ gives rise to
a monodromy in the family representation:
\beqn
\mu [\sigma_\alpha] : B_{W_\alpha}^{\chi_\alpha} \to \on{Aut}
(P_{\chi^{}_\alpha}).
\eeqn
Here, we include $\sigma_\alpha$ in the notation, because the construction
of $\mu [\sigma_\alpha]$ uses the splitting homomorphism
$\tilde r [\sigma_\alpha]$.  However, we have the following proposition.
Recall the subgroup $W_{\alpha, \chi} = W_\alpha \cap W_\chi \subset
W_{\alpha, \chi^{}_\alpha}$ (see equation \eqref{eqn-W-alpha-chi} and
Remark \ref{rmk-W-alpha-chi-one}), and let:
\beqn
B_{W_\alpha}^\chi \coloneqq p_\alpha^{-1} (W_{\alpha, \chi}) 
\subset B_{W_\alpha}^{\chi_\alpha} \, , \;\;
\widetilde B_{W_\alpha}^\chi \coloneqq \tilde q_\alpha^{-1}
(B_{W_\alpha}^\chi) \subset \widetilde B_{W_\alpha}^{\chi_\alpha} \, .
\eeqn

\begin{prop}\label{prop-mu-alpha}
Write:
\beq\label{eqn-mu-alpha}
\mu_\alpha : B_{W_\alpha}^\chi \to \on{Aut} (P_{\chi^{}_\alpha}),
\eeq
for the restriction of the homomorphism $\mu [\sigma_\alpha]$ to the subgroup
$B_{W_\alpha}^\chi \subset B_{W_\alpha}^{\chi_\alpha}$.  The homomorphism
$\mu_\alpha$ is independent of the choice of $\sigma_\alpha \in B_W [\alpha]$.
\end{prop}

\begin{proof}
Consider the following Cartesian commutative diagram:
\beq\label{diagram-V-tilde-alpha}
\begin{CD}
\widetilde V_{\alpha, \chi} @>>>
\widetilde V_{\alpha, \chi^{}_\alpha} @>>> V_\alpha
\\
@V{f_{\alpha, \chi}}VV @VV{f_{\alpha, \chi^{}_\alpha}}V @VV{f_\alpha}V
\\
\Cartan / W_{\alpha, \chi} @>>> \Cartan / W_{\alpha, \chi^{}_\alpha}
@>>> \Cartan / W_\alpha \, .
\end{CD}
\eeq
Here, the right square is the analog of diagram \eqref{diagram-V-tilde-chi}
for the data $(G_\alpha | V_\alpha, \chi_\alpha)$, and the left square is
obtained by further pulling back the family $f_{\alpha, \chi^{}_\alpha}$ to
$\Cartan / W_{\alpha, \chi}$.  Let:
\beqn
\widetilde V_{\alpha, \chi^{}_\alpha}^{rs} = f_{\alpha, \chi^{}_\alpha}^{-1}
((\Cartan - \Cartan_\alpha) / W_{\alpha, \chi^{}_\alpha}), \;\;\;\;
\widetilde V_{\alpha, \chi}^{rs} = (f_{\alpha, \chi})^{-1}
((\Cartan - \Cartan_\alpha) / W_{\alpha, \chi}).
\eeqn
We have:
\beq\label{eqn-pi-one-rs}
\pi_1^{G_\alpha} (\widetilde V_{\alpha, \chi^{}_\alpha}^{rs} , c_0) =
\widetilde B_{W_\alpha}^{\chi_\alpha} \, , \;\;\;\;
\pi_1^{G_\alpha} (\widetilde V_{\alpha, \chi}^{rs}, c_0) =
\widetilde B_{W_\alpha}^\chi \subset
\widetilde B_{W_\alpha}^{\chi_\alpha} \, ,
\eeq
where the basepoint $c_0 \in \Cartan^{reg}$ naturally determines basepoints
in $\widetilde V_{\alpha, \chi^{}_\alpha}^{rs}$ and
$\widetilde V_{\alpha, \chi}^{rs}$, as in \eqref{eqn-tilde-c-zero}.

By analogy with the character $\hat \chi$ of equation \eqref{eqn-hat-chi},
we obtain a character $\hat \chi_\alpha : \widetilde B_{W_\alpha}^{\chi_\alpha} 
\to \bG_m$.  More precisely, the character $\hat \chi_\alpha$ is determined by 
requiring that $\hat \chi_\alpha |_{I_\alpha} = \chi_\alpha$ and:
\beq\label{eqn-hat-chi-alpha}
\hat \chi_\alpha \circ \tilde r [\sigma_\alpha] : B_{W_\alpha}^{\chi_\alpha}
\to \bG_m \;\; \text{is the trivial character}.
\eeq
Recall the group $\widetilde B_{W_\alpha}^{\rm f} \supset
\widetilde B_{W_\alpha}$ of equations
\eqref{eqn-tilde-B-W-alpha-f}-\eqref{eqn-tilde-W-alpha-f}
and diagram \eqref{eqn-main-diagram-alpha-full}.  Define:
\beq\label{eqn-pi-one-rs-full}
\widetilde B_{W_\alpha}^{\chi_\alpha, {\rm f}} \coloneqq
(\tilde q_\alpha^{\rm f})^{-1}
(B_{W_\alpha}^{\chi_\alpha}) \subset
\widetilde B_{W_\alpha}^{\rm f} \, , \;\;\;\;
\widetilde B_{W_\alpha}^{\chi, {\rm f}} \coloneqq
(\tilde q_\alpha^{\rm f})^{-1}
(B_{W_\alpha}^\chi) \subset
\widetilde B_{W_\alpha}^{\chi_\alpha, {\rm f}} \, .
\eeq
The groups \eqref{eqn-pi-one-rs} and \eqref{eqn-pi-one-rs-full} can be
organized into a diagram of inclusions as follows:
\beq\label{diagram-pi-one-rs}
\begin{CD}
\widetilde B_{W_\alpha}^\chi @>>>
\widetilde B_{W_\alpha}^{\chi_\alpha}
\\
@VVV @VVV
\\
\widetilde B_{W_\alpha}^{\chi, {\rm f}} @>>>
\widetilde B_{W_\alpha}^{\chi_\alpha, {\rm f}} \, .
\end{CD}
\eeq

The character $\hat \chi_\alpha : \widetilde B_{W_\alpha}^{\chi_\alpha} 
\to \bG_m$ does not, in general, extend to a 
character of $\widetilde B_{W_\alpha}^{\chi_\alpha, {\rm f}} \supset
\widetilde B_{W_\alpha}^{\chi_\alpha}$.  However, there is a unique
character $\hat \chi_\alpha^{\rm f} : \widetilde B_{W_\alpha}^{\chi, {\rm f}}
\to \bG_m$, such that $\hat \chi_\alpha^{\rm f} |_I = \chi$ and:
\beq\label{eqn-hat-chi-alpha-f}
\hat \chi_\alpha^{\rm f} \circ \tilde r [\sigma_\alpha] : B_{W_\alpha}^\chi
\to \bG_m \;\; \text{is the trivial character}.
\eeq
By Lemma \ref{lemma-conjugate}, condition \eqref{eqn-hat-chi-alpha-f}
is independent of the choice of $\sigma_\alpha \in B_W [\alpha]$.
Therefore, the character $\hat \chi_\alpha^{\rm f}$ is independent of
$\sigma_\alpha$.  By construction, the characters $\hat \chi_\alpha$
and $\hat \chi_\alpha^{\rm f}$ match on the upper left corner of
diagram \eqref{diagram-pi-one-rs}, i.e., we have:
\beq\label{eqn-restrict-characters}
\hat \chi_\alpha |_{\widetilde B_{W_\alpha}^\chi} =
\hat \chi_\alpha^{\rm f} |_{\widetilde B_{W_\alpha}^\chi} \, .
\eeq
Therefore, the LHS of \eqref{eqn-restrict-characters} is independent of
$\sigma_\alpha$.  It remains to note that the monodromy action
$\mu_\alpha$ can be defined using the family $f_{\alpha, \chi}$ and
the character $\hat \chi_\alpha |_{\widetilde B_{W_\alpha}^\chi}$,
in place of the family $f_{\alpha, \chi^{}_\alpha}$ and the full character
$\hat \chi_\alpha$ (see diagram \eqref{diagram-V-tilde-alpha} and
equation \eqref{eqn-pi-one-rs}).
\end{proof}

The following proposition relates the monodromy actions $\mu$ and
$\mu_\alpha$ of equations \eqref{eqn-mu} and \eqref{eqn-mu-alpha}.
Recall that $s_\alpha \in W_\alpha \cong \bZ / n_\alpha$ and
$\sigma \in B_{W_\alpha} \cong \bZ$ are the counter-clockwise generators
(see Sections \ref{subsec-generators}, \ref{subsec-min-poly}).
Recall also that $W_{\alpha, \chi} = \langle s_\alpha^\degalpha \rangle$
(see \eqref{eqn-W-alpha-chi}) and therefore $B_{W_\alpha}^\chi =
\langle \sigma^\degalpha \rangle \subset B_{W_\alpha} = 
\langle \sigma \rangle$.

\begin{definition}\label{defn-bar-R-chi-alpha-mu}
Let $\bar R_{\chi, \alpha}^\mu \in \cR$ be the minimal polynomial of 
$\mu_\alpha (\sigma^{\degalpha}) \in \on{End} (P_{\chi^{}_\alpha})$.
\end{definition}

\begin{prop}\label{prop-family-min-poly}
The minimal polynomial of $\mu (\sigma_\alpha^{\degalpha}) \in
\on{End} (P_\chi)$ is equal to $\bar R_{\chi, \alpha}^\mu$.
\end{prop}

\begin{proof}
This is similar to the proofs of \cite[Propositions 6.14 $\&$ 7.7]{GVX1} and
\cite[Theorem 5.2]{Gr1}.  We briefly outline the argument.  Let $\sigma_\alpha =
\sigma_\alpha [\Gamma]$, with $\Gamma (0) = c_0$ and $\Gamma (1) =
c_{\alpha, 1}$, as in equation \eqref{eqn-sigma-alpha}.  Without loss of
generality, we can assume that $c_0 = c_{\alpha, 1}$ and the path $\Gamma$
is trivial.  Recall the point $c_\alpha \in \Cartan_\alpha^{reg}$ defined by
equation \eqref{eqn-c-alpha}.  The first main idea of the proof is to decompose
the process of specializing the regular value $\hat c_0 \in \Cartan^{reg} / W_\chi$
of $f_\chi$ to $0 \in \Cartan / W_\chi$ into two steps as follows:
\beq\label{eqn-schematic}
\hat c_0 \longrightarrow 0 \;\; \iff \;\;
\hat c_0 \longrightarrow \hat c_\alpha \longrightarrow 0,
\eeq
where $\hat c_\alpha$ is the image of $c_\alpha$ in $\Cartan / W_\chi$.

Equation \eqref{eqn-schematic} is meant to be schematic.  To make
the idea of this equation precise, let $\rD_\gamma = \{ z \in \bC \; | \;
|z| < 2 \}$, and let $\gamma : \rD_\gamma \to \Cartan$ be the analytic
arc defined by:
\beqn
\gamma (z) = c_\alpha + z \cdot (c_0 - c_\alpha).
\eeqn
Next, let $\rD_{\bar\gamma} = \{ z \in \bC \; | \; |z| < 2^{n_\alpha / \degalpha} \}$,
and let $\bar\gamma : \rD_{\bar\gamma} \to \Cartan / W_\chi$ be the unique
analytic arc such that:
\beqn
f_\chi \circ \gamma \, (z) = \bar\gamma (z^{n_\alpha / \degalpha})
\;\; \text{for every} \;\;
z \in \rD_\gamma \, .
\eeqn
We base change the family $f_\chi: \widetilde V_\chi \to \Cartan / W_\chi$ to
$\rD_{\bar\gamma}$, to obtain a family:
\beqn
f_{\bar\gamma} : \widetilde V_{\bar\gamma} =
\rD_{\bar\gamma} \times_{\Cartan / W_\chi} \widetilde V_\chi
\to \rD_{\bar\gamma} \, .
\eeqn
Let $\rD_{\bar\gamma}^* = \rD_{\bar\gamma} \cap \bC^*$, and note
that $\bar\gamma (\rD_{\bar\gamma}^*) \subset \Cartan^{reg} / W_\chi$
(this is a consequence of choosing $c_0 = c_{\alpha, 1}$ to be near
$c_\alpha \in \Cartan_\alpha$).
Let $\widetilde V_{\bar\gamma}^{rs} = f_{\bar\gamma}^{-1}
(\rD_{\bar\gamma}^*)$, note that we have a projection
$\widetilde V_{\bar\gamma}^{rs} \to \widetilde V_\chi^{rs}$, and let
$\hat \cL_{\chi, \bar\gamma}$ be the pull-back of $\hat \cL_\chi$ from
$\widetilde V_\chi^{rs}$ to $\widetilde V_{\bar\gamma}^{rs}$.  Let:
\beqn
\bar c_\alpha = f (c_\alpha), \;\;
X_{\bar c_\alpha} = f^{-1} (\bar c_\alpha) = f_\chi^{-1} (\hat c_\alpha) =
f_{\bar\gamma}^{-1} (0),
\eeqn
(cf. equation \eqref{eqn-bar-c-alpha}).  Form the nearby cycle sheaf:
\beqn
P_{\bar\gamma} = \psi_{f_{\bar\gamma}} (\hat \cL_{\chi, \bar\gamma} [-])
\in \on{Perv}_G (X_{\bar c_\alpha}),
\eeqn
which, as usual, we make perverse by an appropriate shift, and let
$\mu_{\bar\gamma} : P_{\bar\gamma} \to P_{\bar\gamma}$ be the
associated (counter-clockwise) monodromy.  The pair $(P_{\bar\gamma},
\mu_{\bar\gamma})$ encapsulates the first step of the specialization
process in the RHS of \eqref{eqn-schematic}.

For the second step, consider the functor:
\beq\label{eqn-psi-bar-c-s}
\psi_f [\bar c_\alpha] : \on{Perv}_G (X_{\bar c_\alpha}) \to \on{Perv}_G (X_0),
\eeq
defined in the same manner as the functor $\psi_f [\bar c_0]$ of
equation \eqref{eqn-P-chi}, now using the family:
\beqn
\cZ_{\bar c_\alpha} = \{ (x, k) \in V \times \bC \; | \; f (x) = k \, \bar c_\alpha \} 
\to \bC,
\eeqn
in place of the family \eqref{eqn-family-Z}.  We claim that:
\beq
\label{eqn-factorization}
\psi_f [\bar c_\alpha] (P_{\bar\gamma}) \cong P_\chi
\;\;\;\; \text{and} \;\;\;\;
\psi_f [\bar c_\alpha] (\mu_{\bar\gamma}) = \mu (\sigma_\alpha^\degalpha).
\eeq
A proof of this claim is analogous to the proof of \cite[Equation (6.4)]{GVX1}
(see proof of \cite[Proposition 6.14]{GVX1}).

Note that the functor $\psi_f [\bar c_\alpha]$ of equation \eqref{eqn-psi-bar-c-s}
is exact and faithful.  Therefore, in view of \eqref{eqn-factorization}, it suffices to
show that the minimal polynomial of $\mu_{\bar\gamma} \in \on{End}
(P_{\bar\gamma})$ is equal to $\bar R_{\chi, \alpha}^\mu$.  The second
main idea of the proof is to recall (see \eqref{eqn-normal-slice}) that
$V_\alpha \subset V$ is a normal slice to the orbit $G \cdot c_\alpha$
through $c_\alpha$, and therefore, the minimal polynomial of
$\mu_{\bar\gamma}$ can be computed by focusing on the intersection
$X_{\bar c_\alpha} \cap V_\alpha$ near the point $c_\alpha$.

To make the above precise, we reuse some of the construction of the proof of
Proposition \ref{prop-reg-split}.  Recall the fiber $X_{\breve c_\alpha, \alpha} =
f_\alpha^{-1} (f_\alpha (c_\alpha))$ of equation \eqref{eqn-breve-c-alpha},
which is the connected component of $X_{\bar c_\alpha} \cap V_\alpha$
containing $c_\alpha$.  By analogy with \eqref{eqn-transverse}, we have:
\beq\label{eqn-transverse-again}
\text{for every $x \in X_{\breve c_\alpha, \alpha}$, the intersection
$G \cdot x \cap V_\alpha$ is transverse at $x$.}
\eeq
Write $j : X_{\breve c_\alpha, \alpha} \to X_{\bar c_\alpha}$ for the inclusion
map.  By \eqref{eqn-transverse-again}, there is a well-defined perverse
(i.e., properly shifted) restriction functor:
\beqn
j^* : \on{Perv}_G (X_{\bar c_\alpha}) \to
\on{Perv}_{G_\alpha} (X_{\breve c_\alpha, \alpha}).
\eeqn
Since $G \cdot c_\alpha$ is the unique closed $G$-orbit in $X_{\bar c_\alpha}$,
the functor $j^*$ is faithful.

Recall the zero-fiber $X_{0, \alpha} = f_\alpha^{-1} (0)$ and note that parallel
translation by $c_\alpha$ takes $X_{0,\alpha}$ into $X_{\breve c_\alpha, \alpha}$,
establishing an equivalence of categories:
\beqn
T_{c_\alpha}^* :
\on{Perv}_{G_\alpha} (X_{\breve c_\alpha, \alpha})
\stackrel\simeq\longrightarrow
\on{Perv}_{G_\alpha} (X_{0, \alpha}).
\eeqn
By chasing the definitions, and using the transversality assertion
\eqref{eqn-transverse-again}, one can check that:
\beqn
T_{c_\alpha}^* \circ j^* \, (P_{\bar\gamma}) \cong P_{\chi^{}_\alpha}
\;\;\;\; \text{and} \;\;\;\;
T_{c_\alpha}^* \circ j^* \, (\mu_{\bar\gamma}) = \mu_\alpha (\sigma^\degalpha).
\eeqn
Since the composition $T_{c_\alpha}^* \circ j^*$ is faithful, the monodromy
transformations $\mu_{\bar\gamma}$ and $\mu_\alpha (\sigma^\degalpha)$
have the same minimal polynomial.
\end{proof}

By analogy with Proposition-Definition \ref{prop-def-R-chi-alpha} (ii) and
assertion \eqref{eqn-invariance-bar}, we have the following corollary of
Proposition \ref{prop-family-min-poly}.

\begin{corollary}\label{cor-invariance-mu}
For every $\alpha_1, \alpha_2 \in A$ and $w \in W_\chi$, with
$\alpha_2 = w \cdot \alpha_1$, we have $\bar R_{\chi, \alpha_1}^\mu =
\bar R_{\chi, \alpha_2}^\mu$.
\end{corollary}

\begin{proof}
Pick a pair of braid generators $\sigma_1 \in B_W [\alpha_1]$,
$\sigma_2 \in B_W [\alpha_2]$, and let $e = e_{\alpha_1} = e_{\alpha_2}$
(see assertion \eqref{eqn-assignment-degalpha}).  By
\eqref{eqn-conjugate} and \eqref{eqn-conjugate-two}, the elements
$\sigma_1^e, \sigma_2^e \in B_W^\chi$ are conjugate to each other.
The Corollary follows by Proposition \ref{prop-family-min-poly}.
\end{proof}

Write:
\beq\label{eqn-mu-morse}
\mu_{l_0} : B_W^\chi \to \on{Aut} (M_{l_0} (P_\chi)) \, ,
\eeq
for the action induced by the monodromy in the family \eqref{eqn-mu}.
Note that:
\beq\label{eqn-commute}
\text{the actions $\lambda_{l_0}$ and $\mu_{l_0}$ of equations
\eqref{eqn-lambda-grp} and \eqref{eqn-mu-morse} commute with
each other.}
\eeq
Let $W / W_\chi$ be the set of left cosets, and write $\barw = w \, W_\chi \in
W / W_\chi$ for $w \in W$.  For every $\barw \in W / W_\chi$, let:
\beq\label{eqn-M-l-zero-barw}
M_{l_0} (P_\chi) [\barw] \subset M_{l_0} (P_\chi),
\eeq
be the linear span of all Picard-Lefschetz classes $u = \PL [c, \gamma, o, a]
\in M_{l_0} (P_\chi)$, with $c = w_1 \, c_0 \in Z_{l_0}$ and $w_1 \in \barw$.

\begin{prop}\label{prop-decomp}
We have a direct sum decomposition:
\beq\label{eqn-decomp}
M_{l_0} (P_\chi) = \bigoplus_{\barw \in W / W_\chi} M_{l_0} (P_\chi) [\barw],
\eeq
which is invariant under the monodromy action $\mu_{l_0}$.  For every 
$\barw \in W / W_\chi$, we have:
\beqn
\dim M_{l_0} (P_\chi) [\barw] = |W_\chi|.
\eeqn
\end{prop}

\begin{proof}
This follows from Corollary \ref{cor-dim-M}, Proposition
\ref{prop-microlocal-inertia}, and assertion \eqref{eqn-commute}.
Decomposition \eqref{eqn-decomp} is just the decomposition by the
characters of $I$, under the microlocal monodromy action $\lambda_{l_0}$
of equation \eqref{eqn-lambda-grp}.
\end{proof}

\begin{prop}\label{prop-degree-bound}
We have:
\beqn
\deg \bar R_{\chi, \alpha}^\mu \leq n_\alpha / \degalpha.
\eeqn
\end{prop}

\begin{proof}
Using Proposition \ref{prop-family-min-poly}, we interpret
$\bar R_{\chi, \alpha}^\mu$ as the minimal polynomial of
$\mu (\sigma_\alpha^\degalpha) \in \on{End} (P_\chi)$.
Next, using Proposition \ref{prop-fourier}, we interpret
$\bar R_{\chi, \alpha}^\mu$ as the minimal polynomial of
$\mu_{l_0} (\sigma_\alpha^\degalpha) \in \on{End} (M_{l_0} (P_\chi))$.
By Proposition \ref{prop-cyclic-vector} and assertion \eqref{eqn-commute},
it suffices to show that:
\beq\label{eqn-dim-bound}
\dim \bC [\mu_{l_0} (\sigma_\alpha^\degalpha)] \cdot u_0 \leq
n_\alpha / \degalpha.
\eeq

Let $\sigma_\alpha = \sigma_\alpha [\Gamma]$, with $\Gamma (0) = c_0$
and $\Gamma (1) = c_{\alpha, 1}$, as in equation \eqref{eqn-sigma-alpha}.
By moving the point $c_0$ along the path $\Gamma$, while maintaining
the relation $l_0 = \nu (c_0) \in (\Cartan^*)^{reg}$ (see equations
\eqref{eqn-l-zero}-\eqref{eqn-dual-cartan-reg}), we can reduce inequality
\eqref{eqn-dim-bound} to the case where $c_0 = c_{\alpha, 1}$ and the
path $\Gamma$ is trivial.  Proceeding with this assumption, write:
\beq\label{eqn-Z-alpha}
Z_\alpha = W_\alpha \cdot c_0 \subset Z_{l_0} = W \cdot c_0 \, ,
\eeq
(cf. equation \eqref{eqn-Z-l}).  By the choice of the basepoints
$c_0 \in \Cartan^{reg}$ and $l_0 \in (\Cartan^*)^{reg}$,
we have:
\beq\label{eqn-horizontal-separation}
\xi_{l_0} (\cOne) > \xi_{l_0} (\cTwo) \;\; \text{for every} \;\;
\cOne \in Z_\alpha \;\; \text{and} \;\; \cTwo \in Z_{l_0} - Z_\alpha \, ,
\eeq
where $\xi_{l_0} = \on{Re} (l_0) : V \to \bR$, as in Section \ref{sec-fourier}.
Let:
\beq\label{eqn-M-alpha-declare}
M_{l_0} (P_\chi) [\alpha] \subset M_{l_0} (P_\chi),
\eeq
be the linear span of all Picard-Lefschtez classes $u = \PL [\cOne, \gamma, o, a]$,
with $\cOne \in Z_\alpha$ and the path $\gamma$ satisfying:
\beq\label{eqn-M-alpha}
\xi_{l_0} \circ \gamma \, (t) > \xi_{l_0} (\cTwo) \;\; \text{for every} \;\;
t \in [0, 1] \;\; \text{and} \;\; \cTwo \in Z_{l_0} - Z_\alpha.
\eeq
Note that we have:
\beq\label{eqn-dim-M-P-chi-alpha}
\dim M_{l_0} (P_\chi) [\alpha] = |Z_\alpha| = |W_\alpha| = n_\alpha \, .
\eeq

Inequality \eqref{eqn-horizontal-separation} admits a parametrized version,
as the point $c_0$ traces out the path $\Gamma_\alpha [c_{\alpha, 1}] :
[0, 1] \to \Cartan^{reg}$ of equation \eqref{eqn-gamma-alpha}.  Namely,
for every $t \in [0, 1]$, let:
\beqn
Z_{l_0, t} = W \cdot \Gamma_\alpha [c_{\alpha, 1}] \, (t)
\;\;\;\; \text{and} \;\;\;\;
Z_{\alpha, t} = W_\alpha \cdot \Gamma_\alpha [c_{\alpha, 1}] \, (t).
\eeqn
Then we have:
\beqn
\xi_{l_0} (\cOne) > \xi_{l_0} (\cTwo) \;\; \text{for every} \;\;
\cOne \in Z_{\alpha, t} \;\; \text{and} \;\; \cTwo \in Z_{l_0, t} - Z_{\alpha, t} \, .
\eeqn
By a standard Picard-Lefschetz theory analysis, it follows that:
\beq\label{eqn-preserve-M-alpha}
\mu_{l_0} (\sigma_\alpha^\degalpha) \, (M_{l_0} (P_\chi) [\alpha]) =
M_{l_0} (P_\chi) [\alpha].
\eeq

For each $i \in \{ 0, \dots, \degalpha - 1 \}$, let $w_i = s_\alpha^i \in W_\alpha$,
consider the coset $\overline {w_i} = w_i \, W_\chi \in W / W_\chi$,
and form the intersection:
\beqn
M_{l_0} (P_\chi) [\alpha, i] = M_{l_0} (P_\chi) [\alpha] \cap
M_{l_0} (P_\chi) [\, \overline {w_i} \,].
\eeqn
Note that $W_\alpha \cap \, \overline {w_i} \, = |W_{\alpha, \chi}| =
n_\alpha / \degalpha$.  It follows that:
\beq\label{eqn-dim-M-P-chi-alpha-i}
\dim M_{l_0} (P_\chi) [\alpha, i] \geq n_\alpha / \degalpha \, .
\eeq
Proposition \ref{prop-decomp}, equations
\eqref{eqn-dim-M-P-chi-alpha}-\eqref{eqn-preserve-M-alpha},
and inequality \eqref{eqn-dim-M-P-chi-alpha-i} imply that:
\beqn
\dim M_{l_0} (P_\chi) [\alpha, i] = n_\alpha / \degalpha
\;\;\;\; \text{and} \;\;\;\;
\mu_{l_0} (\sigma_\alpha^\degalpha) \, (M_{l_0} (P_\chi) [\alpha, i]) =
M_{l_0} (P_\chi) [\alpha, i],
\eeqn
for every $i \in \{ 0, \dots, \degalpha - 1 \}$.  Since we have $u_0 \in
M_{l_0} (P_\chi) [\alpha, 0]$, inequality \eqref{eqn-dim-bound} follows.
\end{proof}

In the next section, we will prove the following sharpening of Proposition
\ref{prop-degree-bound}.

\begin{prop}\label{prop-R-mu-degree}
We have:
\beqn
\deg \bar R_{\chi, \alpha}^\mu = n_\alpha / \degalpha.
\eeqn
\end{prop}

\section{Rank one representations and the simple carousel}
\label{sec-carousel}

In this section, we prove Propositions \ref{prop-monic},
\ref{prop-R-chi-alpha-at-zero}, \ref{prop-R-alpha-vanishing}, and
\ref{prop-R-mu-degree}.  We also state and prove Proposition
\ref{prop-carousel}, relating the actions of the microlocal monodromy
and the monodromy in the family for the rank one representation
$G_\alpha | V_\alpha$, $\alpha \in A$.

Fix an $\alpha \in A$.  Recall the microlocal monodromy action:
\beqn
\lambda_{l_{0, \alpha}} : \widetilde B_{W_\alpha} \to 
\on{Aut} (M_{l_{0, \alpha}} (P_{\chi^{}_\alpha})),
\eeqn
of equation \eqref{eqn-lambda-grp-rank-one}.  Write:
\beq\label{eqn-mu-l-zero-alpha}
\mu_{l_{0, \alpha}} : B_{W_\alpha}^\chi \to \on{Aut} (M_{l_{0, \alpha}}
(P_{\chi^{}_\alpha})),
\eeq
for the action induced by the monodromy in the family $\mu_\alpha$ of
Proposition \ref{prop-mu-alpha}.  Also, recall that a braid generator
$\sigma_\alpha \in B_W [\alpha]$ gives rise to a splitting homomorphism
$\tilde r [\sigma_\alpha] : B_{W_\alpha} \to \widetilde B_{W_\alpha}$,
as in \eqref{eqn-tilde-r-sigma-alpha}-\eqref{eqn-r-sigma-alpha}.

\begin{prop}\label{prop-carousel}
There exists a sign $k_\alpha \in \{ \pm 1 \}$, such that for every
$\sigma_\alpha \in B_W [\alpha]$, we have:
\beq\label{eqn-prop-carousel}
\mu_{l_{0, \alpha}} (\sigma^{\degalpha}) = k_\alpha \, \cdot \,
\lambda_{l_{0, \alpha}} \circ \, \tilde r [\sigma_\alpha] \, (\sigma^{-\degalpha})
\in \on{Aut} (M_{l_{0, \alpha}} (P_{\chi^{}_\alpha})).
\eeq
\end{prop}

The appearance of the inverse in the RHS of \eqref{eqn-prop-carousel}
is related to Remark \ref{rmk-reversal}.  The proofs of Propositions
\ref{prop-monic}, \ref{prop-R-chi-alpha-at-zero}, \ref{prop-R-alpha-vanishing},
\ref{prop-R-mu-degree}, \ref{prop-carousel} use a variant of the carousel
technique of L\^e from singularity theory (see Remark \ref{rmk-carousel}).
For all five propositions, we can assume, without loss of generality, that the
basepoint $c_0 \in \Cartan^{reg}$ is located near the hyperplane
$\Cartan_\alpha \subset \Cartan$.  More precisely, we pick points
$c_\alpha \in \Cartan_\alpha^{reg}$ and $c_{\alpha, 1} \in \Cartan^{reg}$
as in equation \eqref{eqn-c-alpha}, and we assume that:
\beq\label{eqn-assume-c-zero}
c_0 = c_{\alpha, 1} \, .
\eeq
Recall that we write $\breve c_0 = f_\alpha (c_0) \in Q_\alpha^{reg}$ and
$X_{\breve c_0, \alpha} = f_\alpha^{-1} (\breve c_0)$ (see equation
\eqref{eqn-X-Q-alpha}).  Let $Z_\alpha = W_\alpha \cdot c_0$ be the set of
critical points of the restriction $l_{0, \alpha} |_{X_{\breve c_0, \alpha}}$ (as
in equation \eqref{eqn-Z-alpha}).  The critical values $l_{0, \alpha} (Z_\alpha)
\subset \bC$ appear in the complex plane as the vertices of a regular
$n_\alpha$-gon, centered on $l_{0, \alpha} (c_\alpha) \in \bC$ (see equation
\eqref{eqn-c-alpha}).  The term ``carousel'', in the present instance, refers to
the collective movement of these critical values, as we vary the basepoints
$\breve c_0 \in Q_\alpha^{reg}$ or $l_{0, \alpha} \in (V_\alpha^*)^{rs}$.

\begin{remark}\label{rmk-carousel}{\em
The carousel technique was introduced in \cite{Le} and further developed in
\cite{Ti}.  See also \cite{Mas} and its references.  We point out one distinction
between the arguments in \cite{Le} and in this subsection.  In \cite{Le},
the focus is on the homology of the Milnor fiber $F_{f, 0}$ of a polynomial $f$
at the origin, relative to the intersection of $F_{f, 0}$ with a generic hyperplane
passing through the origin.  In this subsection, the focus is on the homology of
a Milnor fiber at the origin, relative to the intersection with a generic hyperplane
passing near, but not through, the origin.  Under suitable conditions, these two
relative homology groups are isomorphic to each other.  However, monodromy
in the family acts very differently on the two, and only the latter is functorial in
the nearby cycles of $f$.  A special case of the argument of this section can be
found in \cite[Section 7]{BG}.}
\end{remark}

Recall that we write $d_\alpha = \dim X_{\breve c_0, \alpha}$
(see \eqref{eqn-d-alpha}), and recall the rank one $G_\alpha$-equivariant local
system $\cL_{\chi^{}_\alpha}$ on $X_{\breve c_0, \alpha}$, corresponding to the
character $\chi_\alpha \in \hat I_\alpha$ (see Section \ref{subsec-min-poly}).
By Lemma \ref{lemma-describe-morse-group}, applied to the representation
$G_\alpha | V_\alpha$, we have:
\beq\label{eqn-describe-morse-group-rank-one}
M_{l_{0, \alpha}} (P_{\chi^{}_\alpha}) \cong
H_{d_\alpha} (X_{\breve c_0, \alpha}, \{ x \in X_{\breve c_0, \alpha} \; | \;
\xi_{l_0} (x) \geq \xi_0 \}; \cL_{\chi^{}_\alpha}),
\eeq
where $\xi_0$ is any real number with $\xi_0 > l_{0, \alpha} (c_0) =
\langle c_0, c_0 \rangle$ (see equation \eqref{eqn-l-zero}).  
Pick a braid generator:
\beq\label{eqn-fix-sigma-alpha}
\sigma_\alpha \in B_W [\alpha].
\eeq
The generator $\sigma_\alpha$ need not be the obvious
one, given by the path $\Gamma_\alpha [c_{\alpha, 1}]$ of equation
\eqref{eqn-gamma-alpha}.
We will now use isomorphism \eqref{eqn-describe-morse-group-rank-one} to
construct a collection of Picard-Lefschetz classes:
\beq\label{eqn-classes-u-j-alpha}
\{ u_{j, \alpha} \}_{j \in J} \subset M_{l_{0, \alpha}} (P_{\chi^{}_\alpha})
\;\;\;\; \text{for} \;\;\;\; J = \{ 0, \dots n_\alpha \}.
\eeq
The collection $\{ u_{j, \alpha} \}$ will depend on the braid generator
$\sigma_\alpha$, but only through the image $r (\sigma_\alpha) \in 
\widetilde W$.

To begin, recall the data $[c_0, \xi_0, \gamma_0, o_0, a_0]$ used to define
the class $u_0 \in M_{l_0} (P_\chi)$ of equation \eqref{eqn-u-zero}.  Note that
we have $\gamma'_0 (0) > 0$ and $T_{c_0} X_{\breve c_0, \alpha} =
\Lg_\alpha \cdot c_0 \subset V_\alpha$.  Define a real subspace:
\beq\label{eqn-T-plus-alpha-0}
T_{+, \alpha} [c_0, \gamma_0] \subset \Lg_\alpha \cdot c_0 \, ,
\eeq
by analogy with \eqref{eqn-T-plus}, and note that it is just the positive eigenspace
of the partial Hessian $\cH_\alpha [c_0, l_0]$ of equation \eqref{eqn-H-alpha}.
Pick an orientation $o_{0, \alpha}$ of $T_{+, \alpha} [c_0, \gamma_0]$.
We define:
\beq\label{eqn-u-zero-alpha}
u_{0, \alpha} = \PL [c_0, \gamma_0, o_{0, \alpha}, a_0] \in
M_{l_{0, \alpha}} (P_{\chi^{}_\alpha}),
\eeq
by analogy with equation \eqref{eqn-u-zero}.

Next, for $j \in J - \{ 0 \}$, define:
\beq\label{eqn-c-j}
c_j = s_\alpha^j \, c_0 \in Z_\alpha \, .
\eeq
Note that $c_{n_\alpha} = c_0$.  For each $j \in J - \{ 0 \}$, let $\gamma_j :
[0, 1] \to \bC$ be a smooth path with:
\begin{flalign*}
& \;\;\;\;\;\;\;\; \text{(Q1)} \;\;\;\;
\gamma_j (0) = l_{0, \alpha} (c_j); & \\
& \;\;\;\;\;\;\;\; \text{(Q2)} \;\;\;\;
\gamma_j (1) = \xi_0; & \\
& \;\;\;\;\;\;\;\; \text{(Q3)} \;\;\;\;
\gamma'_j (0) = \gamma_j (0) - l_{0, \alpha} (c_\alpha) =
l_{0, \alpha} (c_j - c_\alpha); & \\
& \;\;\;\;\;\;\;\; \text{(Q4)} \;\;\;\;
\on{Re} (\gamma'_j (t) / (\gamma_j (t) - l_{0, \alpha} (c_\alpha))) > 0
\text{ for all } t \in (0, 1]; & \\
& \;\;\;\;\;\;\;\; \text{(Q5)} \;\;\;\;
\on{Im} (\gamma'_j (t) / (\gamma_j (t) - l_{0, \alpha} (c_\alpha))) < 0
\text{ for all } t \in (0, 1], & \\
& \;\;\;\;\;\;\;\;\;\;\;\;\;\;\;\;\;\;\;\;
\text {i.e., $\gamma_j (t)$ moves clockwise around the center
$l_{0, \alpha} (c_\alpha) \in \bC$}; & \\
& \;\;\;\;\;\;\;\; \text{(Q6)} \;\;\;\;
\int_0^1 \on{Im} (\gamma'_j (t) /
(\gamma_j (t) - l_{0, \alpha} (c_\alpha))) \, d t = - 2 \pi \, j / n_\alpha \, . &
\end{flalign*}

\begin{figure}
\begin{tikzpicture}
\filldraw[black] (0, 0) circle (2pt) node[anchor=north] {$l_{0, \alpha} (c_\alpha)$};
\filldraw[black] (3, 0) circle (2pt) node[anchor=south] {$l_{0, \alpha} (c_0)$};
\filldraw[black] (3*0.3090, 3*0.9510) circle (2pt) node[anchor=north]
{$l_{0, \alpha} (c_1)$};
\filldraw[black] (-3*0.8090, 3*0.5877) circle (2pt) node[anchor=north]
{$l_{0, \alpha} (c_2)$};
\filldraw[black] (-3*0.8090, -3*0.5877) circle (2pt) node[anchor=south]
{$l_{0, \alpha} (c_3)$};
\filldraw[black] (3*0.3090, -3*0.9510) circle (2pt) node[anchor=south]
{$l_{0, \alpha} (c_4)$};
\filldraw[black] (10, 0) circle (2pt) node[anchor=north] {$\xi_0$};

\draw (3, 0) to (10, 0);
\node at (5, -0.25) {$\gamma_0$};

\draw (3*0.3090, 3*0.9510) to [out = 72,
in = 140] (4, 1.5) to [out = -40,
in = 180] (10, 0);
\node at (3.2, 1.8) {$\gamma_1$};

\draw (-3*0.8090, 3*0.5877) to [out = 144,
in = -156] (-4*0.3090, 4*0.9510) to [out = 24,
in = -192] (4.2*0.3090, 4.2*0.9510) to [out = -12,
in = 140] (1.1*4, 1.1*1.5) to [out = -40,
in = 180] (10, 0);
\node at (-1, 3.5) {$\gamma_2$};

\draw (-3*0.8090, -3*0.5877) to [out = -144,
in = -84] (-4, 0) to [out = 96,
in = -120] (-4.2*0.8090, 4.2*0.5877) to [out = 60,
in = -192] (4.6*0.3090, 4.6*0.9510) to [out = -12,
in = 140] (1.2*4, 1.2*1.5) to [out = -40,
in = 180] (10, 0);
\node at (-3.6, 0) {$\gamma_3$};

\draw (3*0.3090, -3*0.9510) to [out = -72,
in = -12] (-4*0.3090, -4*0.9510) to [out = 168,
in = -48] (-4.2*0.8090, -4.2*0.5877) to [out = 132,
in = -120] (-4.6*0.8090, 4.6*0.5877) to [out = 60,
in = -192] (5.0*0.3090, 5.0*0.9510) to [out = -12,
in = 140] (1.3*4, 1.3*1.5) to [out = -40,
in = 180] (10, 0);
\node at (-1.4, -3.4) {$\gamma_4$};

\draw (3, 0) to [out = 0,
in = 60] (4*0.8090, -4*0.5877) to [out = -120,
in = 24] (4.2*0.3090, -4.2*0.9510) to [out = -156,
in = -48] (-4.6*0.8090, -4.6*0.5877) to [out = 132,
in = -120] (-5*0.8090, 5*0.5877) to [out = 60,
in = -192] (5.4*0.3090, 5.4*0.9510) to [out = -12,
in = 140] (1.4*4, 1.4*1.5) to [out = -40,
in = 180] (10, 0);
\node at (3, -2.1) {$\gamma_5$};

\draw (3, 0) to [out = 0,
in = 80] (3.5, -0.5) to [out = -100,
in = 10] (3, -1) to [out = 190,
in = -90] (1.8, 0) to [out = 90,
in = 180] (3, 1.2) to [out = 0,
in = 162] (5, 0.5) to [out = -18,
in = 180] (10, 0);
\node at (1.9, 1) {$\widetilde\gamma_5$};
\end{tikzpicture}
\vspace{.1in}
\begin{center}
Figure 1: The paths $\gamma_0, \dots, \gamma_5, \widetilde\gamma_5$ for
$n_\alpha = 5$.
\end{center}
\end{figure}

Figure 1 illustrates conditions (Q1)-(Q6) in the case $n_\alpha = 5$.  The path
$\widetilde\gamma_{n_\alpha}$ (shown here as $\widetilde\gamma_5$)
will be introduced in the proof of Proposition \ref{prop-R-chi-alpha-at-zero}.
Note that conditions (Q1)-(Q6) ensure that:
\beq\label{eqn-Q1-Q6-P1-P5}
\begin{gathered}
\text{the path $\gamma_j$ satisfies conditions (P1)-(P5) of
Section \ref{sec-fourier} for the representation} \\
\text{$G_\alpha | V_\alpha$,
the covector $l_{0, \alpha} \in (V_\alpha^*)^{rs}$,
and the critical point $c_j \in Z_\alpha$.}
\end{gathered}
\eeq
Moreover, conditions (Q1)-(Q6) determine the path $\gamma_j$ uniquely up
to homotopy within the class of all smooth paths satisfying conditions (P1)-(P5).

For each $j \in J$, we consider the tangent space:
\beqn
T_{c_j} X_{\breve c_0, \alpha} = \Lg_\alpha \cdot c_j =
\Lg_\alpha \cdot \Cartan,
\eeqn
the partial Hessian:
\beqn
\cH_\alpha [c_j, l_0] \in Sym^2 ((\Lg_\alpha \cdot \Cartan)^*),
\eeqn
(see equation \eqref{eqn-H-alpha}), and the positive eigenspace:
\beq\label{eqn-T-plus-alpha}
T_{+, \alpha} [c_j] = T_+ [c_j, \gamma_j] \subset \Lg_\alpha \cdot \Cartan,
\eeq
of $\gamma'_j (0)^{-1} \cdot \cH_\alpha [c_j, l_0]$ (cf. equation
\eqref{eqn-T-plus}).  Note that, for $j = 0$, the notation of
\eqref{eqn-T-plus-alpha} is consistent with the notation of
\eqref{eqn-T-plus-alpha-0}.

\begin{lemma}\label{lemma-carousel-hessians}
$\;$
\begin{enumerate}[topsep=-1.5ex]
\item[(i)]    For every $j \in J -\{ 0 \}$, we have:
\beqn
\cH_\alpha [c_j, l_0] = \exp (-2 \pi \bi \, j / n_\alpha) \cdot \cH_\alpha [c_0, l_0]
\;\;\;\; \text{and} \;\;\;\;
T_{+, \alpha} [c_j] = \exp (2 \pi \bi \, j / n_\alpha) \cdot T_{+, \alpha} [c_0].
\eeqn

\item[(ii)]   For every $j \in J$ and every $g \in N_K (\Cartan)$, representing
the element $s_\alpha^j \in W$, we have:
\beqn
g \cdot T_{+, \alpha} [c_0] = T_{+, \alpha} [c_j].
\eeqn
\end{enumerate}
\end{lemma}

\begin{proof}
Let $\Cartan_\alpha^\p \subset \Cartan$ be the orthogonal complement
to $\Cartan_\alpha$ with respect to $\langle \; , \, \rangle$; it is the
eigenspace of $s_\alpha$ with eigenvalue $2 \pi \bi / n_\alpha$.  Let:
\beqn
\barValpha = \Cartan_\alpha^\p \oplus \Lg_\alpha \cdot \Cartan \subset
V_\alpha \, ,
\eeqn
and let:
\beqn
\bar f_\alpha : \barValpha \to \barValpha \inv G_\alpha = 
\Cartan_\alpha^\p / W_\alpha \cong \bC,
\eeqn
be the quotient map.  Part (i) follows from the fact that the map $\bar f_\alpha$
is given by a homogenous polynomial of degree $n_\alpha$.  Part (ii) follows
from the fact that $g \, l = \exp (-2 \pi \bi \, j / n_\alpha) \cdot l$, while the space
$X_{\breve c_0, \alpha}$ and the inner product $\langle \; , \, \rangle$ are
preserved by the action of $g$.
\end{proof}

For every $j \in J - \{ 0 \}$, let:
\beq\label{eqn-o-j-alpha}
o_{j, \alpha} = (\exp (2 \pi \bi \, j / n_\alpha))_* \,o_{0, \alpha} \, ,
\eeq
be the orientation of $T_{+, \alpha} [c_j]$ obtained as the push-forward of
$o_{0, \alpha}$ via the scalar multiplication by $\exp (2 \pi \bi \, j / n_\alpha)
\in \bC$; see Lemma \ref{lemma-carousel-hessians} (i).  Note that we have
$o_{n_\alpha, \alpha} = o_{0, \alpha}$.

Note that the $G_\alpha$-equivariant structure on $\cL_{\chi^{}_\alpha}$
gives rise to a $\widetilde W_\alpha$-equivariant structure on the restriction
of $\cL_{\chi^{}_\alpha}$ to the critical set $Z_\alpha = W_\alpha \cdot c_0$.
Thus, for every $\widetilde w \in \widetilde W_\alpha$ and $c_j \in Z_\alpha$,
we obtain an action map $(\cL_{\chi^{}_\alpha})_{c_j} \to
(\cL_{\chi^{}_\alpha})_{\widetilde w \, c_j}$, which we denote
by $a \mapsto \widetilde w \cdot a$.

In view of diagram \eqref{eqn-main-diagram-alpha} and equations
\eqref{eqn-p-alpha-sigma}-\eqref{eqn-r-sigma-alpha}, for each $j \in J - \{ 0 \}$,
we have:
\beq\label{eqn-c-zero-c-j}
r (\sigma_\alpha^{-j}) \, c_0 = p_\alpha (\sigma^{-j}) \, c_0 =
s_\alpha^j \, c_0 = c_j \, .
\eeq
Using \eqref{eqn-c-zero-c-j}, for each $j \in J - \{ 0 \}$, we define:
\beq\label{eqn-define-a-j}
a_j = r (\sigma_\alpha^{-j}) \cdot a_0 \in (\cL_{\chi^{}_\alpha})_{c_j} \, ,
\eeq
\beq\label{eqn-u-j-alpha}
u_{j, \alpha} = \PL [c_j, \gamma_j, o_{j, \alpha}, a_j] \in M_{l_{0, \alpha}}
(P_{\chi^{}_\alpha}),
\eeq
where the notation of \eqref{eqn-PL} is applied to the representation
$G_\alpha | V_\alpha$.  This completes the construction of the classes
\eqref{eqn-classes-u-j-alpha}.

\begin{remark}\label{rmk-basis-dependence}{\em
As noted following equation \eqref{eqn-classes-u-j-alpha}, the collection
$\{ u_{j, \alpha} \}_{j \in J}$ depends on the choice of the braid generator
$\sigma_\alpha \in B_W [\alpha]$ in \eqref{eqn-fix-sigma-alpha}.
However, one can use Lemma \ref{lemma-conjugate} and equation
\eqref{eqn-restrict-characters} to show that the class $u_{j, \alpha}$ is
independent of $\sigma_\alpha$ for every $j \in J$ which is divisible by
$\degalpha$.}
\end{remark}

By Lemma \ref{lemma-carousel-hessians}, for every $j \in J$ and every
$g \in N_K (\Cartan)$, representing $s_\alpha^j \in W_\alpha$, we have
an $\bR$-linear self-map:
\beq\label{eqn-self-map}
\exp (-2 \pi \bi \, j / n_\alpha) \circ g : T_{+, \alpha} [c_0] \to
T_{+, \alpha} [c_0].
\eeq
The map \eqref{eqn-self-map} preserves the inner product
$\langle \; , \, \rangle$, and therefore has determinant $\pm 1$.
We write $Sgn_\alpha (g) \in \{ \pm 1 \}$ for this determinant, or
equivalently, the effect of the map \eqref{eqn-self-map} on the orientation
of $T_{+, \alpha} [c_0]$.  Using Proposition \ref{prop-core} (iii), it is not
hard to check that the assignment $g \mapsto Sgn_\alpha (g)$ descends
to a character:
\beq\label{eqn-sgn-alpha}
Sgn_\alpha : \widetilde W_\alpha^{\rm f} \to \{ \pm 1\},
\eeq
of the subgroup $\widetilde W_\alpha^{\rm f} = q^{-1} (W_\alpha) \subset
\widetilde W$ of equations
\eqref{eqn-tilde-B-W-alpha-f}-\eqref{eqn-tilde-W-alpha-f}
and diagram \eqref{eqn-main-diagram-alpha-full}.

The character $Sgn_\alpha$ of equation \eqref{eqn-sgn-alpha} can be described
more concretely.  Recall that the group $\Gaf = Z_G (\Cartan_\alpha)$ acts on the
subspace $V_\alpha \subset V$ (see assertion \eqref{eqn-Gaf-V-alpha}).  Write:
\beqn
\on{det}_\alpha : \Gaf \to \bG_m \, ,
\eeqn
for the determinant of this action (cf. \eqref{eqn-det-alpha-notalpha}).
By Proposition \ref{prop-char-tau}, applied to the representation
$G_\alpha | V_\alpha$, the character $\on{det}_\alpha$ descends
to a character:
\beqn
\hat \tau_\alpha : \widetilde W_\alpha^{\rm f} \to \bG_m \, ,
\eeqn
satisfying $\hat \tau_\alpha |_{I_\alpha} = \tau_\alpha$.  Let $\zeta_\alpha :
W_\alpha \to \bG_m$ be the character defined by:
\beqn
\zeta_\alpha (s_\alpha) = \exp (2 \pi \bi / n_\alpha), 
\eeqn
and let:
\beqn
\tilde \zeta_\alpha = \zeta_\alpha \circ q_\alpha^{\rm f}:
\widetilde W_\alpha^{\rm f} \to \bG_m \, ,
\eeqn
(see diagram \eqref{eqn-main-diagram-alpha-full}).  By reviewing the definition
of the  character $Sgn_\alpha$, and observing that $T_{+, \alpha} [c_0] \subset
\Lg_\alpha \cdot \Cartan$ is a real form, one can see that:
\beq\label{eqn-sgn-alpha-hat-tau}
Sgn_\alpha = \hat \tau_\alpha \cdot \tilde \zeta_\alpha^{-d_\alpha - 1}
: \widetilde W_\alpha^{\rm f} \to \bG_m \, ,
\eeq
where $d_\alpha = \dim X_{\breve c_0, \alpha} = \dim \Lg_\alpha \cdot \Cartan$.
Note, however, that the description \eqref{eqn-sgn-alpha-hat-tau} does
not make it clear that $Sgn_\alpha$ takes values in $\{ \pm 1\}$.
By restricting \eqref{eqn-sgn-alpha-hat-tau} to
$I_\alpha \subset \widetilde W_\alpha^{\rm f}$, we obtain:
\beq\label{eqn-sgn-alpha-tau}
Sgn_\alpha |_{I_\alpha} = \tau_\alpha \, .
\eeq

\begin{remark}\label{rmk-hat-tau-alpha}{\em
By an argument similar to the proof of Proposition \ref{prop-char-tau},
we have $\hat \tau_\alpha (x) \in \{ \pm 1 \}$ for every $x \in I$.  By Theorem
\ref{thm-root-space-decomp} (iii), we further have:
\beqn
\tau (x) = \prod_{\alpha \in A} \hat \tau_\alpha (x)
\;\; \text{for every} \;\; x \in I.
\eeqn}
\end{remark}

By Lemma \ref{lemma-conjugate}, we have:
\beq\label{eqn-sgn-alpha-invariant}
\forall \; \sigma_1, \sigma_2 \in B_W [\alpha] \; : \;
Sgn_\alpha (r (\sigma_1)) = Sgn_\alpha (r (\sigma_2)). 
\eeq
The following lemma encapsulates the carousel technique and forms the
basis of all the main arguments in this section.  Let $J^0 = J - \{ n_\alpha \}$.

\begin{lemma}\label{lemma-carousel}
$\;$
\begin{enumerate}[topsep=-1.5ex]
\item[(i)]    The elements $\{ u_{j, \alpha} \}_{j \in J^0}$ form a basis of
$M_{l_{0, \alpha}} (P_{\chi^{}_\alpha})$.

\item[(ii)]   We have:
\beqn
\mu_{l_{0, \alpha}} (\sigma^\degalpha) \, u_{0, \alpha} =
u_{\degalpha, \alpha} \, .
\eeqn

\item[(iii)]  For every $j \in J^0$, we have:
\beqn
\lambda_{l_{0, \alpha}} \circ \; \tilde r [\sigma_\alpha] \; (\sigma^{-1}) \;
u_{j, \alpha} = Sgn_\alpha (r (\sigma_\alpha)) \cdot u_{j+1, \alpha} \, .
\eeqn
\end{enumerate}
\end{lemma}

\begin{proof}
Part (i) follows form the fact that the set $\{ u_{j, \alpha} \}_{j \in J^0}$
contains exactly one Picard-Lefschez class for every critical point in
$Z_\alpha$, and the paths $\{ \gamma_j \}_{j \in J_0}$ are mutually disjoint;
see Figure 1.

Part (ii) is a standard carousel argument (cf. proof of \cite[Lemma 7.7 (ii)]{BG}).
For each $\pTau \in [0, 1]$, we can emulate equation \eqref{eqn-gamma-alpha}
to define:
\beq\label{eqn-c-tau}
c [\pTau] = c_\alpha + \exp (2 \pi \bi \, \pTau \, \degalpha / n_\alpha) 
(c_0 - c_\alpha) \in \Cartan^{reg},
\eeq
(see assumption \eqref{eqn-assume-c-zero}).
Note that we have $c [0] = c_0$, $c [1] = c_\degalpha$, and the path:
\beqn
\pTau \mapsto f_\alpha (c [\pTau]) \in Q_\alpha^{reg},
\eeqn
represents the element $\sigma^\degalpha \in B_{W_\alpha}$.
Let $X_{\pTau, \alpha} = f_\alpha^{-1} (f_\alpha (c [\pTau]))$, and let
$Z_{\pTau, \alpha} \subset X_{\pTau, \alpha}$ be the critical locus of
the restriction $l_{0, \alpha} |_{X_{\pTau, \alpha}}$.  We then have
$Z_{\pTau, \alpha} = W_\alpha \cdot c [\pTau]$.  We can extend the path
$\gamma_0 : [0, 1] \to \bC$ to a smoothly varying family of smooth paths:
\beqn
\{ \gamma [\pTau] : [0, 1] \to \bC\}_{\pTau \in [0, 1]} \, ,
\eeqn
such that $\gamma [0] = \gamma_0$, $\gamma [1] = \gamma_\degalpha$,
and each $\gamma [\pTau]$ satisfies conditions (Q1)-(Q5), where we replace
$c_j$ by $c [\pTau]$ and $\gamma_j$ by $\gamma [\pTau]$.  It follows that:
\beqn
\mu_{l_{0, \alpha}} (\sigma^\degalpha) \, u_0 =
\PL [c_\degalpha, \gamma_\degalpha, o, a],
\eeqn
for a certain orientation $o$ of $T_{+, \alpha} [c_\degalpha]$ and a certain
$a \in (\cL_{\chi^{}_\alpha})_{c_\degalpha}$, which we now proceed to
identify.

To identify the orientation $o$, we consider the family of positive eigenspaces:
\beq\label{eqn-T-plus-alpha-tau}
T_{+, \alpha} [\pTau] = T_{+, \alpha} [c [\pTau], \gamma [\pTau]] \subset
\Lg_\alpha \cdot c [\pTau] = \Lg_\alpha \cdot \Cartan, \;\; \pTau \in [0, 1], 
\eeq
defined by analogy with the subspaces $T_{+, \alpha}  [c_j]$ of equation
\eqref{eqn-T-plus-alpha}.  The orientation $o$ is then the parallel translate
of the orientation $o_{0, \alpha}$ of $T_{+, \alpha} [0] = T_{+, \alpha}  [c_0]$
(see equation \eqref{eqn-u-zero-alpha}) to an orientation of $T_{+, \alpha} [1] =
T_{+, \alpha}  [c_\degalpha]$, in the family \eqref{eqn-T-plus-alpha-tau}.  By
analogy with Lemma \ref{lemma-carousel-hessians} (i), we have:
\beqn
T_{+, \alpha} [\pTau] = \exp (2 \pi \bi \, \pTau \, \degalpha / n_\alpha) \cdot
T_{+, \alpha} [0],
\eeqn
and therefore, we have $o = o_{\degalpha, \alpha}$ (see equation
\eqref{eqn-o-j-alpha}).

To identify the element $a \in (\cL_{\chi^{}_\alpha})_{c_\degalpha}$, we
refer to diagram \eqref{diagram-V-tilde-alpha} from the proof of Proposition
\ref{prop-mu-alpha}.  As in that proof, we can use the family $f_{\alpha, \chi}$
to define the monodromy action $\mu_\alpha$.  We write $\hat c_{0, \alpha} \in
\Cartan / W_{\alpha, \chi}$ for the image of $c_0$.  By analogy with
\eqref{eqn-hat-no-hat}, we identify the fibers:
\beq\label{eqn-hat-no-hat-alpha-1}
(f_{\alpha, \chi})^{-1} (\hat c_{0, \alpha}) \cong X_{\breve c_0, \alpha},
\eeq
and we use \eqref{eqn-hat-no-hat-alpha-1} to regard $X_{\breve c_0, \alpha}$
as a subset of $V_{\alpha, \chi}^{rs}$. The path:
\beqn
\pTau \mapsto c [\pTau], \;\; [0, 1] \to \Cartan^{reg} \subset V_\alpha^{rs},
\eeqn
determines a path:
\beq\label{eqn-tilde-c-tau}
\pTau \mapsto \tilde c [\pTau], \;\; [0, 1] \to \widetilde V_{\alpha, \chi}^{rs} \, ,
\eeq
in the obvious way.  Note that $\tilde c [0] = c_0$ and
$\tilde c [1] = c_\degalpha$.
Recall that the choice of the braid generator $\sigma_\alpha \in B_W [\alpha]$
in \eqref{eqn-fix-sigma-alpha} defines an extension $\hat \chi_\alpha :
\widetilde B_{W_\alpha}^{\chi_\alpha} \to \bG_m$ of the character
$\chi_\alpha$, as in \eqref{eqn-hat-chi-alpha}.  Recall also that
the restriction $\hat \chi_\alpha |_{\widetilde B_{W_\alpha}^\chi}$
is independent of the choice of $\sigma_\alpha$, as noted following
\eqref{eqn-restrict-characters}.  The character $\hat \chi_\alpha
|_{\widetilde B_{W_\alpha}^\chi}$ gives rise to a rank one $G$-equivariant
local system $\hat \cL_{\alpha, \chi}$ on $\widetilde V_{\alpha, \chi}^{rs}$, with:
\beq\label{eqn-hat-no-hat-alpha-2}
\hat \cL_{\alpha, \chi} |_{X_{\breve c_0, \alpha}} \cong \cL_{\chi^{}_\alpha} \, ,
\eeq
as in \eqref{eqn-hat-no-hat}.  The element
$a \in (\cL_{\chi^{}_\alpha})_{c_\degalpha} \cong
(\hat \cL_{\alpha, \chi})_{c_\degalpha}$ is obtained by parallel transporting
$a_0 \in (\cL_{\chi^{}_\alpha})_{c_0} \cong (\hat \cL_{\alpha, \chi})_{c_0}$
along the path \eqref{eqn-tilde-c-tau}, where the isomorphisms of stalks are
given by \eqref{eqn-hat-no-hat-alpha-2}.

To compute the element $a \in (\cL_{\chi^{}_\alpha})_{c_\degalpha}$,
consider the element:
\beqn
\tilde b \coloneqq \tilde r [\sigma_\alpha] \, (\sigma^\degalpha) \in
\widetilde B_{W_\alpha}^\chi \subset \widetilde B_{W_\alpha} \, ,
\eeqn
and let $h_{\tilde b} : (\cL_{\chi^{}_\alpha})_{c_0} \to 
(\cL_{\chi^{}_\alpha})_{c_0}$ be the holonomy of the local system
$\hat \cL_{\alpha, \chi}$ along $\tilde b$.  By condition \eqref{eqn-hat-chi-alpha},
we have $h_{\tilde b} (a_0) = a_0$.  On the other hand, if we represent $\tilde b$
by a path $\Gamma : [0, 1] \to V_\alpha^{rs}$, as in equation
\eqref{eqn-concrete-loop}, we can compute:
\beqn
h_{\tilde b} (a_0) = r (\sigma_\alpha^\degalpha) \cdot a.
\eeqn
It follows that $a = a_\degalpha$, as defined in \eqref{eqn-define-a-j}.
This completes the proof of part (ii).

For part (iii), we apply the construction of equation \eqref{eqn-concrete-loop}
to the element:
\beqn
\tilde r [\sigma_\alpha] \, (\sigma^{-1}) \in \widetilde B_{W_\alpha} \, ,
\eeqn
(cf. proofs of \cite[Proposition 6.12 (i) $\&$ Lemma 7.5]{GVX1}).
Namely, recall the compact form $K_\alpha \subset G_\alpha$ of equation
\eqref{eqn-K-alpha}, and use Proposition \ref{prop-core} (iii) to pick a
representative $g_1 \in N_{K_\alpha} (\Cartan)$ of $r (\sigma_\alpha^{-1})
\in \widetilde W_\alpha$.  Choose a path $g : [0, 1] \to K_\alpha$, with
$g (0) = 1$ and $g (1) = g_1$.  By analogy with equation
\eqref{eqn-gamma-alpha} and using assumption \eqref{eqn-assume-c-zero}
(see also equation \eqref{eqn-c-tau}), define a path $\Gamma_1 : [0, 1] \to
\Cartan^{reg}$ by:
\beq\label{eqn-Gamma-one}
\Gamma_1 (t) = c_\alpha + \exp (-2 \pi \bi \, t / n_\alpha) (c_0 - c_\alpha).
\eeq
Note that $\Gamma_1 (0) = c_0$ and $\Gamma_1 (1) = c_{n_\alpha - 1}$.
Next, define a path $\Gamma_2 : [0, 1] \to X_{\breve c_0, \alpha}$ by:
\beqn
\Gamma_2 (t) = g (t) \, c_{n_\alpha - 1} \, .
\eeqn
Note that $\Gamma_2 (0) = c_{n_\alpha - 1}$ and $\Gamma_2 (1) = c_0$.
As in equation \eqref{eqn-concrete-loop}, define a path $\Gamma : [0, 1]
\to V_\alpha^{rs}$ as a composition:
\beqn
\Gamma = \Gamma_2 \star \Gamma_1 \, .
\eeqn
The path $\Gamma$ represents the element $\tilde r [\sigma_\alpha] \,
(\sigma^{-1}) \in \widetilde B_{W_\alpha}$.

Recall the identification of fundamental groups \eqref{eqn-identify-rank-one},
which is analogous to \eqref{eqn-identify-pi-one}.  Write:
\beqn
V_\alpha^{mrs} = K_\alpha \cdot (\Cartan - \Cartan_\alpha) \subset
V_\alpha^{rs}
\;\;\;\; \text{and} \;\;\;\;
(V_\alpha^*)^{mrs} = K_\alpha \cdot (\Cartan^* - \Cartan^*_\alpha) \subset
(V_\alpha^*)^{rs} \, ,
\eeqn
for the minimal semisimple loci (cf. equation \eqref{eqn-core}).  Also, write:
\beqn
\nu_\alpha : V_\alpha \to V_\alpha^* \, , \;\;\;\;
\nu_\alpha^{mrs} : V_\alpha^{mrs} \to (V_\alpha^*)^{mrs} \, ,
\eeqn
for the analogs of the maps $\nu$, $\nu^{mrs}$ of equations
\eqref{eqn-nu}, \eqref{eqn-nu-mrs}.  Note that we have $\Gamma ([0, 1])
\subset V^{mrs}_\alpha$.  Therefore, we can use holonomy along the path:
\beqn
\Gamma^* = \nu_\alpha^{mrs} \circ \Gamma : [0, 1] \to (V_\alpha^*)^{mrs} \, ,
\eeqn
to define the operator $\lambda_{l_{0, \alpha}} \circ \tilde r [\sigma_\alpha] \,
(\sigma^{-1}) \in \on{Aut} (M_{l_{0, \alpha}} (P_{\chi^{}_\alpha}))$.  Let:
\beq\label{eqn-factor-Gamma-star}
\Gamma_1^* = \nu_\alpha^{mrs} \circ \Gamma_1 \, , \;\;
\Gamma_2^* = \nu_\alpha^{mrs} \circ \Gamma_2 \, ,
\;\; \text{so that} \;\;
\Gamma^* = \Gamma_2^* \star \Gamma_1^* \, .
\eeq
Note that we have $l_{0, \alpha} = \nu_\alpha^{mrs} (c_0)$, and define
$l_{\alpha, \alpha} = \nu_\alpha (c_\alpha)$, $l_{1, \alpha} = \nu_\alpha^{mrs} 
(c_{n_\alpha - 1})$.  By equation \eqref{eqn-c-alpha} and assumption
\eqref{eqn-assume-c-zero}, we have:
\beq\label{eqn-l-alpha-alpha}
l_{\alpha, \alpha} |_{X_{\breve c_0, \alpha}} \equiv
\langle c_\alpha, c_\alpha \rangle.
\eeq
By equation \eqref{eqn-Gamma-one} for $t = 1$ and the antilinear property
of the map $\nu_\alpha$, we have:
\beq\label{eqn-l-one-alpha}
l_{1, \alpha} = l_{\alpha, \alpha} + \exp (2 \pi \bi \, t / n_\alpha)
(l_{0, \alpha} - l_{\alpha, \alpha}).
\eeq
Equation \eqref{eqn-factor-Gamma-star} enables us to write:
\beq\label{eqn-factor-lambda}
\lambda_{l_{0, \alpha}} \circ \tilde r [\sigma_\alpha] \, (\sigma^{-1}) =
\lambda_2 \circ \lambda_1 \, ,
\eeq
where $\lambda_1 : M_{l_{0, \alpha}} (P_{\chi^{}_\alpha}) \to M_{l_{1, \alpha}}
(P_{\chi^{}_\alpha})$ and $\lambda_2 : M_{l_{1, \alpha}} (P_{\chi^{}_\alpha})
\to M_{l_{0, \alpha}} (P_{\chi^{}_\alpha})$ are the holonomy operators for the
local system $M (P_{\chi^{}_\alpha})$, corresponding to the paths
$\Gamma_1^*$ and $\Gamma_2^*$.

Consider the case $j = 0$ of the claim of part (iii).  By an argument similar
to the proof of part (ii), one can verify that:
\beqn
\lambda_1 (u_{0, \alpha}) = \PL [c_0, \gamma_1, o_{0, \alpha}, a_0] \in
M_{l_{1, \alpha}} (P_{\chi^{}_\alpha}),
\eeqn
(cf. equation \eqref{eqn-u-zero-alpha}).
To see that the Picard-Lefschetz parameters $[c_0, \gamma_1, o_{0, \alpha},
a_0]$ specify an element of $M_{l_{1, \alpha}} (P_{\chi^{}_\alpha})$, note that:
\beqn
l_{1, \alpha} (c_0) = l_{0, \alpha} (c_1) = \gamma_1 (0).
\eeqn
Further, let $l_1 = \nu (c_{n_\alpha - 1}) \in V^*$, so that $l_{1, \alpha} =
l_1 |_{V_\alpha}$.  By equations
\eqref{eqn-l-alpha-alpha}-\eqref{eqn-l-one-alpha}, we have:
\beqn
\cH_\alpha [c_0, l_1] = \exp (2 \pi \bi \, t / n_\alpha) \cdot
\cH_\alpha [c_0, l_0].
\eeqn
It follows that the positive eigenspace of $\gamma'_1 (0)^{-1} \cdot
\cH_\alpha [c_0, l_1]$ is equal to the subspace
$T_{+, \alpha} [c_0] = T_{+, \alpha} [c_0, \gamma_0] \subset 
\Lg_\alpha \cdot \Cartan$ of equations \eqref{eqn-T-plus-alpha-0}
and  \eqref{eqn-T-plus-alpha}, as required.

For the second step of equation \eqref{eqn-factor-lambda}, we observe that:
\beqn
\lambda_2 (\PL [c_0, \gamma_1, o_{0, \alpha}, a_0]) =
g_1 \, \PL [c_0, \gamma_1, o_{0, \alpha}, a_0] =
\PL [c_1, \gamma_1, (g_1)_* \, o_{0, \alpha}, a_1],
\eeqn
(cf. equation \eqref{eqn-define-a-j}).
It remains to note that, by equation \eqref{eqn-o-j-alpha} and the definition
of the character $Sgn_\alpha$ of equation \eqref{eqn-sgn-alpha}, we have:
\beqn
(g_1)_* \, o_{0, \alpha} = Sgn_\alpha (r (\sigma_\alpha^{-1})) \cdot
o_{1, \alpha} \, .
\eeqn
Since $Sgn_\alpha (r (\sigma_\alpha^{-1})) = Sgn_\alpha (r (\sigma_\alpha))$,
we obtain:
\beqn
\lambda (l_{0, \alpha}) \circ \tilde r [\sigma_\alpha] \, (\sigma^{-1}) \, u_0 =
Sgn_\alpha (r (\sigma_\alpha)) \cdot u_1 \, ,
\eeqn
as required.

The case $j > 0$ is completely analogous to the case $j = 0$.  We omit the 
rest of the details.
\end{proof}

\begin{proof}[Proof of Proposition \ref{prop-monic}]
We use the braid generator $\sigma_\alpha \in B_W [\alpha]$, which was
fixed in \eqref{eqn-fix-sigma-alpha} and utilized in the construction of the
Picard-Lefschetz classes \eqref{eqn-classes-u-j-alpha}, to define the
minimal polynomial $R_{\chi, \alpha}$.  Parts (i) and (iii) of Lemma
\ref{lemma-carousel} imply that:
\beq
\label{eqn-cyclic-rank-one}
u_{0, \alpha}
\;\; \text{is a cyclic vector for the action} \;\;
\lambda_{l_{0, \alpha}} \circ \; \tilde r [\sigma_\alpha] : B_{W_\alpha} \to
\on{Aut} (M_{l_{0, \alpha}} (P_{\chi^{}_\alpha})).
\eeq
The proposition follows from Lemma \ref{lemma-carousel} (i) and assertion
\eqref{eqn-cyclic-rank-one}.
\end{proof}

\begin{remark}\label{rmk-cyclic-rank-one}{\em
Note that assertion \eqref{eqn-cyclic-rank-one} implies the claim of Proposition
\ref{prop-cyclic-vector} for the representation $G_\alpha | V_\alpha$.}
\end{remark}

\begin{proof}[Proof of Proposition \ref{prop-carousel}]
Note that the group $B_{W_\alpha} \cong \bZ$ is abelian, and the action:
\beqn
\lambda_{l_{0, \alpha}} \circ \; \tilde r [\sigma_\alpha] :
B_{W_\alpha} \to \on{Aut} (M_{l_{0, \alpha}} (P_{\chi^{}_\alpha})),
\eeqn
commutes with the action $\mu_{l_{0, \alpha}}$ of equation
\eqref{eqn-mu-l-zero-alpha}, as in \eqref{eqn-commute}.  Therefore,
in view of assertion \eqref{eqn-cyclic-rank-one}, it suffices to check that:
\beq\label{eqn-apply-to-cyclic}
\mu_{l_{0, \alpha}} (\sigma^{\degalpha}) \, u_{0, \alpha} =
k_\alpha \cdot \lambda_{l_{0, \alpha}} \circ \, \tilde r [\sigma_\alpha] \,
(\sigma^{-\degalpha}) \, u_{0, \alpha} \, ,
\eeq
for some $k_\alpha \in \{ \pm 1 \}$, which is independent of $\sigma_\alpha
\in B_W [\alpha]$.  By Lemma \ref{lemma-carousel} (ii)-(iii), equation 
\eqref{eqn-apply-to-cyclic} holds with:
\beq\label{eqn-k-sigma-alpha}
k_\alpha = Sgn_\alpha (r (\sigma_\alpha))^\degalpha.
\eeq
The sign $k_\alpha$ is independent of $\sigma_\alpha$ by assertion
\eqref{eqn-sgn-alpha-invariant}.
\end{proof}

\begin{remark}\label{rmk-k-alpha-properties}{\em
Equations \eqref{eqn-k-sigma-alpha} and \eqref{eqn-sgn-alpha-hat-tau}
imply the following explicit formula for $k_\alpha$:
\beq\label{eqn-k-alpha-explicit}
k_\alpha = \hat \tau_\alpha (r (\sigma_\alpha^\degalpha)) \cdot
\exp (- 2 \pi \bi \cdot (d_\alpha+1) \cdot \degalpha / n_\alpha),
\;\;\;\;
\sigma_\alpha \in B_W [\alpha].
\eeq
Note that, for $\alpha \in A_\chi^1$ (where $\degalpha = n_\alpha$),
equation \eqref{eqn-k-alpha-explicit} simplifies as follows:
\beq\label{eqn-k-alpha-explicit-1}
k_\alpha = \tau_\alpha (r (\sigma_\alpha^{n_\alpha})),
\;\;\;\;
\sigma_\alpha \in B_W [\alpha],
\;\;
\alpha \in A_\chi^1 \, .
\eeq
Equation \eqref{eqn-k-alpha-explicit} plus assertion \eqref{eqn-conjugate-two}
imply that the assignment $\alpha \mapsto k_\alpha$ is invariant under the action
of $W_\chi$ on $A$.  Also  by \eqref{eqn-k-alpha-explicit}, the sign $k_\alpha$
depends on the character $\chi \in \hat I$ only through the integer $\degalpha$.
Finally, we note that, in some examples, such as Example \ref{ex-tau-not-one}
with $\chi = 1$, the sign $k_\alpha$ will depend on the choice of the regular
splitting $\tilde r$ in \eqref{eqn-tilde-r}.}
\end{remark}

For the proof of Propositions \ref{prop-R-alpha-vanishing} and
\ref{prop-R-mu-degree}, recall the space $\cR$ of all possible minimal
polynomials introduced in \eqref{eqn-space-R}.  This space is equipped
with a natural involution:
\beqn
\InvTheta : \cR \to \cR, \,\; \text{defined by} \;\;
\InvTheta R (z) = z^{\deg R} \cdot R (z^{-1}) / R (0).
\eeqn
The significance of the involution $\InvTheta$ is that, if $\cA$ is an associative
$\bC$-algebra with unit, $a \in \cA$ is an invertible element, and $R \in \cR$ is
the minimal polynomial of $a$, then $\InvTheta R \in \cR$ is the minimal
polynomial of $a^{-1} \in \cA$.  Note that we have:
\beq\label{eqn-theta-preserves-degree}
\deg (\InvTheta R) = \deg (R), \; R \in \cR.
\eeq

\begin{proof}[Proof of Propositions \ref{prop-R-alpha-vanishing} and
\ref{prop-R-mu-degree}]
We use the braid generator $\sigma_\alpha \in B_W [\alpha]$, which was
fixed in \eqref{eqn-fix-sigma-alpha}, to define the minimal polynomials
$R_{\chi, \alpha}$ and $\bar R_{\chi, \alpha}^\mu$.  Define
$\bar R_{\chi, \alpha} \in \cR$ as the minimal polynomial of the
holonomy operator $\lambda_{l_0, \alpha} \circ \, \tilde r [\sigma_\alpha] \,
(\sigma^\degalpha)$ of equation \eqref{eqn-interpret-bar-R-chi-alpha}.
For Proposition \ref{prop-R-alpha-vanishing}, it suffices to establish that:
\beqn
\deg \bar R_{\chi, \alpha} = n_\alpha / \degalpha \, .
\eeqn
By Proposition \ref{prop-carousel} and equation
\eqref{eqn-theta-preserves-degree}, we have:
\beq\label{eqn-apply-theta}
\bar R_{\chi, \alpha}^\mu \, (z) =  k_\alpha^{\deg \bar R_{\chi, \alpha}} \cdot
(\InvTheta \bar R_{\chi, \alpha}) \, (k_\alpha \cdot z),
\eeq
where the overall sign ensures that the RHS is a monic polynomial.
Combining this with equation \eqref{eqn-theta-preserves-degree}, we obtain:
\beqn
\deg \bar R_{\chi, \alpha}^\mu = \deg \bar R_{\chi, \alpha} \, .
\eeqn
By Proposition \ref{prop-monic}, we have $\deg R_{\chi, \alpha} = n_\alpha$,
and therefore:
\beqn
\deg \bar R_{\chi, \alpha} \geq n_\alpha / \degalpha \, .
\eeqn
By Proposition \ref{prop-degree-bound}, we have:
\beqn
\deg \bar R_{\chi, \alpha}^\mu \leq n_\alpha / \degalpha \, .
\eeqn
Propositions \ref{prop-R-alpha-vanishing} and \ref{prop-R-mu-degree}
follow.
\end{proof}

\begin{remark}\label{rmk-apply-theta}{\em
Having established Proposition \ref{prop-R-alpha-vanishing}, and using
equation \eqref{eqn-k-alpha-explicit}, we can rewrite the relationship
\eqref{eqn-apply-theta} between minimal polynomials as follows:
\beqn
\bar R_{\chi, \alpha}^\mu \, (z) = 
\tau_\alpha (r (\sigma_\alpha^{n_\alpha})) \cdot
(\InvTheta \bar R_{\chi, \alpha}) \, (k_\alpha \cdot z),
\;\;\;\;
\sigma_\alpha \in B_W [\alpha].
\eeqn
Another way to express this relationship is in terms of roots:
\beqn
\text{if $\{ z_i \}_{i = 1}^{n_\alpha / \degalpha}$ are the roots of
$\bar R_{\chi, \alpha}$, then
$\{ k_\alpha / z_i \}_{i = 1}^{n_\alpha / \degalpha}$ are the roots of
$\bar R_{\chi, \alpha}^\mu$.}
\eeqn
Since the group $I$ is finite, and all our characters take values in
finite subgroups of $\bG_m$, general results on the quasi-unipotence
of the monodromy in the family apply, and we can conclude that all
roots of the polynomials $\bar R_{\chi, \alpha}^\mu$,
$\bar R_{\chi, \alpha}$, and  $R_{\chi, \alpha}$ are roots of unity.}
\end{remark}

\begin{proof}[Proof of Proposition \ref{prop-R-chi-alpha-at-zero}]
Let $\tilde\gamma_{n_\alpha} : [0, 1] \to \bC$ be a smooth path, as shown in
Figure 1 for $n_\alpha = 5$.  More precisely, the path $\tilde\gamma_{n_\alpha}$
satisfies conditions (P1)-(P5) of Section \ref{sec-fourier} for the representation
$G_\alpha | V_\alpha$, the covector $l_{0, \alpha} \in (V_\alpha^*)^{rs}$,
and the critical point $c_{n_\alpha} = c_0 \in Z_\alpha$, as well as the following
conditions:
\begin{flalign*}
& \;\;\;\;\;\;\;\; \bullet \;\;\;\;
\tilde\gamma'_{n_\alpha} (0) = \tilde\gamma'_0 (0); & \\
& \;\;\;\;\;\;\;\; \bullet \;\;\;\;
\text{$\xi_{l_0} (\tilde\gamma_{n_\alpha} (t)) > \xi_{l_0} (c_j)$
for all $t \in [0, 1]$ and $j \in \{ 1, \dots, n_\alpha - 1 \}$;} & \\
& \;\;\;\;\;\;\;\; \bullet \;\;\;\;
\text{the path $\tilde\gamma_{n_\alpha}$ makes one clockwise
turn around $l_{0, \alpha} (c_0)$ before reaching $\xi_0$.} & 
\end{flalign*}

Write:
\beqn
z = \lambda_{l_{0, \alpha}} \circ \; \tilde r [\sigma_\alpha] \;
(\sigma^{-1}) \in \on{End} (M_{l_{0, \alpha}} (P_{\chi^{}_\alpha})),
\eeqn
and let $R_z \in \cR$ be the minimal polynomial of $z$.
By Lemma \ref{lemma-carousel} (iii) and equation \eqref{eqn-sgn-alpha-tau},
we have:
\beq\label{eqn-z-n-alpha}
z^{n_\alpha} (u_{0, \alpha}) = Sgn_\alpha (r (\sigma_\alpha^{n_\alpha}))
\cdot u_{n_\alpha, \alpha} = \tau_\alpha (x) \cdot u_{n_\alpha, \alpha} \, .
\eeq
By the definition of the classes $\{ u_{j, \alpha} \}$ for $j \neq 0$ in equation
\eqref{eqn-u-j-alpha}, plus a standard Picard-Lefschetz  theory argument,
we have:
\beq\label{eqn-change-path}
u_{n_\alpha, \alpha} = \PL [c_0, \gamma_{n_\alpha}, o_{0, \alpha}, a_{n_\alpha}] 
= \PL [c_0, \tilde\gamma_{n_\alpha}, o_{0, \alpha}, a_{n_\alpha}] +
L_1(u_{1, \alpha}, \dots, u_{j-1, \alpha}),
\eeq
where $L_1 (u_{1, \alpha}, \dots, u_{j-1, \alpha})$ is some linear combination of
the $u_{1, \alpha}, \dots, u_{j-1, \alpha}$.  By the definition of $a_{n_\alpha} \in
(\cL_{\chi^{}_\alpha})_{c_0}$ in \eqref{eqn-define-a-j}, plus Lemma
\ref{lemma-carousel} (iii), we can rewrite equation \eqref{eqn-change-path}
as follows:
\beq\label{eqn-use-z}
u_{n_\alpha, \alpha} = \chi (x^{-1}) \cdot
\PL [c_0, \tilde\gamma_{n_\alpha}, o_{0, \alpha}, a_0] +
L_2 (z (u_{0, \alpha}), \dots, z^{j-1} (u_{0, \alpha})),
\eeq
where $L_2 (z (u_{0, \alpha}), \dots, z^{j-1} (u_{0, \alpha}))$ is again a
linear combination.

By considering the obvious homotopy between the paths
$\tilde \gamma_{n_\alpha}$ and $\gamma_0$, within the class
of all smooth paths satisfying (P1)-(P5) (see observation
\eqref{eqn-Q1-Q6-P1-P5}), we obtain:
\beq\label{eqn-obvious-homotopy}
\PL [c_0, \tilde\gamma_{n_\alpha}, o_{0, \alpha}, a_0] =
(-1)^{d_\alpha} \cdot \PL [c_0, \gamma_0, o_{0, \alpha}, a_0] =
(-1)^{d_\alpha} \cdot u_{0, \alpha} \, ,
\eeq
(see \eqref{eqn-u-zero-alpha}).  By combining equations
\eqref{eqn-z-n-alpha}, \eqref{eqn-use-z}, \eqref{eqn-obvious-homotopy},
we obtain:
\beqn
z^{n_\alpha} (u_{0, \alpha}) =
(-1)^{d_\alpha} \cdot \chi (x^{-1}) \cdot \tau_\alpha (x) \cdot u_{0, \alpha} +
\tau_\alpha (x) \cdot L_2 (z (u_{0, \alpha}), \dots, z^{j-1} (u_{0, \alpha})).
\eeqn
In view of Proposition \ref{prop-monic} and assertion
\eqref{eqn-cyclic-rank-one}, we can conclude that:
\beqn
R_z (0) = (-1)^{d_\alpha + 1} \cdot \chi (x^{-1}) \cdot \tau_\alpha (x).
\eeqn
It remains to note that $R_{\chi, \alpha} = \InvTheta R_z$, so
$R_{\chi, \alpha} (0) = R_z (0)^{-1}$.
\end{proof}

\section{Microlocal monodromy for a braid generator}
\label{sec-microlocal-generator}

The following proposition is a partial analog of Proposition \ref{prop-carousel}
in the context of the full representation $G | V$.

\begin{prop}\label{prop-microlocal-generator}
For each $\alpha \in A$ and $\sigma_\alpha \in B_{W} [\alpha]$, we have:
\begin{align}
\label{eqn-microlocal-generator-one}
\mu_{l_0} (\sigma_\alpha^{\degalpha}) \, u_0 & = k_\alpha \cdot
\lambda_{l_0} \circ \, \tilde r \, (\sigma_\alpha^{-\degalpha}) \, u_0
\in M_{l_0} (P_\chi), \\
\label{eqn-microlocal-generator-two}
\mu_{l_0} (\sigma_\alpha^{-\degalpha}) \, u_0 & = k_\alpha \cdot
\lambda_{l_0} \circ \, \tilde r \, (\sigma_\alpha^{\degalpha}) \, u_0
\in M_{l_0} (P_\chi),
\end{align}
where $k_\alpha \in \{ \pm 1 \}$ is the sign provided by Proposition
\ref{prop-carousel}.
\end{prop}

Our proof of Proposition \ref{prop-microlocal-generator} will be an adaptation
of the proof of Proposition \ref{prop-carousel} in Section \ref{sec-carousel}.
We begin with a preliminary discussion geared towards the analysis of the
Hessians at the critical points of $l_0 |_{X_{\bar c_0}}$.

Fix an $\alpha \in A$.  Recall the notation of equation \eqref{eqn-not-alpha},
and define:
\beqn
\Cartan_\notalpha^* \; = \; \bigcup_{\beta \in A_\notalpha}
\Cartan_\beta^* \, , \;\;\;\;
(\Cartan_\alpha^*)^{reg} \; = \; \Cartan_\alpha^* - \Cartan_\notalpha^* \, .
\eeqn
By Theorem \ref{thm-root-space-decomp} (iv), we have:
\beq\label{eqn-t-notalpha-cartan}
\Lg \cdot v \supset \Lg \cdot \Cartan_\alpha =
\bigoplus_{\beta \in A_\notalpha} \Lg_\beta \cdot \Cartan
\;\; \text{for every} \;\; v \in \Cartan - \Cartan_\notalpha \, .
\eeq
Consider the image:
\beqn
\nu (V_\alpha) = \Cartan^* \oplus \Lg_\alpha \cdot \Cartan^* =
(\Lg \cdot \Cartan_\alpha)^\perp \subset V^*,
\eeqn
(see \eqref{eqn-V-alpha} and \eqref{eqn-nu}).  Here, the first equality follows
from \eqref{eqn-orthogonal} and \eqref{eqn-Lg-alpha-Lk-alpha}, and the
second from Proposition \ref{prop-decompose-hessian}.  Note that restriction
to $V_\alpha$ produces an isomorphism: $\nu (V_\alpha) \cong V_\alpha^*$.
Let $V_\alpha^{ms} = V_\alpha \cap V^{ms}$ (see Section
\ref{subsec-dual-rep}), and consider the set of pairs:
\beqn
P_\alpha = \{ (v, l) \in V_\alpha^{ms} \times \nu (V_\alpha^{ms})
\; | \; l |_{\Lg \cdot v} = 0 \} \subset V \times V^*.
\eeqn
Note that $P_\alpha \subset V \times V^*$ is a real algebraic subset
(see \eqref{eqn-kempf-ness}), and we have:
\beqn
P_\alpha = K_\alpha \cdot (\Cartan \times \Cartan^*),
\eeqn
(cf. \eqref{eqn-core}).  Let:
\beqn
P_\alpha^\circ [\Cartan] = (\Cartan - \Cartan_\notalpha)
\times (\Cartan^* - \Cartan_\notalpha^*)
\subset \Cartan \times \Cartan^*
\;\;\;\; \text{and} \;\;\;\;
P_\alpha^\circ = K_\alpha \cdot P_\alpha^\circ [\Cartan]
\subset P_\alpha \, .
\eeqn
Since the group $K_\alpha$ is compact, we have:
\beq\label{eqn-P-alpha-not-open}
\text{the subset $P_\alpha^\circ \subset P_\alpha$ is open
in the classical topology.}
\eeq

Using equation \eqref{eqn-t-notalpha-cartan} and the $K_\alpha$-invariance
of the subspace $\Lg \cdot \Cartan_\alpha \subset V$, we obtain:
\beqn
\Lg \cdot v \supset \Lg \cdot \Cartan_\alpha \;\; \text{for every} \;\; (v, l) \in 
P_\alpha^\circ \, .
\eeqn
Let $H_\notalpha = Sym^2 ((\Lg \cdot \Cartan_\alpha)^*)$ and let
$H_\notalpha^\circ \subset H_\notalpha$ be the open subset consisting
of all non-degenerate quadratic forms.  Note that $H_\notalpha$ inherits
a Hermitian metric from $\langle \; , \, \rangle$, and $K_\alpha$ acts on
$H_\notalpha$ by isometries, preserving the subset
$H_\notalpha^\circ \subset H_\notalpha$.  We have a map:
\beq\label{eqn-H-notalpha}
\cH_\notalpha : P_\alpha^\circ \to H_\notalpha \, , \;\;
(v, l) \mapsto \cH_\notalpha [v, l] \coloneqq 
\cH [v, l] |_{\Lg \cdot \Cartan_\alpha} \, ,
\eeq
(see \eqref{eqn-H-v-l}).

\begin{lemma}\label{lemma-H-notalpha-cont}
$\;$
\begin{enumerate}[topsep=-1.5ex]
\item[(i)]
The map $\cH_\notalpha$ of equation \eqref{eqn-H-notalpha} is
$K_\alpha$-equivariant.

\item[(ii)]
The map $\cH_\notalpha$ is continuous in the classical topology.

\item[(iii)]
We have $\cH_\notalpha [v, l] \in H_\notalpha^\circ$ for every
$(v, l) \in P_\alpha^\circ$.
\end{enumerate}
\end{lemma}

\begin{proof}
Part (i) follows from the definitions.  For part (ii), note that:
\beq\label{eqn-part-cont}
\text{the restriction $\cH_\notalpha |_{P_\alpha^\circ [\Cartan]}$
is complex algebraic, and therefore continuous.}
\eeq
The continuity of $\cH_\notalpha$ follows from \eqref{eqn-part-cont} plus
the $K_\alpha$-equivariance of part (i) and the compactness of $K_\alpha$.
Part (iii) follows from Propositions \ref{prop-decompose-hessian} and
\ref{prop-morse-atom}, plus the $K_\alpha$-equivariance of part (i).
\end{proof}

For every $(v, l) \in P_\alpha^\circ$, let:
\beq\label{eqn-T-plus-notalpha-v-l}
T_{+, \notalpha} [v, l] \subset \Lg \cdot \Cartan_\alpha \, ,
\eeq
be the positive eigenspace of $\cH_\notalpha [v, l]$, relative to the inner
product $\langle \; , \, \rangle$.  Here, as usual, by ``the positive eigenspace''
we mean the direct sum of all the eigenspaces of the real part
$\on{Re} (\cH_\notalpha [v, l])$, corresponding to positive eigenvalues.
By Lemma \ref{lemma-H-notalpha-cont}, the vector spaces
$\{ T_{+, \notalpha} [v, l] \}$ form a $K_\alpha$-equivariant real vector
bundle $T_{+, \notalpha} [P_\alpha^\circ]$ over $P_\alpha^\circ$, of rank
$d - d_\alpha = \dim_\bC (\Lg \cdot \Cartan_\alpha)$.  Moreover, each
$T_{+, \notalpha} [v, l]$ is a real form of $\Lg \cdot \Cartan_\alpha$, i.e.,
we have $\Lg \cdot \Cartan_\alpha = T_{+, \notalpha} [v, l] \otimes_\bR \bC$.  
Of particular interest to us is the orientation bundle of $T_{+, \notalpha}
[P_\alpha^\circ]$, which we denote by $\cO [P_\alpha^\circ]$.  It is a
$K_\alpha$-equivariant $\bZ / 2$-torsor over $P_\alpha^\circ$.

Pick a $c_\alpha \in \Cartan_\alpha^{reg}$ and let $l_\alpha = \nu (c_\alpha)
\in (\Cartan_\alpha^*)^{reg}$.  We then have $(c_\alpha, l_\alpha) \in
P_\alpha^\circ$.  Use \eqref{eqn-P-alpha-not-open} to pick an $\epsilon > 0$
such that:
\beq\label{eqn-P-alpha-epsilon}
P_{\alpha, \epsilon} \coloneqq \{ (v, l) \in P_\alpha \; | \;
\on{dist} (c_\alpha, v) < \epsilon, \on{dist} (l_\alpha, l) < \epsilon \}
\subset P_\alpha^\circ \, ,
\eeq
where the distances are taken with respect to $\langle \; , \, \rangle$.
Note that $P_{\alpha, \epsilon} \subset P_\alpha^\circ$ is a contractible,
$K_\alpha$-invariant open subset.  Since the group $K_\alpha$ is
connected, we have:
\beq\label{eqn-trivial-torsor}
\text{the restriction $\cO [P_{\alpha, \epsilon}]$ of $\cO [P_\alpha^\circ]$ to
$P_{\alpha, \epsilon}$ is a trivial $K_\alpha$-equivariant $\bZ / 2$-torsor.}
\eeq

\begin{proof}[Proof of Proposition \ref{prop-microlocal-generator}]
The proofs of \eqref{eqn-microlocal-generator-one} and
\eqref{eqn-microlocal-generator-two} are similar, and we will only
give the former.  Let $\sigma_\alpha = \sigma_\alpha [\Gamma]$, with
$\Gamma (0) = c_0$ and $\Gamma (1) = c_{\alpha, 1}$, as in equation
\eqref{eqn-sigma-alpha}.  As in the proof of Proposition \ref{prop-degree-bound},
we can reduce to the case where $c_0 = c_{\alpha, 1}$ and the path $\Gamma$
is trivial.  Moreover, we can assume that:
\beq\label{eqn-assume-near}
\on{dist} (c_\alpha, c_{\alpha, 1}) < \epsilon,
\eeq
where $c_\alpha \in \Cartan_\alpha^{reg}$ is given by equation
\eqref{eqn-c-alpha}, and $\epsilon > 0$ is chosen to satisfy the
containment $P_{\alpha, \epsilon} \subset P_\alpha^\circ$ of equation
\eqref{eqn-P-alpha-epsilon}.  Proceeding with these assumptions, the proof
of \eqref{eqn-microlocal-generator-one} is essentially a paraphrase of the
proof of equation \eqref{eqn-apply-to-cyclic} for $k_\alpha$ as in equation
\eqref{eqn-k-sigma-alpha} (see the proof of Proposition \ref{prop-carousel}).
The only non-trivial issue in adapting the argument from the setting of the
rank one representation $G_\alpha | V_\alpha$ to the setting of the full
representation $G | V$ is the treatment of the orientations of the positive
eigenspaces of the Hessians.  This issue is essentially handled by
invoking Proposition \ref{prop-decompose-hessian} and assertion
\eqref{eqn-trivial-torsor}.

Recall the subset $Z_\alpha = W_\alpha \cdot c_0 \subset Z_{l_0}$ and the
subspace $M_{l_0} (P_\chi) [\alpha] \subset M_{l_0} (P_\chi)$, introduced in
the proof of Proposition \ref{prop-degree-bound} (see equations
\eqref{eqn-Z-alpha}, \eqref{eqn-M-alpha-declare}).  In the present argument,
the subspace $M_{l_0} (P_\chi) [\alpha]$ will play the same role as the Morse
group $M_{l_{0, \alpha}} (P_{\chi^{}_\alpha})$ did in Section \ref{sec-carousel}.
As in that section, we let $J = \{ 0, \dots, n_\alpha \}$, and we construct a
collection of Picard-Lefschetz classes:
\beq\label{eqn-classes-u-j}
\{ u_j \}_{j \in J} \subset M_{l_0} (P_\chi) [\alpha],
\eeq
by analogy with the classes $\{ u_{j, \alpha} \}_{j \in J}$ of equation
\eqref{eqn-classes-u-j-alpha}.  The class $u_0$ has already been defined
by equation \eqref{eqn-u-zero}.

Recall the critical points $\{ c_j \}_{j \in J - \{ 0 \}} \subset Z_\alpha$ of
equation \eqref{eqn-c-j}.  For each $j \in J - \{ 0 \}$, let $\gamma_j : [0, 1]
\to \bC$ be a smooth path, satisfying conditions (Q1)-(Q6) of Section
\ref{sec-carousel}, as well as condition \eqref{eqn-M-alpha} from the proof
of Proposition \ref{prop-degree-bound}, which we paraphrase as follows:
\beq\label{eqn-M-alpha-paraphrase}
\xi_{l_0} \circ \gamma_j \, (t) > \xi_{l_0} (c) \;\; \text{for every} \;\;
t \in [0, 1] \;\; \text{and} \;\; c \in Z_{l_0} - Z_\alpha.
\eeq
Note that these conditions ensure that $\gamma_j$ satisfies conditions
(P1)-(P5) of Section \ref{sec-fourier}, and furthermore, determine $\gamma_j$
uniquely up to homotopy within the class of all paths satisfying (P1)-(P5).

Next, recall the elements $a_j \in (\cL_{\chi^{}_\alpha})_{c_i}$, $j \in J - \{ 0 \}$,
of equation \eqref{eqn-define-a-j}.  Note that we have $\cL_{\chi^{}_\alpha}
\cong \cL_\chi |_{X_{\breve c_0, \alpha}}$.  Therefore, we can view each $a_j$
as an element of $(\cL_\chi)_{c_j}$.

We now discuss orientations.  For each $j \in J$, let:
\beq\label{eqn-T-plus-c-j}
T_+ [c_j] \subset \Lg \cdot c_j = \Lg \cdot \Cartan,
\eeq
be the positive eigenspace of $\gamma'_j (0)^{-1} \cdot \cH [c_j, l_0]$.
By Theorem \ref{thm-root-space-decomp} (iv), we have:
\beq\label{eqn-decompose-tangent}
\Lg \cdot \Cartan = \Lg_\alpha \cdot \Cartan \oplus \Lg \cdot \Cartan_\alpha \, .
\eeq
By Proposition \ref{prop-decompose-hessian}, decomposition
\eqref{eqn-decompose-tangent} gives rise to a decomposition:
\beq\label{eqn-decompose-T-plus}
T_+ [c_j] = T_{+, \alpha} [c_j] \oplus T_{+, \notalpha} [c_j],
\eeq
where $T_{+, \alpha} [c_j] = T_+ [c_j] \cap \Lg_\alpha \cdot \Cartan$
is the positive eigenspace of equation \eqref{eqn-T-plus-alpha}, and
$T_{+, \notalpha} [c_j] = T_+ [c_j] \cap \Lg \cdot \Cartan_\alpha$.  

We have already chosen an orientation $o_{j, \alpha}$ of $T_{+, \alpha} [c_j]$
for every $j \in J$ (see equations \eqref{eqn-u-zero-alpha},
\eqref{eqn-o-j-alpha}).  We also have the orientation $o_0$ of $T_+ [c_0]$,
used to define the class $u_0$ of equation \eqref{eqn-u-zero}.  Let
$o_{0, \notalpha}$ be the orientation of $T_{+, \notalpha} [c_0]$ such that:
\beq\label{eqn-o-zero-notalpha}
o_0 = (o_{0, \alpha}, o_{0, \notalpha}),
\eeq
in terms for equation \eqref{eqn-decompose-T-plus} for $j = 0$.  We now
proceed to define an orientation $o_{j, \notalpha}$ of $T_{+, \notalpha} [c_j]$
for every $j \in J - \{ 0 \}$.

By assumption \eqref{eqn-assume-near}, we have $(c_j, l_0) \in
P_{\alpha, \epsilon} \subset V \times V^*$ for every $j \in J$.  Recall the
trivial orientation torsor $\cO [P_{\alpha, \epsilon}]$ on $P_{\alpha, \epsilon}$
of assertion \eqref{eqn-trivial-torsor}.  Because of the multiple
$\gamma'_j (0)^{-1}$ appearing the definition of the subspace
$T_+ [c_j]$ of equation \eqref{eqn-T-plus-c-j}, in the notation of
\eqref{eqn-T-plus-notalpha-v-l}, we have:
\beq\label{eqn-T-plus-notalpha-j}
T_{+, \notalpha} [c_j] = \exp (\pi \bi \, j / n_\alpha) \cdot
T_{+, \notalpha} [c_j, l_0], \;\; j \in J,
\eeq
(see Figure 1 and condition (Q3) of Section \ref{sec-carousel}).
In particular, we have $T_{+, \notalpha} [c_0] = T_{+, \notalpha} [c_0, l_0]$.
Thus, the orientation $o_{0, \notalpha}$ of $T_{+, \notalpha} [c_0]$
determines a global section $o_\notalpha [P_{\alpha, \epsilon}]$ of
$\cO [P_{\alpha, \epsilon}]$.  We denote the value of this section at
$(v, l) \in P_{\alpha, \epsilon}$ by $o_\notalpha [v, l]$, so that
$o_\notalpha [c_0, l_0] = o_{0, \notalpha}$.  For each $j \in J - \{ 0 \}$,
we now use \eqref{eqn-T-plus-notalpha-j} to define an orientation:
\beq\label{eqn-o-j-notalpha}
o_{j, \notalpha} = (\exp (\pi \bi \, j / n_\alpha))_* \, o_\notalpha [c_j , l_0],
\eeq
of $T_{+, \notalpha} [c_j]$.  Note that we have:
\beqn
c_{n_\alpha} = c_0, \;\;\;\;
T_{+, \notalpha} [c_{n_\alpha}] = T_{+, \notalpha} [c_0], \;\;\;\;
\text{and} \;\;\;\;
o_{n_\alpha, \notalpha} = (-1)^{d - d_\alpha} \cdot o_{0, \notalpha} \, ,
\eeqn
where $d - d_\alpha = \dim_\bC (\Lg \cdot \Cartan_\alpha)$.

Next, for each $j \in J - \{ 0 \}$, we use decomposition
\eqref{eqn-decompose-T-plus} to define an orientation:
\beqn
o_j = (o_{j, \alpha}, o_{j, \notalpha}),
\eeqn
of $T_+ [c_j]$
and a Picard-Lefshetz class:
\beqn
u_j = \PL [c_j, \gamma_j, o_j, a_j] \in M_{l_0} (P_\chi) [\alpha],
\eeqn
(cf. equation \eqref{eqn-u-j-alpha}).  This completes the construction
of the classes \eqref{eqn-classes-u-j}.  With these definitions, we can
proceed as in the proof of Lemma \ref{lemma-carousel}, to verify that:
\beqn
\mu_{l_0} (\sigma_\alpha^\degalpha) \, u_0 = u_\degalpha
\;\;\;\; \text{and} \;\;\;\;
\lambda_{l_0} \circ \, \tilde r \, (\sigma_\alpha^{-\degalpha}) \, u_0 =
k_\alpha \cdot u_\degalpha \, ,
\eeqn
where $k_\alpha \in \{ \pm 1 \}$ is given by equation \eqref{eqn-k-sigma-alpha}.
There are only two new features of the argument in this setting.  First, one
needs to check that critical points in $Z_{l_0} - Z_\alpha$ do not interfere
with the paths $\{ \gamma_j \}_{j \in J}$.  This is guaranteed by the choice of
$c_0 = c_{\alpha, 1}$ to be near the regular part $\Cartan_\alpha^{reg} \subset
\Cartan_\alpha$, and by assumption \eqref{eqn-M-alpha-paraphrase}
(cf. proof of assertion \eqref{eqn-preserve-M-alpha}).  Second, in view of
\eqref{eqn-decompose-T-plus}, one needs to match the orientations of the
positive eigenspace $T_{+, \notalpha} [c_\degalpha]$.  This matching follows
form assertion \eqref{eqn-trivial-torsor} and the choice of the orientations
$\{ o_{j, \notalpha} \}_{j \in J}$ in \eqref{eqn-o-zero-notalpha} and
\eqref{eqn-o-j-notalpha}.  We omit the rest of the details.
\end{proof}

\section{Proof of the main theorem}
\label{sec-proof}

In this section, we assemble the results of Sections
\ref{sec-fourier}-\ref{sec-microlocal-generator} to prove Theorem
\ref{thm-main}.  Recall that we have $B_W^{\chi, 0} \subset B_W^\chi$,
and define:
\beq\label{eqn-M-zero}
M_0 = \mu_{l_0} (\bC [B_W^{\chi, 0}]) \cdot u_0 \subset M_{l_0} (P_\chi),
\eeq
where $u_0$ is the class of equation \eqref{eqn-u-zero} and $\mu_{l_0}$
is the monodromy in the family of equation \eqref{eqn-mu-morse}.
In the argument that follows, we will show that the subspace
$M_0 \subset M_{l_0} (P_\chi)$ plays the role of the
$\widetilde B_W^{\chi, 0}$-representation $M_\chi^0$ of equation
\eqref{eqn-M-chi-induced}.  We begin with the following corollary of
Proposition \ref{prop-microlocal-inertia}.

\begin{corollary}\label{cor-inertia-M-zero}
For every $u \in M_0$ and every $x \in I$, we have:
\beqn
\lambda_{l_0} (x) \, u = \chi (x) \cdot \tau (x) \cdot u \, .
\eeqn
\end{corollary}

\begin{proof}
For $u = u_0$, this is a special case of Proposition \ref{prop-microlocal-inertia},
and for general $u$, this follows from assertion \eqref{eqn-commute}.
\end{proof}

\begin{lemma}\label{lemma-generate}
The subgroup $B_W^{\chi, 0} \subset B_W$ is generated by all elements
of the form $\sigma_\alpha^{\degalpha} \in B_W$ for $\alpha \in A$ and
$\sigma_\alpha \in B_W [\alpha]$.
\end{lemma}

\begin{proof}
By definition, we have $B_W^{\chi, 0} = \pi_1 (\Cartan^{reg} / W_\chi^0, c_0)$.
The space $\Cartan^{reg} / W_\chi^0$ is the complement of the divisor
$\Cartan^{sing} / W_\chi^0$ in the affine space $\Cartan / W_\chi^0$.
By equation \eqref{eqn-W-zero-alpha}, each element
$\sigma_\alpha^\degalpha \in B_W^{\chi, 0}$ is represented by a loop
which links this divisor once.  The lemma follows, for example, from
\cite[Proposition A1, p$. \; 181$]{BMR}, by induction on the number of
irreducible components of the divisor $\Cartan^{sing} / W_\chi^0$.
\end{proof}

\begin{lemma}\label{lemma-microlocal-M-zero}
We have:
\beqn
\lambda_{l_0} (\bC [\widetilde B_W^{\chi, 0}]) \cdot M_0 = 
\lambda_{l_0} (\bC [\widetilde B_W^{\chi, 0}]) \cdot u_0 =
\lambda_{l_0} \circ \, \tilde r \, (\bC [B_W^{\chi, 0}]) \cdot u_0 =
M_0 \, ,
\eeqn
where $\lambda_{l_0}$ is the microlocal monodromy of equation
\eqref{eqn-lambda-grp}.
\end{lemma}

\begin{proof}
The last equality follows from Lemma \ref{lemma-generate}, Proposition
\ref{prop-microlocal-generator}, and assertion \eqref{eqn-commute}.
The middle equality follows from Corollary \ref{cor-inertia-M-zero}.
And the first equality, once again, follows from assertion
\eqref{eqn-commute}.
\end{proof}

\begin{lemma}\label{lemma-dim-M-zero}
We have:
\beqn
\dim M_0 = | W_\chi^0 |.
\eeqn
\end{lemma}

\begin{proof}
Recall the decomposition $A = A_\chi^0 \cup A_\chi^1$ of Section
\ref{subsec-W-zero} (see equation \eqref{eqn-A-zero-one}).  By Propositions
\eqref{prop-family-min-poly} and \ref{prop-R-mu-degree}, we have:
\beq\label{eqn-scalar}
\text{for every $\alpha \in A_\chi^1$ and $\sigma_\alpha \in B_W [\alpha]$, the
operator $\mu_{l_0} (\sigma_\alpha^{n_\alpha}) \in \on{Aut} (M_{l_0} (P_\chi))$
is a scalar.}
\eeq
To capture these scalars, we define a character:
\beqn
\rho^\mu = \rho_\chi^\mu : B_W^{\chi, 0} \to \bG_m \, ,
\eeqn
by analogy with the character $\rho$ of equation \eqref{eqn-rho}.  Namely,
we require that:
\begin{align*}
\rho^\mu (\sigma_\alpha^\degalpha) & = 1 \;\;
\text{for all} \;\; \alpha \in A_\chi^0 \;\; \text{and} \;\;
\sigma_\alpha \in B_W [\alpha], \\
\bar R_{\chi, \alpha}^\mu (\rho^\mu (\sigma_\alpha^{n_\alpha})) & = 0 \;\;
\text{for all} \;\; \alpha \in A_\chi^1 \;\; \text{and} \;\;
\sigma_\alpha \in B_W [\alpha].
\end{align*}
Recall that, for $\alpha \in A_\chi^1$, we have $\degalpha = n_\alpha$ and 
$\deg \bar R_{\chi, \alpha}^\mu = 1$.  The character $\rho^\mu$ is well-defined
by Corollary \eqref{cor-invariance-mu}.  Note that the characters $\rho$ and
$\rho^{\mu}$ determine each other by Remark \ref{rmk-apply-theta}.  I.e., for
every $\alpha \in A_\chi^1$ and $\sigma_\alpha \in B_W [\alpha]$, we have:
\beqn
\rho (\sigma_\alpha^{n_\alpha}) \cdot \rho^{\mu} (\sigma_\alpha^{n_\alpha}) =
k_\alpha \, ,
\eeqn
where $k_\alpha \in \{ \pm 1 \}$ is provided by Proposition \ref{prop-carousel}.
Using equations \eqref{eqn-rho-explicit} and \eqref{eqn-k-alpha-explicit-1},
this yields an explicit expression:
\beq\label{eqn-rho-mu-explicit}
\rho^{\mu} (\sigma_\alpha^{n_\alpha}) = (-1)^{d_\alpha} \cdot
\chi (r (\sigma_\alpha^{-n_\alpha})),
\eeq
for all $\alpha \in A_\chi^1$ and $\sigma_\alpha \in B_W [\alpha]$.

Consider the diagram:
\beq\label{eqn-diagram-rho-mu}
B_{W_\chi^0}
\xleftarrow{\;\;\;\varphi\;\;\;}
B_W^{\chi, 0}
\xrightarrow{\;\;\rho^\mu\;\;} 
\bG_m \, ,
\eeq
(see equation \eqref{eqn-varphi}).  By an argument as in the proof of Lemma
\ref{lemma-generate}, the kernel $\ker (\varphi) \subset B_W^{\chi, 0}$ is
generated by all elements of the form $\sigma_\alpha^{n_\alpha}$ for
$\alpha \in A_\chi^1$ and $\sigma_\alpha \in B_W [\alpha]$.  It follows that
the restriction of the monodromy action $\mu_{l_0}$ to $\ker (\varphi) \subset
B_W^{\chi, 0}$ is given by the character $\rho^\mu$.  Therefore, we have:
\beqn
\mu_{l_0} |_{\ker (\varphi) \cap \ker (\rho^\mu)} = 1,
\eeqn
and the restriction of $\mu_{l_0}$ to $\ker (\rho^\mu) \subset B_W^{\chi, 0}$
factors through the homomorphism $\varphi$.  Since the restriction
$\varphi |_{\ker (\rho^\mu)}$ is surjective, we obtain an action:
\beqn
\breve\mu_{l_0} : B_{W_\chi^0} \to \on{Aut} (M_{l_0} (P_\chi)),
\eeqn
such that $\mu_{l_0} |_{\ker (\rho^\mu)} = \breve\mu_{l_0} \circ \varphi$.
By Lemma \ref{lemma-generate} and assertion \eqref{eqn-scalar}, we have:
\beqn
M_0 = \mu_{l_0} (\bC [B_W^{\chi, 0}]) \cdot u_0 =
\breve\mu_{l_0} (\bC [B_{W_\chi^0}]) \cdot u_0 \, .
\eeqn
The inequality:
\beqn
\dim M_0 \leq | W_\chi^0 |,
\eeqn
now follows from Proposition \ref{prop-R-mu-degree} and
\cite[Theorem 1.3]{Et}, while the inequality:
\beqn
\dim M_0 \geq | W_\chi^0 |,
\eeqn
follows from Corollary \ref{cor-dim-M}, Proposition \ref{prop-cyclic-vector},
and Lemma \ref{lemma-microlocal-M-zero}.
\end{proof}

For every $\alpha \in A$, let $\bar R_{\chi, \alpha}^0 \in \cR$ be the minimal
polynomial of the restriction:
\beqn
\lambda_{l_0} \circ \, \tilde r \, (\sigma_\alpha^{\degalpha}) |_{M_0} \in
\on{End} (M_0),
\eeqn
for some $\sigma_\alpha \in B_W [\alpha]$.  Note that $\bar R_{\chi, \alpha}^0$
is independent of the choice of $\sigma_\alpha$ by assertion
\eqref{eqn-conjugate}.

\begin{lemma}\label{lemma-microlocal-min-poly}
For every $\alpha \in A$, we have:
$\bar R_{\chi, \alpha}^0 = \bar R_{\chi, \alpha}$.
\end{lemma}

\begin{proof}
Pick a $\sigma_\alpha \in B_W [\alpha]$, and let $\bar R_{\chi, \alpha}^{\mu, 0}
\in \cR$ be the minimal polynomial of $\mu_{l_0} (\sigma_\alpha^{\degalpha})
|_{M_0}$.  By Propositions \ref{prop-fourier}, \ref{prop-cyclic-vector},
\ref{prop-family-min-poly} and assertion \eqref{eqn-commute}, we have:
\beq\label{eqn-R-mu-zero}
\bar R_{\chi, \alpha}^{\mu, 0} = \bar R_{\chi, \alpha}^\mu \, .
\eeq
By Proposition \ref{prop-microlocal-generator}, Lemma
\ref{lemma-microlocal-M-zero}, assertion \eqref{eqn-commute},
and equation \eqref{eqn-theta-preserves-degree}, we have:
\beq\label{eqn-apply-theta-two}
\bar R_{\chi, \alpha}^{\mu, 0} \, (z) = k_\alpha^{\deg \bar R_{\chi, \alpha}^0}
\cdot (\InvTheta \bar R_{\chi, \alpha}^0) \, (k_\alpha \cdot z),
\eeq
(cf. equation \eqref{eqn-apply-theta}).  Equation \eqref{eqn-apply-theta-two}
uniquely determines the polynomial $\bar R_{\chi, \alpha}^0 \in \cR$.
Therefore, the lemma follows from equations \eqref{eqn-apply-theta}
and \eqref{eqn-R-mu-zero}-\eqref{eqn-apply-theta-two}.
\end{proof}

\begin{proof}[Proof of Theorem \ref{thm-main}]
By Lemma \ref{lemma-microlocal-M-zero}, the action $\lambda_{l_0}
\circ \, \tilde r : B_W \to \on{Aut} (M_{l_0} (P_\chi))$ restricts to an action:
\beqn
\lambda_{l_0}^0 : B_W^{\chi, 0} \to \on{Aut} (M_0).
\eeqn
By analogy with diagram \eqref{eqn-diagram-rho-mu}, consider the diagram:
\beqn
B_{W_\chi^0}
\xleftarrow{\;\;\;\varphi\;\;\;}
B_W^{\chi, 0}
\xrightarrow{\;\;\;\rho\;\;\;} 
\bG_m \, .
\eeqn
Arguing as in the proof of Lemma \ref{lemma-dim-M-zero}, and using Lemma
\ref{lemma-microlocal-min-poly} for $\alpha \in A_\chi^1$, we observe that the
restriction of $\lambda_{l_0}^0$ to $\ker (\rho) \subset B_W^{\chi, 0}$ factors
through the homomorphism $\varphi$.  Since the restriction
$\varphi |_{\ker (\rho)}$ is surjective, we obtain an action:
\beqn
\breve\lambda_{l_0}^0 : B_{W_\chi^0} \to \on{Aut} (M_0),
\eeqn
such that $\lambda_{l_0}^0 |_{\ker (\rho)} = \breve\lambda_{l_0} \circ \varphi$.

Lemmas \ref{lemma-microlocal-M-zero}, \ref{lemma-dim-M-zero},
Lemma \ref{lemma-microlocal-min-poly} for $\alpha \in A_\chi^0$,
and equation \eqref{eqn-dim-H} enable us to conclude that:
\beqn
M_0 \cong \cH_{W_\chi^0} \;\; \text{as} \;\;
\bC [B_{W_\chi^0}]\text{-modules},
\eeqn
where the $\bC [B_{W_\chi^0}]$-module structure on $M_0$ is given
by the action $\breve\lambda_{l_0}$.  Combining this with Lemma
\ref{lemma-generate} and Lemma \ref{lemma-microlocal-min-poly}
for $\alpha \in A_\chi^1$, we obtain:
\beqn
M_0 \cong \bC_\rho \otimes \cH_{W_\chi^0}
\;\; \text{as} \;\; \bC [B_W^{\chi, 0}]\text{-modules},
\eeqn
where the $\bC [B_W^{\chi, 0}]$-module structure on $M_0$ is given
by the action $\lambda_{l_0}^0$.  Combining this further with Corollary
\ref{cor-inertia-M-zero}, we obtain:
\beq\label{eqn-identify-M-zero}
M_0 \cong \bC_\chi \otimes \bC_\tau \otimes \bC_\rho \otimes \cH_{W_\chi^0}
\;\; \text{as} \;\; \bC [\widetilde B_W^{\chi, 0}]\text{-modules},
\eeq
where the $\bC [\widetilde B_W^{\chi, 0}]$-module structure on $M_0$ is
given by the microlocal monodromy action $\lambda_{l_0}$.  By Proposition
\ref{prop-cyclic-vector}, plus a dimension count using Corollary
\ref{cor-dim-M} and Lemma \ref{lemma-dim-M-zero}, we obtain:
\beq\label{eqn-induce}
M_{l_0} (P_\chi) \cong \bC [\widetilde B_W]
\otimes_{\bC [\widetilde B_W^{\chi, 0}]} M_0
\;\; \text{as} \;\; \bC [\widetilde B_W]\text{-modules},
\eeq
where the $\bC [\widetilde B_W]$-module structure on $M_{l_0} (P_\chi)$ is
again given by $\lambda_{l_0}$.  It remains to note that the factor $\bC_\tau$
of equation \eqref{eqn-identify-M-zero} can be taken outside the tensor product
of equation \eqref{eqn-induce}, because the character $\tau$ of equation
\eqref{eqn-char-tau-I} is preserved by the action of $W$ on $\hat I$.
\end{proof}

\section{A vanishing conjecture for intersection numbers}
\label{sec-conj}

The proof of Theorem \ref{thm-main} in Sections \ref{sec-fourier}-\ref{sec-proof}
is an elaborate application of Picard-Lefschetz theory.  However, it avoids any
direct use of the central result of this theory: the Picard-Lefschetz formula.
In this section, we remind the reader the statement of this formula in the 
context of our paper, and formulate a related conjecture (Conjecture
\ref{conj-vanishing}) which, if true, would elucidate the structure behind
Theorem \ref{thm-main}.

In order to state Conjecture \ref{conj-vanishing} in a suitable generality,
pick an arbitrary basepoint:
\beqn
l \in (\Cartan^*)^{reg} \subset (V^*)^{rs}.
\eeqn
We do not wish to restrict to the case $l = l_0$, partly in view of Remark
\ref{rmk-not-generic}.  Recall the identification of the Morse group
$M_l (P_\chi)$ in Lemma \ref{lemma-describe-morse-group}.  Note that we
have $Z_l = \{ w \, c_0 \}_{w \in W}$, independent of $l$.  The generic property
of $l \in (V^*)^{rs}$, given by Proposition \ref{prop-A-f-stratification}, ensures
that the restriction $l |_{X_{\bar c_0}}$ is a locally trivial fibration over the
non-critical locus $\bC - l (Z_l)$.  It follows that we have:
\beq\label{eqn-M-l}
M_l (P_\chi) \cong H_d (X_{\bar c_0}, X_{\bar c_0, \xi_0}; \cL_\chi),
\eeq
where $X_{\bar c_0, \xi_0} = \{ x \in X_{\bar c_0} \; | \; l (x) = \xi_0 \}$.
Note that each Picard-Lefschetz class $\PL [c, \gamma, o, a] \in M_l
(P_\chi)$, as in equation \eqref{eqn-PL}, is constructed as an element
of the RHS of equation \eqref{eqn-M-l}.

The space $X_{\bar c_0, \xi_0}$ is a smooth manifold of real dimension
$2d-2$, which we orient by the complex orientation.  Thus, there is a
well-defined intersection pairing:
\beqn
H_{d-1} (X_{\bar c_0, \xi_0}; \cL_\chi) \otimes
H_{d-1} (X_{\bar c_0, \xi_0}; \cL_\chi^*) \to \bC,
\eeqn
where $\cL_\chi^* \cong \cL_{\chi^{-1}}$ is the dual of the local system
$\cL_\chi$.  We denote this pairing by $a \otimes b \mapsto a \pitchfork b$.

Pick a pair of critical points $c_1, c_2 \in Z_l$ and a triple of paths
$\gamma_1, \gamma_2, \gamma_3 : [0, 1] \to \bC$, such that
$\gamma_1 (0) = \gamma_3 (0) = l (c_1)$, $\gamma_2 (0) = l (c_2)$,
$\gamma'_1 (0) = \gamma'_3 (0)$, and all three paths satisfy conditions
(P2)-(P5) of Section \ref{sec-fourier}.  Assume that the paths $\gamma_1,
\gamma_2, \gamma_3$ do not meet each other, except at the end-points.
Also, assume that $l (c_2)$ is the unique element of $l (Z_l)$, contained
inside the closed curve formed by the paths $\gamma_1$ and $\gamma_3$,
and that $l (c) \neq l (c_2)$ for every $c \in Z_l - \{ c_2 \}$.  Pick orientations
$o_1$ of $T_+ [c_1, \gamma_1]$ and $o_2$ of $T_+ [c_2, \gamma_2]$
(see \eqref{eqn-T-plus}).  Also, pick generators $a_1 \in (\cL_\chi)_{c_1}$,
$a_2 \in (\cL_\chi)_{c_2}$.  Let $a_2^* \in (\cL_\chi^*)_{c_2}$ be the unique
element with $\langle a_2, a_2^* \rangle = 1$.  Define Picard-Lefschetz
classes $u_1, u_2, u_3 \in H_d (X_{\bar c_0}, X_{\bar c_0, \xi_0}; \cL_\chi)$
and $u_2^* \in H_d (X_{\bar c_0}, X_{\bar c_0, \xi_0}; \cL_\chi^*)$ as follows:
\begin{align}
\notag
u_1 = \PL [c_1, \gamma_1, o_1, a_1], \;\;\;\;
u_3 = \PL [c_1, \gamma_3, o_1, a_1], \\
\label{eqn-u-two}
u_2 = \PL [c_2, \gamma_2, o_2, a_2], \;\;\;\;
u_2^* = \PL [c_2, \gamma_2, o_2, a_2^*].
\end{align}
The Picard-Lefschetz formula, in this setting, takes the following form:
\beq\label{eqn-PL-formula}
u_1 - u_3 = \pm (\partial \, u_1 \pitchfork \partial \, u_2^*) \cdot u_2 \, ,
\eeq
where the symbol $\partial$ denotes the boundary maps:
\beqn
H_d (X_{\bar c_0}, X_{\bar c_0, \xi_0}; \cL_\chi) \xrightarrow{\;\;\;\partial\;\;\;}
H_{d-1} (X_{\bar c_0, \xi_0}; \cL_\chi) \;\; \text{and} \;\;
H_d (X_{\bar c_0}, X_{\bar c_0, \xi_0}; \cL_\chi^*) \xrightarrow{\;\;\;\partial\;\;\;}
H_{d-1} (X_{\bar c_0, \xi_0}; \cL_\chi^*).
\eeqn
Note that we specify the RHS of equation \eqref{eqn-PL-formula} only up to
sign, since we will only be interested in the question of when the intersection
number $\partial \, u_1 \pitchfork \partial \, u_2^*$ is zero.  The
Picard-Lefschetz formula in this form can be derived from the discussion
in \cite[Chapter 1.3]{AGV}.

The following is an immediate corollary of Proposition \ref{prop-decomp} and
its proof.

\begin{corollary}\label{cor-vanishing}
In the situation of equation \eqref{eqn-PL-formula}, assume that $c_1 =
w_1 \, c_0$, $c_2 = w_2 \, c_0$, $w_1, w_2 \in W$, and the cosets
$\barw_1 = w_1 \, W_\chi$ and $\barw_2 = w_2 \, W_\chi$ are distinct.
Then we have:
\beqn
\partial \, u_1 \pitchfork \partial \, u_2^* = 0.
\eeqn
\end{corollary}

\begin{proof}
First, consider the case $l = l_0$.  Using the notation of
\eqref{eqn-M-l-zero-barw}, and with reference to equation
\eqref{eqn-PL-formula}, we have:
\beqn
\on{LHS} \in M_{l_0} (P_\chi [\barw_1]) \;\; \text{while} \;\;
\on{RHS} \in M_{l_0} (P_\chi [\barw_2]).
\eeqn
Therefore, by Proposition \ref{prop-decomp}, we have
$\on{LHS} = \on{RHS} = 0$.  But $u_2 \neq 0$, as a Picard-Lefschetz class
with $a_2 \neq 0$ (see equation \eqref{eqn-u-two}).  Therefore, we must
have $\partial \, u_1 \pitchfork \partial \, u_2^* = 0$.

For general $l \in (\Cartan^*)^{reg}$, we note that the proof of Proposition
\ref{prop-decomp} goes through unchanged for the group $M_l (P_\chi)$
in place of $M_{l_0} (P_\chi)$.  Therefore, the same argument as in the case
$l = l_0$ applies.
\end{proof}

The following conjecture is a slight modification of Corollary \ref{cor-vanishing}.

\begin{conj}\label{conj-vanishing}
In the situation of equation \eqref{eqn-PL-formula}, assume that
$c_1 = w_1 \, c_0$, $c_2 = w_2 \, c_0$, $w_1, w_2 \in W$, and
the cosets $w_1 \, W_\chi^0$ and $w_2 \, W_\chi^0$ are distinct,
as elements of $W / W_\chi^0$.  Then we have:
\beqn
\partial \, u_1 \pitchfork \partial \, u_2^* = 0.
\eeqn
\end{conj}

Conjecture \ref{conj-vanishing} plus equation \eqref{eqn-PL-formula} readily
imply a version of Proposition \ref{prop-decomp} with the subgroup $W_\chi
\subset W$ replaced by $W_\chi^0$, yielding a decomposition:
\beqn
M_{l_0} (P_\chi) = \bigoplus_{\barw \in W / W_\chi^0} M_{l_0} (P_\chi) [\barw].
\eeqn
Such a decomposition, once established, would enable us to define:
\beqn
M_0 = M_{l_0} (P_\chi) [\,\barone\,],
\eeqn
where $\barone \in W / W_\chi^0$ is the coset of the identity in $W$, thus
replacing the equivalent, but rather less direct, definition \eqref{eqn-M-zero}.
With this simpler definition, it would be essentially immediate that
$\dim M_0 = |W_\chi^0|$ (cf. Lemma \ref{lemma-dim-M-zero}) and
that $M_0 \subset M_{l_0} (P_\chi)$ is preserved by the microlocal
monodromy action $\lambda_{l_0} |_{\widetilde B_W^{\chi, 0}} \,$
(cf. Lemma \ref{lemma-microlocal-M-zero}).  Moreover, for every
$\tilde b \in \widetilde B_W$, we would have:
\beqn
\lambda_{l_0} (\tilde b) \, (M_0) = M_{l_0} (P_\chi) [\barw],
\eeqn
where $w = p \circ \tilde q \, (\tilde b)$ and $\barw = w \, W_\chi^0 \in
W / W_\chi^0$.

Thus, in our view, Conjecture \ref{conj-vanishing} clarifies the nature of the
subspace $M_0 \subset M_{l_0} (P_\chi)$, and therefore, provides some
context for the statement of Theorem \ref{thm-main}.  An independent
proof of this conjecture would be likely to provide a simplified proof of
Theorem \ref{thm-main}.  At present, we are unable to prove Conjecture
\ref{conj-vanishing} outside of the cases where $W_\chi^0 = W_\chi$.

\end{document}